\documentclass[12pt]{article}
\usepackage{amsmath,amssymb,amsthm,amsfonts,graphicx,caption}
\usepackage[all]{xy}

\newtheorem{thm}{Theorem}[section]
\newtheorem{lemma}[thm]{Lemma}
\newtheorem{prop}[thm]{Proposition}
\newtheorem{cor}[thm]{Corollary}
\newtheorem{defin}[thm]{Definition}
\newtheorem{rem}[thm]{Remark}
\newtheorem{exam}[thm]{Example}

\newcommand{\R}{{\mathbb{R}}}
\newcommand{\D}{{\mathbb{D}}}
\newcommand{\Q}{{\mathbb{Q}}}
\newcommand{\T}{{\mathbb{T}}}
\newcommand{\Z}{{\mathbb{Z}}}
\newcommand{\N}{{\mathbb{N}}}
\newcommand{\C}{{\mathbb{C}}}
\newcommand{\SP}{{\mathbb{S}}}

\newcommand{\tP}{{\widetilde{\mathcal{P}}}}

\newcommand{\cA}{{\mathcal{A}}}
\newcommand{\cB}{{\mathcal{A}}}

\newcommand{\cF}{{\mathcal{F}}}

\newcommand{\cK}{{\mathcal{K}}}
\newcommand{\cL}{{\mathcal{L}}}
\newcommand{\tcL}{{\widetilde{\mathcal{L}}}}

\newcommand{\cP}{{\mathcal{P}}}

\newcommand{\cS}{{\mathcal{S}}}

\newcommand{\Fix}{{\hbox{\rm Fix\,}}}

\newcommand{\Maslov}{{\hbox{\it Maslov\,}}}

\newcommand{\Ham}{{\it Ham}}

\newcommand{\id}{{\text{{\bf 1}}}}
\newcommand{\tHam}{\widetilde{\hbox{\it Ham}\, }}
\newcommand{\Symp}{{\hbox{\it Symp} }}

\newcommand{\Qed}{\hfill \qedsymbol \medskip}

\begin{document}

\title{Rigid subsets of symplectic manifolds\\
%{\it (preliminary version)}
}

\renewcommand{\thefootnote}{\alph{footnote}}

\author{\textsc Michael Entov$^{a}$\ and Leonid
Polterovich$^{b}$ }

\footnotetext[1]{Partially supported by E. and J. Bishop Research
Fund and by the Israel Science Foundation grant $\#$ 881/06.}
\footnotetext[2]{Partially supported by the Israel Science
Foundation grant $\#$ 11/03.}

\date{\today}

\maketitle

\bigskip

\begin{abstract}
\noindent

We show that there is an hierarchy of intersection rigidity
properties of sets in a closed symplectic manifold: some sets
cannot be displaced by symplectomorphisms from more sets than the
others. We also find new examples of rigidity of intersections
involving, in particular, specific fibers of moment maps of
Hamiltonian torus actions, monotone Lagrangian submanifolds
(following the works of P.Albers and P.Biran-O.Cornea), as well as
certain, possibly singular, sets defined in terms of
Poisson-commutative subalgebras of smooth functions.  In addition,
we get some geometric obstructions to semi-simpli\-ci\-ty of the
quantum homology of symplectic manifolds. The proofs are based on
the Floer-theoretical machinery of partial symplectic
quasi-states.
\end{abstract}

\vfil \eject

\tableofcontents

\renewcommand{\thefootnote}{\arabic{footnote}}
%\addtocounter{footnote}{-1}
\vfil \eject

\section{Introduction and main results}
\label{sec-intro}

\subsection{Many facets of displaceability}
\label{subsec-intro-many-facets-of-displaceability}

A well-studied and easy to visualize rigidity property of subsets of
a symplectic manifold $(M,\omega)$ is the rigidity of intersections:
a subset $X\subset M$ cannot be displaced from the closure of a
subset $Y\subset M$ by a compactly supported Hamiltonian isotopy:
\[
\phi (X)\cap \overline{Y}\neq \emptyset \ \ \forall \phi\in \Ham
(M)\;.
\]
We say in such a case that $X$ {\it cannot be displaced} from $Y$.
If $X$ cannot be displaced from itself we call it {\it
non-displaceable}. These properties become especially interesting
and purely symplectic when $X$ can be displaced from itself or from
$Y$ by a (compactly supported) smooth isotopy.

One of the main themes of the present paper is that {\it  ``some
non-dis\-pla\-ce\-able sets are more rigid than others."} To
explain this, we need the following ramifications of the notion of
a non-displaceable set:

\medskip
\noindent {\sc Strong non-displaceability:} A subset $X\subset M$
is called {\it strongly non-displaceable} if one cannot displace
it by any (not necessarily Hamiltonian) symplectomorphism of
$(M,\omega)$.

\medskip
\noindent {\sc Stable non-displaceability:}  Consider $T^* \SP^1 =
\R\times \SP^1$ with the coordinates $(r, \theta)$ and the
symplectic form $dr \wedge d\theta$. We say that $X\subset M$ is
{\it stably non-displaceable} if $X \times \{ r = 0\}$ is
non-displaceable in $M \times T^* \SP^1$ equipped with the split
symplectic form $\bar{\omega}=\omega \oplus (dr \wedge d\theta)$.
Let us mention that detecting stably non-displaceable subsets is
useful for studying geometry and dynamics of Hamiltonian flows
(see for instance \cite{P-IMRN} for their role in Hofer's geometry
and \cite{P-Rudnick} for their appearance in the context of kick
stability in Hamiltonian dynamics).

\medskip
\noindent  Formally speaking, the properties of strong and stable
non-displaceability are mutually independent  and both are
strictly stronger than displaceability.

\medskip
\noindent In the present paper we refine the machinery of partial
symplectic quasi-states introduced in \cite{EP-qst} and get new
examples of stably non-displaceable sets, including certain fibers
of moment maps of Hamiltonian torus actions as well as monotone
Lagrangian submanifolds discussed by Albers \cite{Albers} and
Biran-Cornea \cite{BC}. Further, we address the following
question: given the class of stably non-displaceable sets, can one
distinguish those of them which are also strongly non-displaceable
by means of the Floer theory? Or, other way around, what are the
Floer-homological features of stably non-displaceable but strongly
displaceable sets?  Toy examples are given by the equator of the
symplectic two-sphere and by the meridian on a symplectic
two-torus. Both are stably non-displaceable since their Lagrangian
Floer homologies are non-trivial. On the other hand, the equator
is strongly non-displaceable, while the meridian is strongly
displaceable by a non-Hamiltonian shift. Later on we shall explain
the difference between these two examples from the viewpoint of
Hamiltonian Floer homology and present various generalizations.

\medskip
\noindent The question on Floer-homological characterization of
(strongly) non-dis\-pla\-ce\-able but stably displaceable sets is
totally open, see Section~\ref{subsubsec-pack} below for an
example involving Gromov's packing theorem and discussion.

\medskip
\noindent Leaving Floer-theoretical considerations for the next
section, let us outline (in parts, informally) the general scheme
of our results: Given a symplectic manifold $(M,\omega)$, we shall
define (in the language of the Floer theory) two collections of
closed subsets of $M$, {\it heavy subsets} and {\it superheavy
subsets}. Every superheavy subset is heavy, but, in general, not
vice versa. Formally speaking, the hierarchy heavy-superheavy
depends in a delicate way on the choice of an idempotent in the
quantum homology ring of $M$. This and other nuances will be
ignored in this outline. The key properties of these collections
are as follows (see Theorems~\ref{thm-intro-heavy} and
\ref{thm-products} below):

\medskip
\noindent {\bf \underline{Invariance}:} Both collections are
invariant under the group of all symplectomorphisms of $M$.

\medskip
\noindent{\bf \underline{Stable non-displaceability}:} Every heavy
subset is stably non-displace\-able.

\medskip
\noindent {\bf \underline{Intersections}:} Every superheavy subset
intersects every heavy subset. In particular, superheavy subsets are
strongly non-displaceable. In contrast to this, heavy subsets can be
mutually disjoint and strongly displaceable.

\medskip
\noindent {\bf \underline{Products}:} Product of any two
(super)heavy subsets is (super)heavy.

\medskip
\noindent{\bf \underline{What is inside the collections?}} The
collections of heavy and superheavy sets include the following
examples:

\medskip
\noindent {\sc Stable stems}:   Let ${\mathbb A} \subset
C^{\infty}(M)$ be a finite-dimensional Poisson-com\-mu\-ta\-tive
subspace (i.e. any two functions from ${\mathbb A}$ commute with
respect to the Poisson brackets). Let $\Phi: M \to {\mathbb A}^*$
be the moment map: $\langle\Phi(x),F\rangle = F(x)$. A non-empty
fiber $\Phi^{-1}(p)$, $p \in {\mathbb A}^*$, is called a {\it
stem} of ${\mathbb A}$ (see \cite{EP-qst}) if all non-empty fibers
$\Phi^{-1}(q)$ with $q \neq p$ are displaceable and a {\it stable
stem} if they are stably displaceable. If a subset of $M$ is a
(stable) stem of a finite-dimensional Poisson-commutative subspace
of $C^\infty (M)$, it will be called just {\it a (stable) stem}.
Clearly, any stem is a stable stem. {\bf The collection of
superheavy subsets includes all stable stems} (see
Theorem~\ref{thm-stem} below). One readily shows that a direct
product of stable stems is a stable stem and that the image of a
stable stem under {\it any} symplectomorphism is again a stable
stem.

The following example of a  stable stem is borrowed (with a minor
modification) from \cite{EP-qst}: Let $X \subset M$ be a closed
subset whose complement is a finite disjoint union of stably
displaceable sets. Then $X$ is a stable stem. For instance, the
codimension-1 skeleton of a sufficiently fine triangulation of any
closed symplectic manifold is a stable stem. Another example is
given by the equator of $\SP^2$: it divides the sphere into two
displaceable open discs and hence is a stable stem. By taking
products, one can get more sophisticated examples of stable stems.
Already the product of equators of the two-spheres gives rise to a
Lagrangian Clifford torus in $\SP^2 \times \ldots \times \SP^2$.
To prove its rigidity properties (such as stable
non-displaceability) one has to use non-trivial symplectic tools
such as Lagrangian Floer homology, see e.g. \cite{Oh-involution}.
Products of the 1-skeletons of fine triangulations of the
two-spheres can be considered as {\it singular Lagrangian
submanifolds}, an object which is currently out of reach of the
Lagrangian Floer theory.

Another example of stable stems comes from Hamiltonian torus
actions. Consider an effective Hamiltonian action $\varphi:
\T^k\to \Ham  (M)$ with the moment map $\Phi = (\Phi_1,\ldots,
\Phi_k): M\to\R^k$. Assume that $\Phi_i$ is a normalized
Hamiltonian, that is $\int_M \Phi_i=0 $ for all $i=1,\ldots, k$. A
torus action is called {\it compressible} if the image of the
homomorphism $\varphi_{\sharp}: \pi_1 (\T^k)\to \pi_1 (\Ham(M))$,
induced by the action $\varphi$, is a finite group. One can show
that for compressible actions the fiber $\Phi^{-1}(0)$ is a stable
stem (see Theorem~\ref{thm-stab-disp} below).

\medskip
\noindent {\sc Special fibers of Hamiltonian torus actions}:
Consider  an effective  Hamiltonian torus action $\varphi$ on a
spherically monotone symplectic manifold. Let $I:\pi_1(\Ham(M))\to
\R$ be the mixed action-Maslov homomorphism introduced in
\cite{Pol-mixed}. Since the target space $\R^k$ of the moment map
$\Phi$ is naturally identified with $\text{Hom}(\pi_1(\T^k),\R)$,
the pull back $p_{spec}:=-\varphi_{\sharp}^*I$ of the mixed
action-Maslov homomorphism with the reversed sign can be
considered as a point of $\R^k$.  The preimage
$\Phi^{-1}({p}_{spec})$ is called {\it the special fiber} of the
action. We shall see below that the special fiber is always
non-empty. For monotone symplectic toric manifolds (that is when
$2k =\dim M$) the special fiber is a monotone Lagrangian torus.
Note that when the action is compressible we have $p_{spec}=0$ and
therefore the special fiber is a stable stem according to the
previous example. It is unknown whether the latter property
persists for general non-compressible actions. Thus in what
follows we treat stable stems and special fibers as separate
examples. {\bf The collection of superheavy subsets includes all
special fibers} (see Theorem~\ref{thm-main-nondispl-toric} below).

For instance, consider $\C P^2$ and the Lagrangian Clifford torus
in it (i.e. the torus $\{ [z_0 : z_1 : z_2]\in \C P^2\ |\ |z_0| =
|z_1| = |z_2| \}$). Take the standard Hamiltonian $\T^2$-action on
$\C P^2$ preserving the Clifford torus. It has three global fixed
points away from the Clifford torus. Make an equivariant
symplectic blow-up, $M$,  of $\C P^2$ at $k$ of these fixed
points, $0\leq k\leq 3$, so that the obtained symplectic manifold
is spherically monotone. The torus action lifts to a Hamiltonian
action on $M$.  One can show that its special fiber is the proper
transform of the Clifford torus.

\medskip
\noindent {\sc Monotone Lagrangian submanifolds:} Let
$(M^{2n},\omega)$ be a spherically monotone symplectic manifold,
and let $L \subset M$ be a closed monotone Lagrangian submanifold
with the minimal Maslov number  $N_L \geq 2$. We say that $L$ {\it
satisfies the Albers condition} \cite{Albers} if the image of the
natural morphism $H_* (L; \Z_2) \to H_* (M; \Z_2)$ contains a
non-zero element $S$ with $$\deg S
> \dim L +1 -N_L\;.$$

\noindent {\bf The collection of heavy sets includes all closed
monotone Lagran\-gi\-an submanifolds satisfying the Albers
condition} (see Theorem~\ref{thm-Albers-elem-implies-heavy}
below).

Specific examples include the meridian on $\T^2$, $\R P^n \subset
\C P^n$ and all Lagrangian spheres in complex projective
hypersurfaces of degree $d$ in $\C P^{n+1}$ with $n> 2d-3$. In the
case when the fundamental class $[L]$ of $L$ divides a non-trivial
idempotent in the quantum homology algebra of $M$, $L$ is, in
fact, superheavy (see Theorem~\ref{thm-monot-superh} below). For
instance, this is the case for $\R P^n \subset \C P^n$.
Furthermore, a version of superheaviness holds for any Lagrangian
sphere in the complex quadric of even (complex) dimension.

However, there exist examples of heavy, but not superheavy,
Lagrangian submanifolds: For instance, the meridian of the 2-torus
is strongly displaceable by a (non-Hamiltonian!) shift and hence
is not superheavy. Another example of heavy but not superheavy
Lagrangian submanifold is the sphere arising as the real part of
the Fermat hypersurface $$M = \{-z_0^d+z_1^d+\ldots+z_{n+1}^d =
0\} \subset \C P^{n+1}\;$$ with even $d \geq 4$ and $n > 2d-3$. We
refer to Section~\ref{subsec-mon-lagr} for more details on
(super)heavy monotone Lagrangian submanifolds.

\medskip
\noindent {\bf \underline{Motivation}:} Our motivation for the
selection of examples appearing in the list above is as follows.
Stable stems provide a playground for studying symplectic rigidity
of singular subsets. In particular, no visible analogue of the
conventional Lagrangian Floer homology technique is applicable to
them.

Detecting (stable) non-displaceability of Lagrangian submanifolds
via Lagrangian Floer homology is one of the central themes of
symplectic topology. In contrast to this, detecting {\it strong}
non-displaceabilty has at the moment the status of art rather than
science. That's why we were intrigued by Albers' observation that
monotone Lagrangian submanifolds satisfying his condition are in
some situations strongly non-displaceable. In the present work we
tried to digest Albers' results \cite{Albers} and look at them
from the viewpoint of theory of partial symplectic quasi-states
developed in \cite{EP-qst}. In addition, our result on
superheaviness of the Lagrangian anti-diagonal in $\SP^2 \times
\SP^2$ allows us to detect an  ``exotic" monotone Lagrangian torus
in this symplectic manifold: this torus does not intersect the
anti-diagonal, and hence is not heavy in contrast to the standard
Clifford torus, see Example~\ref{exam-exotic-torus} below.

In \cite{EP-qst} we proved a theorem which roughly speaking states
that every (singular) coisotropic foliation has at least one
non-displaceable fiber. However, our proof is non-constructive and
does not tell us which specific fibers are non-displaceable. The
notion of the special fiber arose as an attempt to solve this
problem for Hamiltonian circle actions.

\medskip
\noindent  Let us mention also that the {\bf product property}
enables us to produce even more examples of (super)heavy subsets by
taking products of the subsets appearing in the list.

\medskip
\noindent A few comments on the methods involved into our study of
heavy and superheavy subsets are in order. These collections are
defined in terms of partial symplectic quasi-states which were
introduced in \cite{EP-qst}. These are certain real-valued
functionals on $C^\infty (M)$ with rich algebraic properties which
are constructed by means of the Hamiltonian Floer theory and which
conveniently encode a part of the information contained in this
theory. In general, the definition of a partial symplectic
quasi-state involves the choice of an {\it idempotent element} in
the commutative part $QH_{\bullet} (M)$ of the quantum homology
algebra of $M$. Though the default choice is just the unity of the
algebra, there exist some other meaningful choices, in particular
in the case when $QH_\bullet (M)$ is semi-simple. This gives rise
to another theme discussed in this paper:  ``visible" topological
obstructions to semi-simplicity (see
Corollary~\ref{cor-semisimple-Torelli}
 and Theorem \ref{thm-lagr-semismp-a}
below). For instance, we shall show that if a monotone symplectic
manifold $M$ contains  ``too many" disjoint monotone Lagrangian
spheres whose minimal Maslov numbers exceed $n+1$, the quantum
homology $QH_\bullet (M)$ cannot be semi-simple.

Let us pass to the precise set-up. For  the reader's convenience,
the material presented in this brief outline will be repeated in
parts in the next sections in a less compressed form.

%XXXXXXXXXXXXXXXXXXXXXXXXXXXXXXXXXXXXXXXXXXXXXXXXXXXXXXXXXXXXXXXXXXXXX

\subsection{Preliminaries on quantum homology}
\label{subsec-prelim-qhom-1}

\medskip
\noindent {\sc The Novikov Ring:} Let $\cF$ denote a base field
which in our case will be either $\C$ or $\Z_2$, and let $\Gamma
\subset \R$ be a countable subgroup (with respect to the addition).
Let $s,q$ be formal variables. Define a field $\cK_{\Gamma}$ whose
elements are generalized Laurent series in $s$ of the following
form:
$$\cK_{\Gamma} := \bigg\{\ \sum_{\theta \in \Gamma} z_\theta s^\theta, \ z_\theta \in
\cF,\ \sharp\big\{ \theta > c\ |\ z_\theta\neq 0\big\} <\infty,\
\forall c\in\R\ \bigg\}\;.$$ Define a ring $\Lambda_{\Gamma} : =
\cK_{\Gamma} [q, q^{-1}]$ as the ring of polynomials in $q,
q^{-1}$ with coefficients in $\cK_{\Gamma}$.  We turn
$\Lambda_{\Gamma}$ into a graded ring by setting the degree of $s$
to be $0$ and the degree of $q$ to be $2$.

\medskip
 The ring $\Lambda_{\Gamma}$ serves as an abstract
model of the Novikov ring associated to a symplectic manifold. Let
$(M,\omega)$ be a closed connected symplectic manifold. Denote by
$H_2^S (M)$ the subgroup of spherical homology classes in the
integral homology group $H_2 (M; \Z)$. Abusing the notation we
will write $\omega (A)$, $c_1 (A)$ for the results of evaluation
of the cohomology classes $[\omega]$ and $c_1 (M)$ on $A\in H_2
(M; \Z)$. Set
$$\bar{\pi}_2 (M) := H_2^S (M) / \sim,$$
where by definition
$$A \sim B\  {\rm iff}\  \omega (A) = \omega (B) \ {\rm and}\ c_1 (A) = c_1 (B).$$
 Denote by $\Gamma(M,\omega) := [\omega] (H_2^S
(M))\subset \R$ the subgroup of periods of the symplectic form on
$M$ on spherical homology classes. By definition, the Novikov ring
of a symplectic manifold $(M,\omega)$ is
$\Lambda_{\Gamma(M,\omega)}$. In what follows, when $(M,\omega)$
is fixed, we abbreviate and write  $\Gamma$, $\cK$ and $\Lambda$
instead of $\Gamma(M,\omega)$, $\cK_{\Gamma(M,\omega)}$ and
$\Lambda_{\Gamma(M,\omega)}$ respectively.

\medskip
\noindent {\sc Quantum homology:} Set $2n = {\rm dim}\, M$. The
quantum homology $QH_* (M)$ is defined as follows. First, it is a
graded module over $\Lambda$ given by
$$QH_* (M) := H_* (M; \cF)\otimes_{\cF}  \Lambda,$$
with the grading defined by the gradings on $H_* (M; \cF)$ and
$\Lambda$:
$$ {\rm deg}\, (a\otimes zs^\theta q^k) := {\rm deg}\, (a) + 2k\;.$$

Second, and most important, $QH_* (M)$ is equipped with a {\it
quantum product}: if $a\in H_k (M; \cF)$, $b\in H_l (M; \cF)$, their
quantum product is a class $a * b\in QH_{k+l - 2n} (M)$, defined by
$$ a*b = \sum_{A\in \bar{\pi}_2 (M)} (a*b)_A \otimes s^{-\omega
(A)} q^{-c_1 (A)},$$ where $(a*b)_A \in H_{k+l - 2n + 2c_1 (A)}
(M)$ is defined by the requirement
$$ (a*b)_A \circ c = GW_A^{\cF} (a,b,c) \ \forall c\in H_* (M; \cF).$$
Here  $\circ$ stands for the intersection index and $GW_A^{\cF}
(a,b,c)\in \cF$ denotes the Gromov-Witten invariant which, roughly
speaking, counts the number of pseudo-holomorphic spheres in $M$
in the class $A$ that meet cycles representing $a,b,c\in H_*
(M;\cF)$ (see \cite{Ru-Ti}, \cite{Ru-Ti-1}, \cite{MS2} for the
precise definition).

Extending this definition by $\Lambda$-linearity to the whole $QH_*
(M)$ one gets a correctly defined graded-commutative associative
product operation $*$ on $QH_\ast (M)$ which is a deformation of the
classical $\cap$-product in singular homology \cite{Liu},
\cite{MS2}, \cite{Ru-Ti}, \cite{Ru-Ti-1}, \cite{Wi}. The {\it
quantum homology algebra} $QH_* (M)$ is a ring whose unity  is the
fundamental class $[M]$ and which is a module of finite rank over
$\Lambda$. If $a, b\in QH_\ast (M)$ have graded degrees $deg\, (a)$,
$deg\, (b)$ then
\begin{equation}
\label{eqn-quantum-product-grading} deg\, (a\ast b) = deg\, (a) +
deg\, (b) - 2n.
\end{equation}

We will be mostly interested in the commutative part of the
quantum homology ring (which in the case $\cF=\Z_2$ is, of course,
the whole quantum homology ring). For this purpose we introduce
the following notation:

\medskip
\noindent {\bf We denote by $QH_\bullet (M)$ the whole quantum
homology $QH_* (M)$ if $\cF=\Z_2$ and the even-degree part of
$QH_* (M)$ if $\cF=\C$.

\smallskip
\noindent In general, given a topological space $X$, we denote by
$H_\bullet (X; \cF)$ the whole singular homology group $H_* (X;
\cF)$ if $\cF=\Z_2$ and the even-degree part of $H_* (X; \cF)$ if
$\cF=\C$.}

\medskip
\noindent Thus, in our notation the ring $QH_\bullet (M) =
H_\bullet (M; \cF)\otimes_{\cF} \Lambda$ is always a commutative
subring with unity of $QH_* (M)$ and a module of finite rank over
$\Lambda$. We will identify $\Lambda$ with a subring of
$QH_\bullet (M)$ by $\lambda \mapsto [M]\otimes \lambda$.

\subsection{An hierarchy of rigid subsets within Floer theory}
\label{subseq-intro-2a}

Fix a non-zero idempotent $a\in QH_{2n} (M)$ (by obvious grading
considerations the degree of every idempotent equals $2n$). We
shall deal with spectral invariants $c(a, H)$, where $H = H_t:
M\to \R$, $t\in \R$, is a smooth time-dependent and 1-periodic in
time Hamiltonian function on $M$, or $c (a, \phi_H)$, where
$\phi_H$ is an element of the universal cover $\tHam (M)$ of $\Ham
(M)$ represented by an identity-based path given by the time-1
Hamiltonian flow generated by $H$. If $H$ is {\it normalized},
meaning that $\int_M H_t \omega^{{\rm dim}\, M/2} = 0$ for all
$t$, then $c (a, H) = c (a, \phi_H)$. These invariants, which
nowadays are standard objects of the Floer theory, were introduced
in \cite{Oh-spectral} (cf. \cite{Schwarz} in the aspherical case;
also see \cite{Oh1},\cite{Oh2} for an earlier version of the
construction and \cite{EP-qmm} for a summary of definitions and
results in the monotone case).

\medskip
\noindent {\sc Disclaimer:} Throughout the paper we tacitly assume
that $(M,\omega)$ (as well as $(M \times \T^2, \bar{\omega})$,
when we speak of stable displaceability) belongs to the class
$\cS$ of closed symplectic manifolds for which the spectral
invariants are well defined and enjoy the standard list of
properties (see e.g. \cite[Theorem 12.4.4]{MS2}). For instance,
$\cS$ contains all symplectically aspherical and spherically
monotone manifolds. Furthermore, $\cS$ contains all symplectic
manifolds $M^{2n}$ for which, on one hand, either $c_1=0$ or the
minimal Chern number (on $H_2^S (M)$) is at least $n-1$ and, on
the other hand, $[\omega] (H_2^S (M))$ is a discrete subgroup of
$\R$ (cf. \cite{Usher}). The general belief is that the class
$\cS$ includes {\bf all} symplectic manifolds.

\medskip
\noindent Define a functional $\zeta: C^{\infty}(M) \to \R$ by
\begin{equation}
\label{eqn-zeta-part-qstate} \zeta (H): = \lim_{l \to +\infty}\;
\frac{c (a,lH)}{l}
\end{equation}
It is shown in \cite{EP-qst} that the functional $\zeta$ has some
very special algebraic properties (see
Theorem~\ref{thm-partial-qstate-partial-qmm-basic}) which form the
axioms of {\it a partial symplectic quasi-state} introduced in
\cite{EP-qst}. The next definition is motivated in part by the
work of Albers \cite{Albers}.

\begin{defin}\label{def-heavy}{\rm
A closed subset $X \subset M$ is called {\it heavy} (with respect
to $\zeta$ or with respect to $a$ used to define $\zeta$) if
\begin{equation} \label{eq-heavy-0} \zeta(H) \geq \inf_X H \;\;
\forall H \in C^{\infty}(M)\;,
\end{equation}
and is called {\it superheavy} (with respect to $\zeta$ or $a$) if
\begin{equation} \label{eq-superheavy-0} \zeta(H) \leq \sup_X H \;\;
\forall H \in C^{\infty}(M)\;.
\end{equation}}
\end{defin}

\medskip
\noindent The default choice of an idempotent $a$ is the unity
$[M] \in QH_* (M)$. In this case, as we shall see below, the
collections of heavy and superheavy sets satisfy the properties
listed in
Section~\ref{subsec-intro-many-facets-of-displaceability} and
include the examples therein. In view of potential applications
(including geometric obstructions to semi-simplicity of the
quantum homology), we shall work, whenever possible, with general
idempotents.

The asymmetry between $\sup_X H$ and $\inf_X H$ is related to the
fact that the spectral numbers satisfy a triangle inequality $c
(a*b, \phi_F \phi_G) \leq c(a,\phi_F) + c(b,\phi_G)$, while there
may not be a suitable inequality  ``in the opposite direction". In
the case when such an  ``opposite" inequality exists (e.g. when
$a=b$ is an idempotent and $\zeta$ defined by it is a genuine {\it
symplectic quasi-state} -- see
Section~\ref{subsec-effect-semisimp} below) the symmetry between
$\sup_X H$ and $\inf_X H$ gets restored and the classes of heavy
and superheavy sets coincide.

 Let us emphasize
that the notion of (super)heaviness depends on the choice of a
coefficient ring for the Floer theory.   In this paper the
coefficients for the Floer theory will be either $\Z_2$ or $\C$
depending on the situation. Unless otherwise stated, our results
on (super)heavy subsets are valid for any choice  the
coefficients.

The group $\Symp\, (M)$ of all  symplectomorphisms of $M$ acts
naturally on $H_* (M;\cF)$ and hence on $QH_* (M) = H_* (M;
{\mathcal F})\otimes_{{\mathcal F}} \Lambda$. Clearly, the
identity component $\Symp_0 (M)$ of $\Symp\, (M)$ acts trivially
on $QH_* (M)$ and hence for any idempotent $a\in QH_* (M)$ the
corresponding $\zeta$ is $\Symp_0 (M)$-invariant. Thus the image
of a (super)heavy set under an element of $\Symp_0 (M)$ is again a
(super)heavy set with respect to the same idempotent $a$.  If $a$
is invariant under the action of the whole $\Symp\, (M)$ (for
instance, if $a = [M]$) the classes of heavy and superheavy sets
with respect to $a$ are invariant under the action of the whole
$\Symp\, (M)$ in agreement with the {\bf invariance} property
presented in
Section~\ref{subsec-intro-many-facets-of-displaceability} above.

Let us mention also that the collections of (super)heavy sets
enjoy a stability property under inclusions: If $X,Y$, $X \subset
Y$, are closed subsets of $M$ and $X$ is heavy (respectively,
superheavy) with respect to an idempotent $a$ then $Y$ is also
heavy (respectively, superheavy) with respect to the same $a$.

\medskip
\noindent We are ready now to formulate the main results of the
present section.

\begin{thm}\label{thm-intro-heavy}
Assume $a$ and $\zeta$ are fixed. Then
\begin{itemize}
\item[{(i)}] Every superheavy set is heavy, but, in general, not
vice versa. \item[{(ii)}] Every heavy subset is stably
non-displaceable. \item[{(iii)}] Every superheavy set intersects
every heavy set. In particular, a superheavy set cannot be
displaced by a {\bf symplectic} (not necessarily Hamiltonian)
isotopy and if the idempotent $a$ is invariant under the
symplectomorphism group of $(M,\omega)$ (e.g. if $a = [M]$), every
superheavy set is strongly non-displaceable.
\end{itemize}
\end{thm}

The following theorem discusses the relation between
heaviness/super\-heavi\-ness properties with respect to different
idempotents. In particular, it shows that $[M]$ plays a special
role among all the other non-zero idempotents in $QH_* (M)$.

\begin{thm}\label{thm-diff-idempotents}
Assume $a$ is a non-zero idempotent in the quantum homology. Then
\begin{itemize}

\item[{(i)}] Every set that is superheavy with respect to $[M]$ is
also superheavy with respect to $a$.

\item[{(ii)}] Every set that is heavy with respect to $a$ is also
heavy with respect to $[M]$.

\item[{(iii)}] Assume that the idempotent $a$ is a sum of non-zero
idempotents \break $e_1,\ldots,e_l$ and assume that a closed
subset $X\subset M$ is heavy with respect to $a$. Then $X$ is
heavy with respect to $e_i$ for at least one $i$.

\end{itemize}
\end{thm}

The next proposition shows that, in general, the heaviness of a
set {\it does depend} on the choice of an idempotent in the
quantum homology.

\begin{prop}
\label{prop-T2n-blown-up-at-one-point} Consider the torus
$\T^{2n}$ equipped with the standard symplectic structure $\omega
= dp\wedge dq$. Let $M^{2n}=\T^{2n}\sharp \overline{\C P^{n}}$ be
a symplectic blow-up of $\T^{2n}$ at one point (the blow up is
performed in a small ball around the point). Assume that the
Lagrangian torus $L\subset \T^{2n}$ given by $q=0$ does not
intersect the ball in $\T^{2n}$, where the blow up was performed.

Then the proper transform of $L$ (identified with $L$) is a
Lagrangian submanifold of $M$, which is not heavy with respect to
some non-zero idempotent $a\in QH_* (M)$ but heavy with respect to
$[M]$. (Here we work with $\cF=\Z_2$).
\end{prop}

\medskip
\noindent Next, consider direct products of (super)heavy sets. We
start with the following convention on tensor products. Let
$\Gamma_i$, $i=1,2$, be two countable subgroups of $\R$. Let $E_i$
be a module over $\cK_{\Gamma_i}$.   We put
\begin{equation}\label{eq-tensor-hat}
E_1 \widehat{\otimes}_{\cK} E_2 = \bigg{(}E_1
\otimes_{\cK_{\Gamma_1}} \cK_{\Gamma_1 +\Gamma_2}\bigg{)}
\otimes_{ \cK_{\Gamma_1 +\Gamma_2}} \bigg{(}E_2
\otimes_{\cK_{\Gamma_2}} \cK_{\Gamma_1 +\Gamma_2}\bigg{)}\;.
\end{equation}
If $E_1$, $E_2$ are also rings we automatically assume that the
middle tensor product is the tensor product of rings. In simple
words, we extend both modules to $\cK_{\Gamma_1
+\Gamma_2}$-modules and consider the usual tensor product over
$\cK_{\Gamma_1 +\Gamma_2}$.

Given two symplectic manifolds, $(M_1,\omega_1)$ and
$(M_2,\omega_2)$, note that the subgroups of periods of the
symplectic forms satisfy
$$\Gamma(M_1\times
M_2,\omega_1\oplus\omega_2) = \Gamma(M_1,\omega_1) +
\Gamma(M_2,\omega_2)\;.$$ Furthermore, due to the K\"unneth
formula for quantum homology (see e.g. \cite[Exercise
11.1.15]{MS2} for the statement in the monotone case; the general
case in our algebraic setup can be treated similarly) there exists
a natural ring monomorphism linear over $\cK_{\Gamma_1 +\Gamma_2}$
$$ QH_{2n_1} (M_1)\widehat{\otimes}_\cK QH_{2n_2} (M_2)\hookrightarrow QH_{2n_1 +
2n_2}(M_1 \times M_2)\;,$$  We shall fix a pair of idempotents
$a_i \in QH_* (M_i)$, $i=1,2$. The notions of (super)heaviness in
$M_1,M_2$ and $M_1 \times M_2$ are understood in the sense of
idempotents $a_1,a_2$ and $a_1 \otimes a_2$ respectively.

\begin{thm}\label{thm-products}
Assume that $X_i$ is a heavy (resp. superheavy) subset of $M_i$
with respect to some idempotent $a_i$, $i=1,2$. Then the product
$X_1 \times X_2$ is a heavy (resp. superheavy) subset of $M$ with
respect to the idempotent $a_1\otimes a_2\in QH_\bullet (M_1\times
M_2)$.
\end{thm}

\medskip
\noindent An important class of superheavy sets is given by stable
stems introduced and illustrated in
Section~\ref{subsec-intro-many-facets-of-displaceability}.

\begin{thm}\label{thm-stem}
Every stable stem is a superheavy subset with respect to any non-zero
idempotent $a\in QH_* (M)$. In particular, it is
strongly and stably non-displaceable.
\end{thm}

\medskip
\noindent In the next section we present an example of stable stems
coming from Hamiltonian torus actions.

%%%%%%%%%%%%%%%%%%%%%%%%%%%%%%%%%%%%%%%%%%%%%%%%%%%%%%%%%%%

\subsection{Hamiltonian torus actions}\label{subseq-intro-2}

Fibers of the moment maps of Hamiltonian torus actions form an
interesting playground for testing the various notions of
displaceability and heaviness introduced above. Throughout the
paper we deal with {\it effective} actions only, that is we assume
that the map $\varphi: \T^k\to \Ham (M)$ defining the action is a
monomorphism. Furthermore, we assume that the moment map $\Phi =
(\Phi_1,\ldots, \Phi_k): M\to\R^k$ of the action is normalized:
$\Phi_i$ is a normalized Hamiltonian for all $i=1,\ldots, k$. By
the Atiyah-Guillemin-Sternberg theorem \cite{Atiyah},
\cite{Guill-Stern}, the image $\Delta = \Phi (M)$ of $\Phi$ is a
$k$-dimensional convex polytope, called the {\it moment polytope}.
The subsets $\Phi^{-1}(p),\; p \in \Delta$, are called {\it
fibers} of the moment map. A torus action is called {\it
compressible} if the image of the homomorphism $\varphi_{\sharp}:
\pi_1(\T^k)\to \pi_1(\Ham(M))$, induced by the action $\varphi$,
is a finite group.

\begin{thm} \label{thm-stab-disp}
Assume that $(M,\omega)$ is equipped with a compressible
Hamiltonian $\T^k$-action with moment map $\Phi$ and moment
polytope $\Delta$. Let $Y \subset \Delta$ be any closed convex
subset which does not contain $0$. Then the subset $\Phi^{-1}(Y)$
is stably displaceable. In particular, the fiber $\Phi^{-1}(0)$ is
a stable stem.
\end{thm}

\medskip
\noindent Note that for symplectic toric manifolds, that is when
$2k =\dim M$, the point $0$ is the barycenter of the moment
polytope with respect to the Lebesgue measure. This follows from
our assumption on the normalization of the moment map.

\medskip
\noindent Theorems~\ref{thm-stem} and \ref{thm-stab-disp} imply that
the fiber $\Phi^{-1}(0)$ of a compressible torus action is stably
non-displaceable, and thus we get the complete description of stably
displaceable fibers for such actions.

In the case when the action is not compressible, the question of
the complete description of stably non-displaceable fibers remains
open. We make a partial progress in this direction by presenting
at least one such fiber, called {\it the special fiber},
explicitly in the case when $(M,\omega)$ is spherically monotone:
$$\left. [\omega]\right|_{H_2^S (M)} = \kappa \left. c_1
(TM)\right|_{H_2^S (M)}, \;\;\kappa >0\;.$$

The special fiber can be described via the mixed action-Maslov
homomorphism introduced in \cite{Pol-mixed}: Let $(M^{2n},\omega)$
be a spherically monotone symplectic manifold, and let $\{f_t\}, t
\in [0,1]$, be any loop of Hamiltonian diffeomorphisms, with $f_0
=f_1 =\id$, generated by a 1-periodic normalized Hamiltonian
function $F(x,t)$. The orbits of any Hamiltonian loop are
contractible due to the standard Floer theory\footnote{The Floer
theory guarantees the existence of at least one contractible
periodic orbit -- this is not obvious {\it a priori} if  $\{f_t\}$
is not an autonomous flow. Since all the orbits of $\{f_t\}$ are
homotopic, all of them are contractible.}. Pick any point $x \in
M$ and any disc $u:\D^2\to M$ spanning the orbit $\gamma =
\{f_tx\}$. Define the action\footnote{Note that our action
functional and the one in \cite{Pol-mixed} are of opposite signs.}
of the orbit by
$$\cA_F (\gamma,u) := \int_0^1 F(\gamma(t),t) dt - \int_{\D^2} u^\ast
\omega\;.$$

Trivialize the symplectic vector bundle $u^*(TM)$ over $\D^2$ and
denote by $m_F(\gamma,u)$ the Maslov index of the loop of
symplectic matrices corresponding to $\{f_{t*}\}$ with respect to
the chosen trivialization. One readily checks that, in view of the
spherical monotonicity, the quantity
$$I (F) := -\cA_F (\gamma,u)-\frac{\kappa}{2}m_F(\gamma,u)\;$$
does not depend on the choice of the point $x$ and the disc $u$,
and is invariant under homotopies of the Hamiltonian loop
$\{f_t\}$. In fact, $I$ is a well defined homomorphism from
$\pi_1(\Ham(M))$ to $\R$ (see \cite{Pol-mixed}, \cite{Weinst}).

Assume again that $\varphi: \T^k \to \Ham(M,\omega)$ is a
Hamiltonian torus action. Write $\varphi_{\sharp}$ for the induced
homomorphism of the fundamental groups. Since the target space
$\R^k$ of the moment map $\Phi$ is naturally identified with
$\text{Hom}(\pi_1(\T^k),\R)$, the pull back $-\varphi_{\sharp}^*I$
of the mixed action-Maslov homomorphism with the reversed sign can
be considered as a point of $\R^k$. We call it {\it a special point}
and denote by ${p}_{spec}$. The preimage $\Phi^{-1}({p}_{spec})$ is
called {\it the special fiber} of the moment map. In the case $k=1$,
when $\Phi$ is a real-valued function on $M$, we will call
$p_{spec}$ {\it the special value} of $\Phi$.

If $k=n$ and $M$ is a symplectic toric manifold, then
${p}_{spec}$ can be defined in purely combinatorial terms involving only
the polytope $\Delta$.
Namely, pick a vertex ${\bf x}$ of $\Delta$.
Since $\Delta$ in this case is a {\it Delzant polytope} \cite{Delz},
there is a unique (up to a permutation) choice of vectors
${\bf v}_1, \ldots, {\bf v}_n$
which
\begin{itemize}
  \item{originate at ${\bf x}$;}
  \item{span the $n$ rays containing the
edges of $\Delta$ adjacent to ${\bf x}$;      }
  \item{form a basis of $\Z^n$ over $\Z$.      }
\end{itemize}

\begin{prop}
\label{prop-p-spec}
\begin{equation}
\label{eqn-prop-p-spec}
{p}_{spec} = {\bf x} + {\kappa}\sum_{i=1}^n
{\bf v}_i.
\end{equation}

\end{prop}

\begin{proof}

The vertices of the moment polytope are in one-to-one
correspondence with the fixed points of the action. Let $x \in M$
be the fixed point corresponding to the vertex ${\bf x} = ({\bf
x}_1,\ldots, {\bf x}_n)$. Then the vectors ${\bf v}_j =
(v_j^1,\ldots, v_j^n)$, $j=1,\ldots, n$, are simply the weights of
the isotropy $\T^n$-action on $T_x M$. Since the definition of the
mixed action-Maslov invariant of a Hamiltonian circle action does
not depend on the choice of a 1-periodic orbit and a disc spanning
it, let us compute all $I_i$, $l=1, \ldots, n$, using the constant
periodic orbit concentrated at the fixed point $x$ and the
constant disc $u$ spanning it. Clearly, $$\cA_{\Phi_i} (x,u) =
\Phi_i (x)={\bf x}_i \;\; \text{and} \;\; m_{\Phi_i} (x,u) =
2\sum_{j=1}^n v_j^i \;\; \forall i=1,\ldots,n,$$ which readily
yields formula \eqref{eqn-prop-p-spec}.
\end{proof}

\medskip
\noindent E.Shelukhin pointed out to us that by summing up
equations \eqref{eqn-prop-p-spec} over all the vertices ${\bf
x}^{(1)},\ldots,{\bf x}^{(m)}\in\R^n$ of the moment polytope, one
readily gets that ${p}_{spec}= \frac{1}{m}\sum_i {\bf x}^{(i)}$.

\begin{thm}
\label{thm-main-nondispl-toric} Assume $M^{2n}$ is a spherically
monotone symplectic manifold equipped with a Hamiltonian
$\T^k$-action. Then the special fiber of the moment map is
superheavy with respect to any (non-zero) idempotent $a\in QH_{2n}
(M)$. In particular, it is stably and strongly non-displaceable.
\end{thm}

\medskip
\noindent Let us mention that, in particular, the special fiber is
non-empty and so ${p}_{spec} \in \Delta$. Moreover ${p}_{spec}$ is
an interior point of $\Delta$ -- otherwise $\Phi^{-1} ({p}_{spec})$
is isotropic of dimension $< n$ and hence displaceable (see e.g.
\cite{Biran-Ciel}).

\begin{rem}
\label{rem-spec-fiber-1} {\rm If $\dim M = 2\dim \T^k$ (that is we
deal with a symplectic toric manifold), the special fiber, say
$L$, is a Lagrangian torus. In fact, this torus is monotone: for
every $D \in \pi_2(M,L)$ we have
$$\int_D \omega = \kappa\cdot m^L(D)\;,$$
where $m^L$ stands for the Maslov class of $L$. This is an
immediate consequence of the definitions. }
\end{rem}

\begin{rem}
\label{rem-spec-fiber-2} {\rm Note that when $M$ is spherically
monotone and the action is compressible
Theorems~\ref{thm-stab-disp} and \ref{thm-main-nondispl-toric}
match each other: in this case $p_{spec}=0$ and therefore the
special fiber is a stable stem by Theorem~\ref{thm-stab-disp}. It
is unknown whether this property persists for the special fibers
of non-compressible actions. }
\end{rem}

\begin{exam}
\label{exam-2} {\rm Let $M$ be the monotone symplectic blow up of
$\C P^2$ at $k$ points ($0\leq k \leq 3$) which is equivariant
with respect to the standard $\T^2$-action and which is performed
away from the Clifford torus in $\C P^2$. Since the blow-up is
equivariant, $M$ comes equipped with a Hamiltonian $\T^2$-action
extending the $\T^2$-action on $\C P^2$. The Clifford torus is a
fiber of the moment map of the $\T^2$-action on $\C P^2$. Let
$L\subset M$ be the Lagrangian torus which is the proper transform
of the Clifford torus under the blow-up -- it is a fiber of the
moment map of the $\T^2$-action on $M$. Using
Proposition~\ref{prop-p-spec} it is easy to see that $L$ is the
special fiber of $M$. According to
Theorem~\ref{thm-main-nondispl-toric}, it is stably and strongly
non-displaceable. In fact, it is a stem: the displaceability of
all the other fibers was checked for $k=0$ in \cite{BEP}, for
$k=1$ in \cite{EP-qst} and for $k=2,3$ in
\cite{McDuff-letter}.}\end{exam}

\medskip
\noindent We refer to
Section~\ref{subsubsec-questions-Poisson-commut-subspaces} for
further discussion of related problems and very recent advances.

\bigskip
\noindent {\sc Digression: Calabi vs. action-Maslov.}
%\label{ex-calabi-action-maslov}
The method used to prove Theorem~\ref{thm-main-nondispl-toric}
also allows to prove the following result involving the mixed
action-Maslov homomorphism.  Denote by ${\hbox{\rm vol}\, (M)}$
the symplectic volume of $M$. Consider the function $\mu: \tHam
(M)\to\R$ defined by
$$ \mu (\phi_H) := - {\hbox{\rm vol}\, (M)}
\lim_{l\to +\infty} c (a, \phi_H^l )/l. $$ In the case when $a$ is
the unity in a field that is a direct summand in the decomposition
of the $\cK$-algebra $QH_{2n}(M,\omega)$, as an algebra, into a
direct sum of subalgebras, $\mu$ is a homogeneous quasi-morphism
on $\tHam(M)$ called {\it Calabi quasi-morphism}
\cite{EP-qmm},\cite{EP-toric-proc},\cite{Ostr-qmm}; in the general
case it has weaker properties \cite{EP-qst}. With this language
the functional $\zeta$ (on normalized functions) is induced (up to
a constant factor) by the pull-back of $\mu$ to the Lie algebra of
$\tHam(M)$.

Following P.Seidel we described in \cite{EP-qmm} the restriction
of $\mu$ (in fact, for {\it any} spherically monotone $M$) on
$\pi_1 (\Ham  (M))\subset \tHam (M)$ in terms of the Seidel
homomorphism $\pi_1 (\Ham  (M))\to QH^{inv}_* (M)$, where
$QH^{inv}_* (M)$ denotes the group of invertible elements in the
ring $QH_* (M)$. Here we give an alternative description of
$\left. \mu\right|_{\pi_1 (Ham (M))}$ in terms of the mixed
action-Maslov homomorphism $I$ which, in turn, also provides
certain information about the Seidel homomorphism.

\begin{thm}
\label{thm-Calabi-qmm-Ham-loops} Assume $M$ is spherically
monotone and let $\mu$ be defined as above for some  non-zero
idempotent $a\in QH_* (M)$. Then
$$\left. \mu\right|_{\pi_1 (\Ham  (M))} = {\hbox{\rm vol}\, (M)}\cdot I.$$

\end{thm}

Note that, in particular, $\left. \mu\right|_{\pi_1 (\Ham  (M))}$
does not depend on $a$ used to define $\mu$. The theorem also
implies that $\mu$ descends to a quasi-morphism on $\Ham  (M)$ if
and only if $I: \pi_1 (\Ham  (M))\to\R$ vanishes identically
(since $\mu$ descends to a quasi-morphism on $\Ham  (M)$ if and
only if $\left. \mu\right|_{\pi_1 (\Ham  (M))} \equiv 0$ -- see
e.g. \cite{EP-qmm}, Prop. 3.4). The proof of the theorem is given
in Section~\ref{sec-pf-thm-Calabi-qmm-Ham-loops}.

Let us mention also that, interestingly enough, the homomorphism
$I$ coincides with the restriction to $\pi_1 (\Ham  (M))$ of yet
another quasi-morphism on $\tHam (M)$ constructed by P.Py (see
\cite{Py-1,Py-2}).

\bigskip
\noindent {\sc Digression: Action-Maslov homomorphism and Futaki
invariant.}  This remark grew  from an observation pointed out to
us by Chris Woodward -- we are grateful to him for that. Assume
that our symplectic manifold $M$ is complex K\"ahler (i.e. the
symplectic structure on $M$ is induced by the K\"ahler one) and
Fano (by this we mean here that $[\omega]=c_1$). Assume also that
a Hamiltonian $\SP^1$-action $\{ f_t\}$ preserves the K\"ahler
metric and the complex structure. For instance, if $M^{2n}$ is a
symplectic toric manifold it can be equipped canonically with a
complex structure and a K\"ahler metric invariant under the
$\T^n$-action on $M$, hence under the action of any
$\SP^1$-subgroup $\{ f_t\}$ of $\T^n$.

Let $V$ be the Hamiltonian vector field generating the Hamiltonian
flow $\{ f_t\}$. Since $\{ f_t\}$ preserves the complex structure,
one can associate to $V$ its {\it Futaki invariant} ${\mathbb F}
(V)\in\C$ \cite{Futaki}. It has been checked by E.Shelukhin
\cite{Shelukhin} that, up to a universal constant factor, this
Futaki invariant is equal to the value of the mixed action-Maslov
homomorphism on the loop $\{ f_t\}$:
%\begin{equation}
%\label{eqn-Futaki-mixed-action-Maslov}
$${\mathbb F} (V) = {\it const}\, \cdot I (\{f_t\}).$$
%\end{equation}

Note that if such an $M$ admits a K\"ahler-Einstein metric then
the Futaki invariant has to vanish \cite{Futaki} -- thus if $I
(\{f_t\})\neq 0$ the manifold does not admit a K\"ahler-Einstein
metric. Moreover, if $M^{2n}$ is toric the opposite is also true:
if the Futaki invariant vanishes for any $V$ generating a subgroup
of the torus $\T^n$ acting on $M$ then $M$ admits a
K\"ahler-Einstein metric -- this follows from a theorem by Wang
and Zhu \cite{Wang-Zhu}, combined with a previous result of
Mabuchi \cite{Mabuchi}. In terms of the moment polytope, the
vanishing of the Futaki invariant, and accordingly the existence
of a K\"ahler-Einstein metric, on a K\"ahler Fano toric manifold
means precisely that the special point of the polytope coincides
with the barycenter.

%XXXXXXXXXXXXXXXXXXXXXXXXXXXXXXXXXXXXXXXXXXXXXXXXXXXXXXXXXXXXXXXXXXXXXXXXXXXXXXXXXXXX

\subsection{Super(heavy) monotone Lagrangian
submanifolds}\label{subsec-mon-lagr}

Let $(M^{2n},\omega)$ be a closed spherically monotone symplectic
manifold with $[\omega] = \kappa \cdot c_1(TM)$ on $\pi_2(M)$,
$\kappa > 0$. Let $L \subset M$ be a closed monotone Lagrangian
submanifold with the minimal Maslov number $N_L \geq 2$. As
usually, we put $N_L =+\infty$ if $\pi_2(M,L)=0$. As before, we
work with the basic field $\cF$ which is either $\Z_2$ or $\C$. In
the case $\cF =\C$, we assume that $L$ is relatively spin, that is
$L$ is orientable and the 2nd Stiefel-Whitney class of $L$ is the
restriction of some integral cohomology class of $M$.

\medskip \noindent {\bf Disclaimer:} In the case $\cF = \C$ the results of this section
are conditional: We take for granted that Proposition~\ref{pro-BC}
below, which was proved by Biran and Cornea \cite{BC}  for
homologies with $\Z_2$-coefficients, extends to homologies with
$\C$-coefficients. In each of the specific examples below we will
explicitly state which $\cF$ we are using and whenever we use
$\cF=\C$ we assume that $L$ is relatively spin.

\medskip
\noindent Denote by $j$ the natural morphism  $j: H_\bullet
(L;\cF) \to H_\bullet (M;\cF)$. We say that $L$ {\it satisfies the
Albers condition} \cite{Albers} if there exists an element $S \in
H_\bullet (L;\cF)$ so that $j(S) \neq 0$ and
$$\deg S > \dim L +1 -N_L\;.$$
We shall refer to such $S$ as to {\it an Albers element} of $L$.

\begin{exam}
\label{exam-fund-class-as-Albers-element} {\rm Assume $[L]\in
H_\bullet (L;\cF)$ and $j ([L])\in H_\bullet (M; \cF)$ is
non-zero. This means precisely that $[L]$ is an Albers element of
$L$.

A closed monotone Lagrangian submanifold $L$ which satisfies this
condition (and whose minimal Maslov number is greater than $1$)
will be called {\it homologically non-trivial} in $M$.
 }
\end{exam}

\medskip
\noindent \begin{thm} \label{thm-Albers-elem-implies-heavy} Let
$L$ be a closed monotone Lagrangian submanifold satisfying the
Albers condition. Then $L$ is heavy with respect to $[M]$. In
particular, any homologically non-trivial Lagrangian submanifold
is heavy with respect to $[M]$.

\end{thm}

\medskip
\noindent
\begin{exam} \label{exam-lagr-1}{\rm
\noindent Assume that $\pi_2 (M,L)=0$. Then the homology class of
a point is an Albers element of $L$, and hence $L$ is heavy. Note
that in this case heaviness cannot be improved to superheaviness:
the meridian on the two-torus is heavy but not superheavy. Here we
took $\cF=\Z_2$.}
\end{exam}

\medskip
\noindent \begin{exam}[Lagrangian spheres in Fermat hypersurfaces]
\label{exam-Fermat} {\rm More examples of heavy (but not
necessarily superheavy) monotone Lagrangian submanifolds can be
constructed as follows\footnote{We thank P.Biran for his
indispensable help with these examples.}.

Let $M\subset \C P^{n+1}$ be a smooth complex hypersurface of degree
$d$. The pull-back of the standard symplectic structure from $\C
P^{n+1}$ turns $M$ into a symplectic manifold (of real dimension
$2n$). If $d\geq 2$, then, as it is explained, for instance, in
\cite{Biran-ECM}, $M$ contains a Lagrangian sphere: $M$ can be
included into a family of algebraic hypersurfaces of $\C P^{n+1}$
with quadratic degenerations at isolated points and the vanishing
cycle of such a degeneration can be realized by a Lagrangian sphere
following \cite{Arn-Floer-mem-volume}, \cite{Donaldson},
\cite{Seidel-PhD}, \cite{Seidel-graded}, \cite{Seidel-ECM}.

Let $M\subset \C P^{n+1}$ be a projective hypersurface of degree
$d$, $2\leq d< n+2$. The minimal Chern number of $M$ equals
$N:=n+2-d>0$. Let $L^n\subset M^{2n}$ be a simply connected
Lagrangian submanifold (for instance, a Lagrangian sphere).

First, consider the case when $n$ is even, $L$ is relatively spin
and the Euler characteristics of $L$ does not vanish (this is the
case for a sphere).  Then the homology class $j([L])\in H_n (M;
\Z)$ is non-zero: its self-intersection number in $M$ up to the
sign equals the Euler characteristic. Thus $[L]$ is an Albers
element. (Here we use $\cF = \C$). In view of
Theorem~\ref{thm-Albers-elem-implies-heavy}, $L$ is heavy with
respect to $[M]$.

Second, suppose that $n$ is of arbitrary parity but $n > 2d-3$,
and no restriction on the Euler characteristics of $L$ is assumed
anymore. This yields $N_L=2N > n+1$ and thus $L$ satisfies the
Albers condition with the class of a point $P$ as an Albers
element. Thus $L$ is heavy with respect to $[M]$ -- here we use
$\cF=\Z_2$.

Finally, fix  $n \geq 3$ and an even number $d$ such that $4 \leq
d < n+2$. Consider a Fermat hypersurface of degree $d$
$$M = \{-z_0^d+z_1^d+\ldots+z_{n+1}^d = 0\} \subset \C P^{n+1}\;.$$
Its real part $L:= M \cap \R P^{n+1}$ lies in the affine chart $z_0
\neq 0$ and is given by the equation
$$x_1^d+\ldots+x_{n+1}^d=1,$$
where $x_j := \text{Re}(z_j/z_0)\;.$ Since $d$ is even, $L$ is an
$n$-dimensional sphere. As it was explained above, $L$ is heavy
with respect to $[M]$ if either $n$ is even (and $\cF=\C$) or $n >
2d-3$ (and $\cF=\Z_2$). However, in either case $L$ {\it is not
superheavy} with respect to $[M]$. Indeed, let $\Sigma_d \approx
\Z_d$ be the group of complex roots of unity. Given a vector
$\alpha=(\alpha_1,\ldots, \alpha_n)\in (\Sigma_d)^{n+1}$, denote
by $f_\alpha$ the symplectomorphism of $M$ given by
\begin{equation}\label{eq-fermat-alpha}
f_\alpha (z_0:z_1:\ldots :z_{n+1}) = (z_0: {\alpha_1}
z_1:\ldots:{\alpha_{n+1}} z_{n+1})\;.
\end{equation} If all $\alpha_j\in \C\setminus \R$,
then $\alpha_j x \notin \R$ whenever $x \in \R \setminus\{0\}$,
and thus $f_\alpha(L) \cap L = \emptyset$. Therefore $L$ is
strongly displaceable and the claim follows from the part (iii) of
Theorem~\ref{thm-intro-heavy}.}
\end{exam}

\medskip

The next result gives a user-friendly sufficient condition of
superheaviness.

\begin{thm}\label{thm-monot-superh}
Assume $L$ is homologically non-trivial in $M$ and assume $a \in
QH_{2n} (M)$ is a non-zero idempotent divisible by $ j([L])$ in
$QH_\bullet (M)$, that is $a \in j([L]) * QH_\bullet (M)$. Then
$L$ is superheavy with respect to $a$.
\end{thm}

\medskip
\noindent The homological non-triviality of $L$ in the hypothesis
of the theorem means just that $[L]$ is an Albers element of $L$
(see Example~\ref{exam-fund-class-as-Albers-element}). In fact,
the theorem can be generalized to the cases when $L$ has other
Albers elements -- see Remark~\ref{rem-BC-gen} (ii).

\medskip
\noindent
\begin{exam}[Lagrangian spheres in quadrics] \label{exam-quadric} {\rm Here we work with
$\cF=\C$. Let $M$ be the real part of the Fermat quadric
$M=\{-z_0^2+ \sum_{j=1}^{n+1} z_j^2 =0\}$. Assume that $n$ is even
and $L$ is a simply connected Lagrangian submanifold with
non-vanishing Euler characteristic (e.g. a Lagrangian sphere).
Under this assumption, $[L]\in H_\bullet (L)$ and $j([L])\neq 0$,
since $L$ has non-vanishing self-intersection. Denote by $p \in
H_* (M; \cF)$ the class of a point. The quantum homology ring of
$M$ was described by Beauville in \cite{Beauville}. In particular,
$p*p = w^{-2}[M]$, where $w = s^{\kappa n}q^{n}$.  Thus
$$a_{\pm}:= \frac{[M]\pm pw}{2}$$
are idempotents. One can show that $j([L])$ divides $a_-$ and
hence $L$ is $a_-$-superheavy.  Since $a_{-}$ is invariant under
the action of $\Symp (M)$, the manifold $L$ is strongly
non-displaceable.

 For simplicity, we present the calculation in the case $n=2$ -- the
general case is absolutely analogous. The 2-dimensional quadric is
symplectomorphic to $(\SP^2 \times \SP^2,\omega \oplus \omega)$.
Denote by $A$ and $B$ the classes of $[\SP^2]\times
[\text{point}]$ and $[\text{point}] \times [\SP^2]$ respectively.
Since the symplectic form vanishes on $j([L])$ we get that $j([L])
= l(B-A)$ with $l\neq 0$. It is known that $A*B=p$ and $B*B =
w^{-1}[M]$. Thus $j([L]) * \frac{1}{2l} wB = a_-$, that is
$j([L])$ divides $a_-$.

\medskip
\noindent In particular, the Lagrangian anti-diagonal
$$\Delta:= \{(x,y)\in \SP^2 \times \SP^2 \;:\; x=-y\}\;,$$
which is diffeomorphic to the $2$-sphere, is superheavy with
respect to $a_-$. It is unknown whether $\Delta$ is super-heavy
with respect to $a_+$. Further information on superheavy
Lagrangian submanifolds in the quadrics can be extracted from
\cite{BC}.}
\end{exam}

\medskip \noindent
\begin{exam}[A non-heavy monotone Lagrangian
torus in $\SP^2\times \SP^2$]\label{exam-exotic-torus} {\rm
Consider the quadric $M = \SP^2 \times \SP^2$ from the previous
example.  We will think of $\SP^2$ as of the unit sphere in $\R^3$
whose symplectic form is the area form divided by $4\pi$. We will
work again with $\cF=\C$. Interestingly enough, such an $M$
contains a monotone Lagrangian torus that is not heavy with
respect to $a_{-}$.

Namely, consider a submanifold $K$ given by equations\footnote{We
thank Frol Zapolsky for his help with calculations in this
example.}
$$K = \{(x,y) \in \SP^2 \times \SP^2\;:\;
x_1y_1+x_2y_2+x_3y_3=-\frac{1}{2}, x_3+y_3=0\}\;.$$ One readily
checks that $K$ is a monotone Lagrangian torus with $N_K=2$ which
represents a zero element in $H_2(M;\cF)$ (both with $\cF=\C$ and
$\cF=\Z_2$). Thus $H_\bullet (K; \cF)$ does not contain any Albers
element. Furthermore, $K$ is disjoint from the Lagrangian
anti-diagonal $\Delta$ and hence is not heavy with respect to
$a_{-}$ since, as it was shown above, $\Delta$ is superheavy with
respect to $a_{-}$. In particular, $K$ is {\it an exotic monotone
torus}: it is not symplectomorphic to the Clifford torus which is
a stem and hence $a_{-}$-superheavy.  A further study of exotic
tori in products of spheres is currently being carried out by
Y.Chekanov and F.Schlenk.

It is an interesting problem to understand whether $K$ is
superheavy with respect to $a_+$, or at least non-displaceable.
Identify $M \setminus\ \{\rm{ the\ diagonal}\}$ with the unit
co-ball bundle of the 2-sphere. After such an identification
$\Delta$ corresponds to the zero section, while $K$ corresponds to
a monotone Lagrangian torus, say $K'$. Interestingly enough, the
Lagrangian Floer homology of $K'$ in $T^* \SP^2$ (with $\cF =
\Z_2$) does not vanish as was shown by Albers and Frauenfelder in
\cite{AF}, and thus $K$ is not displaceable in $M \setminus\
\{{\rm the\ diagonal}\}$. Thus the question on
(non)-displaceability of $K$ is related to understanding of the
effect of the compactification of the unit co-ball bundle to
$\SP^2 \times \SP^2$. }
\end{exam}

\medskip

The proofs of theorems above are based on spectral estimates due
to Albers \cite{Albers} and Biran-Cornea \cite{BC}. Furthermore,
the results above admit various generalizations in the framework
of Biran-Cornea theory of quantum invariants for monotone
Lagrangian submanifolds, see \cite{BC} and the discussion in
Section~\ref{sec-Lagr-proofs} below.

\subsection{An effect of
semi-simplicity}\label{subsec-effect-semisimp}

Recall that a commutative (finite-dimensional) algebra $Q$ over a
field $\cB$ is called {\it semi-simple} if it splits into a direct
sum of fields as follows: $Q = Q_1 \oplus\ldots\oplus Q_d\;$,
where
\begin{itemize}
\item{}each $Q_i \subset Q$
is a finite-dimensional linear subspace over $\cB$;
\item{} each $Q_i$ is a
field with respect to the induced ring structure;
\item{} the multiplication in $Q$ respects the splitting:
$$(a_1,\ldots,a_d)\cdot(b_1,\ldots,b_d) = (a_1 b_1,\ldots,a_d b_d).$$
\end{itemize}
A classical theorem of Wed\-der\-burn (see e.g.
\cite{VanDerWaerden}, \S 96) implies that the semi-simplici\-ty is
equivalent to the absence of nilpotents in the algebra.

\begin{rem}
\label{rem-semisimplicity-genuine-qstate} {\rm Assume that the
$\cK$-algebra $QH_{2n}(M,\omega)$ splits, as an algebra, into a
direct sum of two algebras, at least one of which is a field, and
let $e$ be the unity in that field. In particular, this is the
case when $QH_{2n}(M,\omega)=Q_1 \oplus\ldots\oplus Q_d$ is
semi-simple and $e$ is the unity in one of the fields $Q_i$. A
slight generalization of the argument in \cite{EP-qst,Ostr-qmm}
(see \cite{EP-toric-proc}, the remark on pp. 56-57) shows that the
partial quasi-state $\zeta(e,\cdot)$ associated to $e$ is
$\R$-homogeneous (and not just $\R_+$-homogeneous as in the
general case).

This immediately yields that {\it every set which is heavy with
respect to $e$ is automatically superheavy with respect to $e$}.

In fact, in this situation $\zeta$ is a genuine {\it symplectic
quasi-state} in the sense of \cite{EP-qst} and, in particular, a
{\it topological quasi-state} in the sense of Aarnes \cite{Aar}
(see \cite{EP-qst} for details). In \cite{Aar} Aarnes proved
 an analogue of the Riesz representation theorem for
topological quasi-states which generalizes the correspondence
between genuine states (that is positive linear functionals on
$C(M)$) and measures. The object $\tau_\zeta$ corresponding to a
quasi-state $\zeta$ is called a {\it quasi-measure} (or a {\it
topological measure}). With this language in place, the sets that
are (super)heavy with respect to $\zeta$ are nothing else but the
closed sets of the full quasi-measure $\tau_\zeta$. Any two such
sets have to intersect for the following basic reason: any
quasi-measure is finitely additive on disjoint closed subsets and
therefore if two closed subsets of $M$ of the full quasi-measure
do not intersect, the quasi-measure of their union must be greater
than the total quasi-measure of $M$, which is impossible.}
\end{rem}

\medskip
\noindent\begin{exam}\label{exam-CPn} {\rm In this example we
again assume that $\cF=\Z_2$. Let $M=\C P^n$ be equipped with the
Fubini-Study symplectic structure $\omega$, normalized so that
$[\omega]=c_1$, and let $A\in H_{2n-2} (M)$ be the homology class
of the hyperplane. One readily verifies the following
$\cK$-algebra isomorphism
$$QH_{2n}(M)\cong \cK [X]/ \langle X^{n+1} -
u^{-1}\rangle,$$ where
$$\cK = \Z_2 [[u] = \{ z_k u^k + z_{k-1} u^{k-1}+\ldots , z_i\in \Z_2\  \forall i \}$$
is the field of Laurent-type series in $u:= s^{n+1}$ with
coefficients in $\Z_2$ and $X= q A$. Since no root of degree $2$
or more of $u^{-1}$ is contained in $\cK$, the polynomial $P$ is
irreducible over $\cK$ for any $n$ (see e.g. \cite{Lang}, Theorem
9.1) and therefore $QH_{2n} (M)$ is a field. Hence the collections
of heavy and superheavy sets with respect to the fundamental class
coincide.

We claim that $L:= \R P^n \subset \C P^n$ is superheavy. The case
$n=1$ corresponds to the equator of the sphere, which is known to
be a stable stem. For $n \geq 2$, note that $N_L =n+1$ and $S=[\R
P^2]$ is an Albers element of $L$. Therefore, $L$ is $[M]$-heavy
by Theorem~\ref{thm-Albers-elem-implies-heavy}, and hence
superheavy. }
\end{exam}

\medskip \noindent
The next result follows directly from
Theorem~\ref{thm-diff-idempotents} (iii) and
Remark~\ref{rem-semisimplicity-genuine-qstate}:

\begin{thm}\label{thm-quant-semis}
Assume that $QH_{2n}(M)$ is semi-simple and splits into a direct
sum of $d$ fields whose unities will be denoted by $e_1,\ldots,
e_d$. Assume that a closed subset $X\subset M$ is heavy with
respect to a non-zero idempotent $a$ -- as one can easily see,
such an idempotent has to be of the form $a=e_{j_1} +\ldots +
e_{j_l}$ for some $1\leq j_1<\ldots < j_l\leq d$. Then $X$ is
superheavy with respect to some $e_{j_i}$, $1\leq i\leq l$.

\end{thm}

\medskip
\noindent  The theorem yields the following geometric
characterization of non-semi\-sim\-pli\-city of $QH_{2n}(M)$.
Namely, define the {\it symplectic Torelli group} as the group of
all symplectomorphisms of $M$ which induce the identity map on
$H_\bullet (M; \cF)$. For instance, this group contains
$\Symp_0(M)$. Note that any element of the symplectic Torelli
group acts trivially on the quantum homology of $M$ and hence maps
sets (super)heavy with respect to an idempotent $a$ to sets
(super)heavy with respect to $a$.

Now Theorem~\ref{thm-quant-semis} readily implies the following

\begin{cor}
\label{cor-semisimple-Torelli} Assume that $(M,\omega)$ contains a
closed subset $X$ which is heavy with respect to a non-zero
idempotent and displaceable by a symplectomorphism from the
symplectic Torelli group. Then $QH_{2n}(M)$ is not semi-simple.

\end{cor}

 The
simplest examples are provided by sets of the form $X \times \{
\text{a\ meridian} \}$ in $M \times \T^2$ with a heavy $X$.

Another result in the same vein is as follows\footnote{In the case
$\cF=\C$, Theorem~\ref{thm-lagr-semismp-a} is conditional, see the
disclaimer in the previous section.}. Given a set $Y$ of positive
integers, put $\beta_Y(M) = \sum_{i \in Y} \beta_i(M)$, where
$\beta_i (M)$ stands for the $i$-th Betti number of $M$ over
$\cF$.

\begin{thm}
\label{thm-lagr-semismp-a} Assume that either of the following
(not mutually excluding) conditions holds:

\medskip
\noindent (a) $M$ contains $ m
> \beta_Y(M)+1$ pair-wise disjoint closed monotone  Lagrangian
submanifolds whose minimal Maslov numbers are greater than $n+1$
and belong to a set $Y$ of positive integers.

\medskip
\noindent (b) $M$ contains pair-wise disjoint homologically
non-trivial Lagrangian submanifolds\footnote{See
Example~\ref{exam-fund-class-as-Albers-element} for the
definition. As in that example we again assume that all our
Lagrangian submanifolds are closed, monotone and have minimal
Maslov number greater than $1$.} whose fundamental classes, viewed
as (non-zero) elements of \break $H_\bullet (M;\cF)$, are linearly
dependent over $\cF$.

\medskip
\noindent (In the case $\cF=\C$ assume that all the Lagrangian
submanifolds above are also relatively spin.)

\medskip
\noindent Then $QH_{2n}(M)$ is not semi-simple.
\end{thm}

\medskip
\noindent The proof is given in Section~\ref{sec-Lagr-proofs}.

\medskip
\noindent
\begin{exam}\label{rem-lagr-semisimp}
{\rm  For instance, if all the Lagrangian submanifolds from part
(a) of the theorem are simply connected, their minimal Maslov
numbers are equal to $2N$, so that the set $Y$ consists of one
element: $Y =\{2N\}$. Thus if $2N
> n+1$ and $QH_{2n}(M)$ is semi-simple, $M$ cannot contain more
than $\beta_{2N}(M)+1$ pair-wise disjoint simply-connected
Lagrangians (provided all of them are relatively spin if we work
with $\cF=\C$).}
\end{exam}

\begin{exam}
\label{exam-semisimplicity-Lagr-spheres} {\rm Set $\cF=\C$. Fix $n
\geq 11$ and an even number $d$ such that $6 \leq d < (n+3)/2$.
Consider a Fermat hypersurface of degree $d$
$$M = \{-z_0^d+z_1^d+\ldots+z_{n+1}^d = 0\} \subset \C P^{n+1}.$$
As we already saw in Example~\ref{exam-Fermat}, the manifold $L:=
M \cap \R P^{n+1}$ is an $n$-dimensional Lagrangian sphere.
Consider the images $f_{\alpha}(L)$, where symplectomorphisms
$f_{\alpha}$ are defined by \eqref{eq-fermat-alpha}. Note that, as
long as $\alpha_j/\beta_j \neq \pm 1$ for all $j$, the Lagrangian
spheres $f_\alpha (L)$ and $f_{\beta} (L)$ are disjoint. Using
this observation, it is easy to find $d/2$ disjoint Lagrangian
spheres in $M$.

The minimal Chern number $N$ of $M$ equals $n+2-d$, and so $2N$
lies in the interval $[n+2,2n-4]$. In this case $\beta_{2N}(M) =
1$ (see e.g. \cite{Hirzebruch}). Since $d/2 >2$, we conclude from
the previous example that $QH_{2n} (M)$ is not semi-simple. This
conclusion agrees with the computation of $QH_* (M)$ by Beauville
\cite{Beauville}.}
\end{exam}

\medskip
\noindent It would be interesting to find examples of symplectic
manifolds where the quantum homology is not known {\it a priori}
and where the above theorems are applicable. Let us mention that
different obstructions to the semi-simplicity of $QH_\bullet (M)$
coming from Lagrangian submanifolds were recently found by Biran
and Cornea \cite{Biran-Cornea}.

\subsection{Discussion and open questions}\label{subsec-disc}

\subsubsection{Strong displaceability beyond Floer theory?
}\label{subsubsec-pack}

Clearly, displaceability implies stable displaceability. The
converse is not true, as the next example shows:

\medskip
\noindent \begin{exam}\label{ex-disp-vs-stab}{\rm Consider the
complex projective space $\C P^n$ equipped with the Fubini-Study
symplectic form (in our normalization the area of a line equals
$1$). Identify $\C P^n$ with the symplectic cut of the Euclidean
ball $B(1) \subset \C^n$ (that is the boundary of $B(1)$ is
collapsed to $\C P^{n-1}$ along the fibers of the Hopf fibration,
see \cite{Lerman}), where $B(r):=\{\pi |z|^2 \leq r\}$. Then $B(r)
\subset \C P^n$ is:
\begin{itemize}
\item[{(i)}] displaceable for $r < 1/2$; \item[{(ii)}] strongly
non-displaceable but stably displaceable for $r \in [1/2, n/n+1)$;
\item[{(iii)}] strongly and stably non-displaceable for $r \geq
n/n+1$.
\end{itemize}
}\end{exam}

\medskip
\noindent It is instructive to analyze the techniques involved in
the proofs: The strong non-displaceability result in (ii) is an
immediate consequence of Gromov's packing-by-two-balls theorem,
which is proved via the $J$-holomorphic variant of the theorem which
states that there exists a $J$-holomorphic line in $\C P^n$ passing
through any two points. In the case (iii) the ball $B(r)$ contains
the Clifford torus, which is stably non-displaceable. This follows
either from the fact that the Clifford torus is a stem (see
\cite{BEP}), or from non-vanishing of its Lagrangian Floer homology
\cite{Cho}.

The displaceability of $B(r)$ in (i) follows from the explicit
construction of the two balls packing (see \cite{Karshon}). The
stable displaceability in (ii) is a direct consequence of
Theorem~\ref{thm-stab-disp} above: Indeed, consider the standard
$\T^n$-action on $\C P^n$. The normalized moment polytope $\Delta
\subset \R^n$ has the form $\Delta = \Delta_{stand} + w$ where
$\Delta_{stand}$ is the standard simplex $\{\rho_i \geq 0, \sum
\rho_i \leq 1\}$ in $\R^n$, where $(\rho_1,\ldots,\rho_n)$ denote
coordinates in $\R^n$, and $w = -\frac{1}{n+1}(1,\ldots,1)$. Note
that the ball $B(r)$ equals to $\Phi^{-1}(\Delta_r)$ where
$\Delta_r := r \cdot \Delta_{stand} +w$. Note that $\Delta_r$ does
not contain the origin exactly when $r \leq \frac{n}{n+1}$ which
yields the stable displaceability in (ii) above.

A mysterious feature of Example~\ref{ex-disp-vs-stab} is as
follows. On the one hand, we believe in the following general
empiric principle: whenever one can establish the
non-displaceability of a subset by means of the Floer homology
theory, one gets for free the stable non-displaceability. On the
other hand, we believe, following a philosophical explanation
provided by Biran, that Gromov's packing-by-two-balls theorem may
be extracted from some  ``operations" in Floer homology.
Example~\ref{ex-disp-vs-stab} shows that at least one of these
beliefs is wrong. It would be interesting to clarify this issue.

\subsubsection{Heavy fibers of Poisson-commutative subspaces}
\label{subsubsec-questions-Poisson-commut-subspaces}

It was shown in \cite{EP-qst} that for any finite-dimensional
Poisson-commutative subspace ${\mathbb A}\subset C^\infty (M)$ at least one of
the fibers of its moment map $\Phi$ has to be non-displaceable.

\medskip
\noindent {\bf Question.} Is it true that at least one fiber of
$\Phi$ has to be heavy (with respect to some non-zero idempotent
$a\in QH_* (M)$)?

\medskip
It is easy to construct an example of ${\mathbb A}$ whose moment
map $\Phi$ has no superheavy fibers: take $\T^2$ with the
coordinates $p,q \;{\rm mod}\ 1$ on it and take ${\mathbb A}$ to
be the set of all smooth functions depending only on $p$ -- the
corresponding $\Phi$ defines the fibration of $\T^2$ by meridians
none of which is superheavy.

Here is another question which concerns fibers of symplectic toric
manifolds, i.e. fibers of a moment map $\Phi$ of an effective
Hamiltonian $\T^n$-action on $(M^{2n}, \omega)$. Assume $M$ is
(spherically) monotone. Theorem~\ref{thm-main-nondispl-toric}
shows that in such a case the special fiber of $M$ is superheavy,
hence stably and strongly non-displaceable. In all the examples
where it has been checked this turns out to be the only
non-displaceable fiber of $M$.

\medskip
\noindent {\bf Question.} Is the special fiber for a monotone
symplectic toric $M$ always a stem? In particular, is it the only
non-displaceable fiber of the moment map?

\medskip

In the monotone case the special fiber is clearly the only heavy
fiber of the moment map, because it is superheavy and any other
heavy fiber would have had to intersect it. On the other hand, if
we consider a Hamiltonian $\T^k$-action on $M^{2n}$ with $k<n$
there can be more than one non-displaceable fiber of the moment
map -- for instance, because of purely topological obstructions:
the simplest Hamiltonian $\T^1$-action on $\C P^2$ provides such
an example. In the case of monotone symplectic toric manifolds of
dimension bigger than $4$ the question above is absolutely open.

After the first draft of this paper appeared, a remarkable
progress in this direction has been achieved in the works by Cho
\cite{Cho1} and Fukaya, Oh, Ohta and Ono \cite{FOOO-toric}: In
particular, it turns out that a non-monotone symplectic toric
manifold can have more than one non-displaceable fiber -- this
happens already for certain equivariant blowups of $\C P^2$.

\medskip
\noindent {\sc Organization of the paper:}

 In Section~\ref{sec-detect-disp} we prove Theorem~\ref{thm-stab-disp} which in
particular states that the special fiber of a compressible torus
action is a stable stem.

In Section~\ref{sec-prelim} we sum up various preliminaries from
Floer theory including basic properties of spectral invariants and
partial symplectic quasi-states. In addition we spell out a useful
property of the Conley-Zehnder index: it is a quasi-morphism on the
universal cover of the symplectic group (see
Proposition~\ref{prop-CZ-qmm}). For completeness we extract a proof
of this property from \cite{Rob-Sal}; alternatively, one can use the
results of \cite{DeGossons}.

In Section~\ref{sec-basic-proofs} we prove parts (i) and (iii) of
Theorem~\ref{thm-intro-heavy} and
Theorem~\ref{thm-diff-idempotents} on basic properties of
(super)heavy sets.

In Section~\ref{sec-proofs-prod} we prove Theorem~\ref{thm-products}
on products of (super)heavy sets. Our approach is based on a quite
general product formula for spectral invariants
(Theorem~\ref{thm-spectral-sum}), which is proved by a fairly
lengthy algebraic argument.

In Section~\ref{sec-stab-nondispl-heavy-sets} we prove
Theorem~\ref{thm-intro-heavy} (ii) on stable
non-dis\-pla\-ce\-abi\-lity of heavy subsets. The argument
involves a  ``baby version" of the above-men\-tio\-ned product
formula.

In Section~\ref{sec-anal-stable-stems} we prove superheaviness of
stable stems.

In Section~\ref{sec-Lagr-proofs} we bring together the proofs of
various results related to (super)heaviness of monotone Lagrangian
submanifolds satisfying the Albers condition, including
Theorems~\ref{thm-Albers-elem-implies-heavy},
\ref{thm-monot-superh}, \ref{thm-lagr-semismp-a} and
Proposition~\ref{prop-T2n-blown-up-at-one-point}.

In Section~\ref{sect-pf-thm-part-qstate} we prove
Theorem~\ref{thm-main-nondispl-toric} on superheaviness of special
fibers of Hamiltonian torus actions on monotone symplectic
manifolds. The proof is quite involved. In fact, two tricks
enabled us to shorten our original argument: First, we use the
Fourier transform on the space of rapidly decaying functions on
the Lie coalgebra of the torus in order to reduce the problem to
the case of Hamiltonian circle actions. Second, we systematically
use the quasi-morphism property of the Conley-Zehnder index for
asymptotic calculations with Hamiltonian spectral invariants.
Finally, in Section~\ref{sec-pf-thm-Calabi-qmm-Ham-loops} we prove
Theorem~\ref{thm-Calabi-qmm-Ham-loops}.

Figure 1 sums up the hierarchy of the non-displaceability
properties discussed above.

\begin{figure}
\label{fig-hierarchy}
\begin{center}
\includegraphics[width=5.4in]{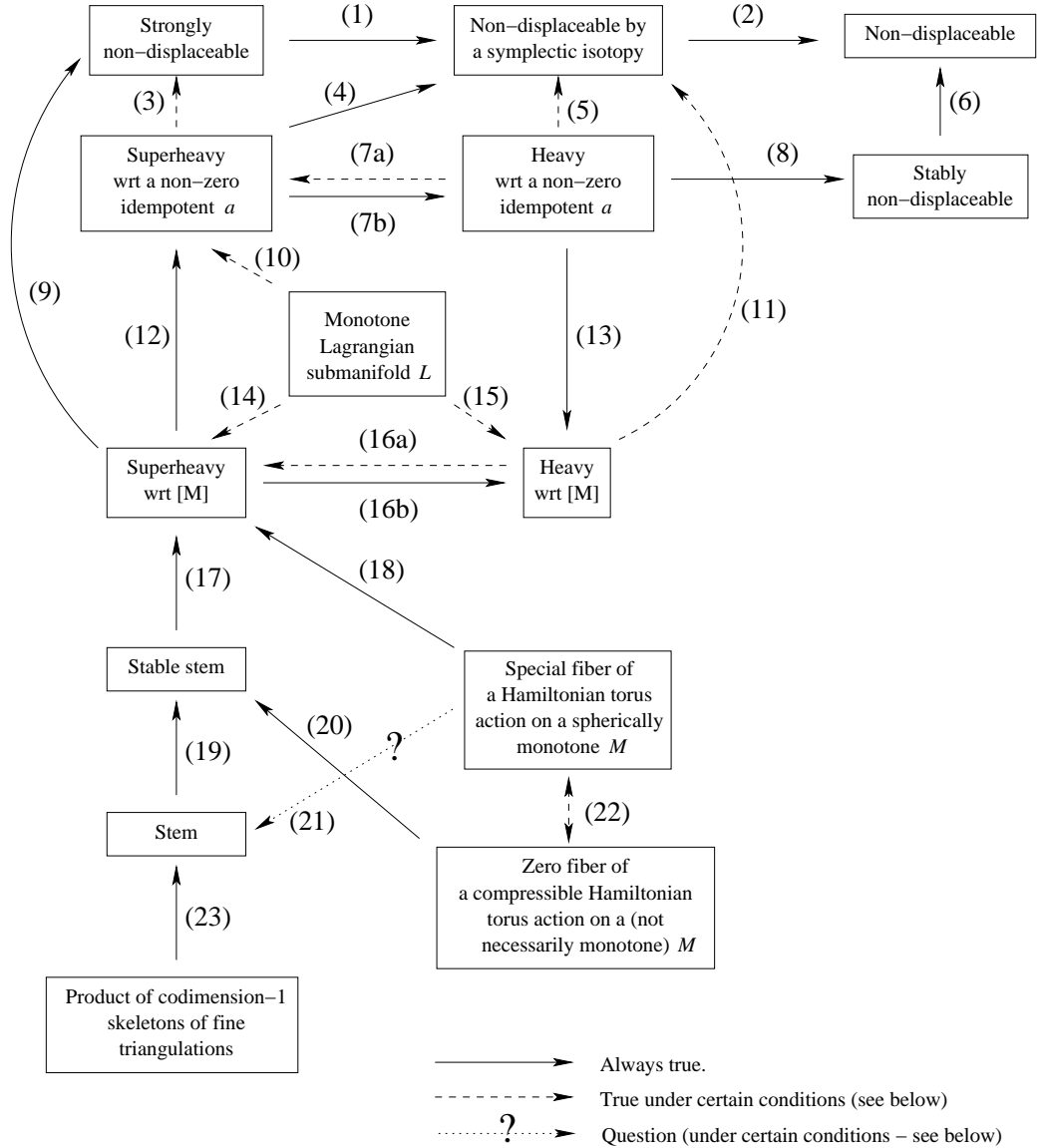}
\end{center}
\caption{Hierarchy of non-displaceability properties}
\end{figure}
{
%\scriptsize

\vfil\eject

\noindent (1),(2),(6),(19) - Trivial.

\noindent (3) True if $a$ is invariant under the action of the
whole group $\Symp\, (M)$ -- Theorem~\ref{thm-intro-heavy}, part
(iii).

\noindent (4), (9) Theorem~\ref{thm-intro-heavy}, part (iii).

\noindent (5) True if the algebra $QH_{2n} (M)$ is semi-simple
 -- see
Corollary~\ref{cor-semisimple-Torelli}.

\noindent (7a) True if the algebra $QH_{2n} (M)$ splits, as an
algebra, into a direct sum of two algebras, at least one of which
is a field, and $a$ is the unity element in that field  -- see
Remark~\ref{rem-semisimplicity-genuine-qstate}.

\noindent (7b), (16b) Theorem~\ref{thm-intro-heavy}, part (i).

\noindent (8) Theorem~\ref{thm-intro-heavy}, part (ii).

\noindent (10) Theorem~\ref{thm-monot-superh} (see the assumptions
on $L$ there).

\noindent (11) True if the algebra $QH_{2n} (M)$ is semi-simple
 -- see
Corollary~\ref{cor-semisimple-Torelli}.

\noindent (12) Theorem~\ref{thm-diff-idempotents}, part (i).

\noindent (13) Theorem~\ref{thm-diff-idempotents}, part (ii).

\noindent (14) Theorem~\ref{thm-monot-superh}  (see the
assumptions on $L$ there) with $a=[M]$ -- i.e. $j(L)$ is
invertible in $QH_\bullet (M)$.

\noindent (15) $L$ satisfies the Albers condition -- see
Theorem~\ref{thm-Albers-elem-implies-heavy}.

\noindent (16a) True if $QH_{2n} (M)$ is a field -- see
Remark~\ref{rem-semisimplicity-genuine-qstate}.

\noindent (17) Theorem~\ref{thm-stem}.

\noindent (18) Theorem~\ref{thm-main-nondispl-toric}.

\noindent (20) Theorem~\ref{thm-stab-disp}.

\noindent (21) Is the special fiber for a monotone {\it symplectic
toric} $M$ always a stem? See
Section~\ref{subsubsec-questions-Poisson-commut-subspaces}.

\noindent (22) True if $M$ is spherically monotone and the torus
action is compressible -- see Remark~\ref{rem-spec-fiber-2}.

\noindent (23) See \cite{EP-qst}. }

\section{Detecting stable displaceability} \label{sec-detect-disp}

For detecting stable displaceability of a subset of a symplectic
manifold we shall use the following result (cf. \cite[Chapter
6]{Pol-book}).

\begin{thm}\label{thm-stab-nondisp} Let $X$ be a closed subset of
a closed symplectic manifold $(M,\omega)$. Assume that there
exists a contractible loop of Hamiltonian diffeomorphisms of
$(M,\omega)$ generated by a normalized time-periodic Hamiltonian
$H_t(x)$ so that $H_t(x) \neq 0$ for all $t \in [0,1]$ and $x \in
X$. Then $X$ is stably displaceable.
\end{thm}

\begin{proof}
Denote by $h_t$ the Hamiltonian loop generated by $H$. Let
$h^{(s)}_t$ be its homotopy to the constant loop: $h^{(1)}_t= h_t$
and $h^{(0)}_t=\id$. Write $H^{(s)}(x,t)$ for the corresponding
normalized Hamiltonians. Consider the family of diffeomorphisms
$\Psi_s$ of $M\times T^*\SP^1$ given by
$$\Psi_s(x,r,\theta) = (h^{(s)}_{\theta}x,
r-H^{(s)}(h^{(s)}_{\theta}x,\theta),\theta)\;.$$ One readily
checks that $\Psi_s, s \in [0,1]$, is a Hamiltonian isotopy (not
compactly supported). We claim that $\Psi_1$ displaces $Y:=X
\times \{r=0\}$. Indeed, if $\Psi_1(x,0,\theta) \in Y$ we have
$h_\theta x \in X$ and $H_{\theta}(h_\theta x)=0$ which
contradicts the assumption of the theorem. This completes the
proof.
\end{proof}

\medskip
\noindent {\bf Proof of Theorem~\ref{thm-stab-disp}:} Choose a
linear functional $F: \R^k \to \R$ with rational coefficients
which is strictly positive on $Y$. Then for some sufficiently
large positive integer $N$ the Hamiltonian $H:= N\Phi^*F$
generates a contractible Hamiltonian circle action on $M$ and $H$
is strictly positive on $X:=\Phi^{-1}(Y)$. Thus $X$ is stably
displaceable in view of the previous theorem. \qed

\section{Preliminaries on Hamiltonian Floer theory}
\label{sec-prelim}

\subsection{Valuation on $QH_*  (M)$}\label{subsec-valuation}

Define a function $\nu: \cK \to \Gamma$ by
$$\nu(\sum z_\theta s^\theta)=
\max\{\ \theta \;{|}\; z_\theta \neq 0\}\;.$$ The convention is
that $\nu(0)=-\infty$.  In algebraic terms, $\exp \nu$ is a
non-Archimedean absolute value on $\cK$.

The function $\nu$ admits a natural extension to $\Lambda$ and
then to $QH_*  (M)$ -- abusing the notation we will denote all of
them by $\nu$. Namely, any element of $\lambda\in\Lambda$ can be
uniquely represented as $\lambda = \sum_\theta u_\theta s^\theta$,
where each $u_\theta$ belongs to $\cF[q,q^{-1}]$, and any non-zero
$a\in QH_*  (M)$ can be uniquely represented as $a = \sum_i
\lambda_i b_i$, $0\neq \lambda_i\in \Lambda$, $0\neq b_i\in H_*
(M;\cF)$. Define
$$\nu (\lambda) :=
\max \big\{ \theta \ |\ u_\theta\neq 0\big\},$$
$$\nu (a) := \max_i \nu(\lambda_i).$$

\subsection{Hamiltonian Floer theory}
\label{sec-Floer-theory-basics}

We briefly recall the notation and conventions for the setup of
the Hamiltonian Floer theory that will be used in the proofs.

Let $\cL$ be the space of all smooth contractible loops $\gamma:
\SP^1 = \R/\Z \to M$. We will view such a $\gamma$ as a 1-periodic
map $\gamma: \R\to M$. Let $\D^2$ be the standard unit disk in
$\R^2$. Consider a covering $\tcL$ of $\cL$ whose elements are
equivalence classes of pairs $(\gamma,u)$, where $\gamma \in \cL$,
$u: \D^2\to M$, $\left. u\right|_{\partial \D^2} = \gamma$ (i.e.
$u (e^{2\pi \sqrt{-1} t}) = \gamma (t)$), is a (piecewise smooth)
disk spanning $\gamma$ in $M$ and the equivalence relation is
defined as follows: $(\gamma_1,u_1) \sim (\gamma_2,u_2)$ if and
only if $\gamma_1 = \gamma_2$ and the 2-sphere $u_1 \# (-u_2)$
vanishes in $H_2^S (M)$. The equivalence class of a pair $(\gamma,
u)$ will be denoted by $[\gamma,u]$. The group of deck
transformations of the covering $\tcL\to\cL$ can be naturally
identified with $H_2^S (M)$. An element $A \in H_2^S (M)$ acts by
the transformation
\begin{equation}
\label{act1} A ([\gamma,u]) = [\gamma,u \# (-A)].
\end{equation}

Let $F: M\times [0,1]\to\R$ be a Hamiltonian function (which is
time-periodic as we always assume). Set $F_t := F (\cdot, t)$. We
will denote by $f_t$ the Hamiltonian flow generated by $F$,
meaning the flow of the time-dependent Hamiltonian vector field
$X_t$ defined by the formula
\[
\omega (\cdot, X_t) = dF_t (\cdot)\ \ \forall t.
\]
(Note our sign convention!)
%For brevity we
%will write $f: =f_1$.

Let $\cP_F \subset \cL$ be the set of all contractible 1-periodic
orbits of the Hamiltonian flow generated by $F$, i.e. the set of
all $\gamma\in \cL$ such that $\gamma (t) = f_t (\gamma (0))$.
Denote by $\tP_F$ the full lift of $\cP_F$ to $\tcL$.

Denote by $\Fix (F)$ the set of those fixed points of $f$ that are
endpoints of contractible periodic orbits of the flow:
\[
\Fix (F) := \{ x\in M \ | \ \exists \gamma\in \cP_F,\ \ x=\gamma
(0)\}.
\]
We say that $F$ is {\it regular} if for any $x\in \Fix (F)$ the
map $d_x f : T_x M\to T_x M$ does not have eigenvalue $1$.

Recall that the {\it action functional} is defined on $\tcL$ by
the formula:
$$\cA_F ([\gamma, u]) = \int_0^1 F(\gamma (t), t) dt - \int_{\D^2}
u^*\omega.$$ Note that
\begin{equation}
\label{act2} \cA_F(Ay) =\cA_F(y) + \omega (A)
\end{equation}
for all $y \in \tcL$ and $A \in H_2^S (M)$.

For a regular Hamiltonian $F$ define a vector space $C (F)$ over
$\cF$ as the set of all formal sums
$$\sum_{i=1}^k  \lambda_i y_i,\, \lambda_i \in \Lambda, y_i\in \tP_F,$$
modulo the relations
$$Ay = s^{-\omega (A)} q^{-c_1 (A)} y, $$
for all $y\in\tP_F, A\in H_2^S (M)$.
The grading on $\Lambda$ together with the Conley-Zehnder index on elements of
$\tP_F$ (see Section~\ref{subsubsec-CZ-Maslov}) defines a $\Z$-grading on $C (F)$. We will denote
the $i$-th graded component by $C_i (F)$.

Given a loop $\{J_t\},\; t\in \SP^1$, of $\omega$-compatible
almost complex structures, define a Riemannian metric on $\cL$ by
$$(\xi_1,\xi_2) = \int_0^1 \omega (\xi_1 (t), J_t\xi_2(t))dt, $$
where $\xi_1,\xi_2 \in T_\gamma \cL$. Lift this metric to $\tcL$
and consider the negative gradient flow of the action functional
$\cA_F$. For a generic choice of the Hamiltonian $F$ and the loop
$\{J_t\}$ (such a pair $(F,J)$ is called {\it regular}) the count
of isolated gradient trajectories connecting critical points of
$\cA_F$ gives rise in the standard way \cite{Floer}, \cite{Ho-Sa},
\cite{Sal} to a Morse-type differential
\begin{equation}
\label{eqn-D-def}
d: C (F) \to C (F),\; d^2 =0.
\end{equation}
The differential $d$ is $\Lambda$-linear and has the graded degree
$-1$. It strictly decreases the action. The homology, defined by
$d$, is called the {\it Floer homology} and will be denoted by
$HF_* (F,J)$. It is a $\Lambda$-module. Different choices of a
regular pair $(F,J)$ lead to natural isomorphisms between the
Floer homology groups.

The following proposition summarizes a few basic algebraic
properties of Floer complexes and Floer homology that will be
important for us further. The proof is straightforward and we omit
it.

\begin{prop}
\label{prop-basic-alg-properties-of-Floer-complex-Floer-homology}
\

\medskip
1) Each $C_i (F)$ and each $HF_i (F,J)$, $i\in\Z$, is a
finite-dimensional vector space over $\cK$.

\medskip
2) Multiplication by $q$ defines isomorphisms $C_i (F)\to C_{i+2} (F)$
and \break $HF_i (F,J)\to HF_{i+2} (F,J)$ of $\cK$-vector spaces.

\medskip
3) For each $i\in\Z$ there exists a basis of $C_i (F)$ over $\cK$
consisting of the elements of the form $q^l [\gamma, u]$, with
$[\gamma, u]\in \tP_F$.

\medskip
4) A finite collection of elements of the form $q^l [\gamma, u]$,
$[\gamma, u]\in \tP_F$, lying in $C_0 (F) \cup C_1 (F)$ is a basis
of the vector space $C_0 (F)\oplus C_1 (F)$ over the field $\cK$
if and only if it is a basis of the module $C(F)$ over the ring
$\Lambda$.

\end{prop}

%XXXXXXXXXXXXXXXXXXXXXXXXXXXX
\subsection{Conley-Zehnder and Maslov indices}
\label{subsubsec-CZ-Maslov}

In this section we briefly outline the definition and recall the
relevant properties of the Conley-Zehnder index referring to
\cite{Rob-Sal, Sal, Sal-Ze} for details. In particular, we show
that {\it the Conley-Zehnder index is a quasi-morphism on the
universal cover $\widetilde{Sp\, (2k)}$ of the symplectic group
$Sp(2k)$} (see Proposition~\ref{prop-CZ-qmm} below), a fact which
will be useful for asymptotic calculations with Floer homology in
the next sections. There are several routes leading to this fact,
which is quite natural since all homogeneous quasi-morphisms on
$\widetilde{Sp\, (2k)}$ are proportional, and hence the same
quasi-morphism admits quite dissimilar definitions
\cite{Barge-Ghys}. We extract the quasi-morphism property from the
paper of Robbin and Salamon \cite{Rob-Sal} by bringing together
several statements contained therein\footnote{We thank V.L.
Ginzburg for stimulating discussions on the material of this
section.}.

The  Conley-Zehnder index assigns to each $[\gamma, u]\in \tP_F$ a
number. Originally the Conley-Zehnder index was defined only for
regular Hamiltonians \cite{Co-Ze} -- in this case it is
integer-valued and gives rise to a grading of the homology groups in
Floer theory. Later the definition was extended in different ways by
different authors to arbitrary Hamiltonians. We will use such an
extension introduced in \cite{Rob-Sal} (also see \cite{Sal-Ze,
Sal}). In this case the Conley-Zehnder index may take also
half-integer values.

Let $k$ be a natural number. Consider the symplectic vector space
$\R^{2k}$ with a symplectic form $\omega_{2k}$ on it. Denote by
$p=(p_1,\ldots,p_k),q=(q_1,\ldots, q_k)$ the corresponding Darboux
coordinates on the vector space $\R^{2k}$.

\medskip \noindent {\sc Robbin-Salamon index of Lagrangian paths:}
Let $V\subset \R^{2k}$ be a Lagrangian subspace. Consider the
Grassmannian ${\it Lagr}\, (k)$ of all Lagrangian subspaces in
$\R^{2k}$ and consider the hypersurface $\Sigma_V\subset {\it
Lagr}\, (k)$ formed by all the Lagrangian subspaces that are {\it
not} transversal to $V$. To such a $V$ and to any smooth path $\{
L_t \}$, $0\leq t\leq 1$, in ${\it Lagr}\, (k)$ Robbin and Salamon
\cite{Rob-Sal} associate an index, which may take integer or
half-integer values and which we will denote by $RS (\{ L_t\}, V)$.
The definition of the index can be outlined as follows.

A number $t\in [0,1]$ is called a {\it crossing} if $L_t
\in\Sigma_V$. To each crossing $t$ one associates a certain
quadratic form $Q_t$ on the space $L(t)\cap V$ -- see
\cite{Rob-Sal} for the precise definition. The crossing $t$ is
called {\it regular} if the quadratic form $Q_t$ is
non-degenerate. The {\it index} of such a regular crossing $t$ is
defined as the signature of $Q_t$ if $0<t<1$ and as half of the
signature of $Q_t$ if $t=0,1$. One can show that regular crossings
are isolated. For a path $\{ L_t \}$ with only regular crossings
the index $RS (\{ L_t \}, V)$ is defined as the sum of the indices
of its crossings. An arbitrary path can be perturbed, keeping the
endpoints fixed, into a path with only regular crossings and the
index of the perturbed path does not depend on the perturbation --
in fact, it depends only on the fixed endpoints homotopy class of
the path. Moreover, it is additive with respect to the
concatenation of paths and satisfies the naturality property: $RS
(\{AL_t\}, AV) = RS (\{ L_t\}, V)$ for any symplectic matrix $A$.

\medskip\noindent {\sc Indices of paths in $Sp\, (2k)$:}
Consider the group $Sp\, (2k)$ of symplectic $2k\times 2k$-matrices.
Denote by $\widetilde{Sp\, (2k)}$ its universal cover. One can use
the index $RS$ in order to define two indices on the space of smooth
paths in $Sp\, (2k)$.

The first index, denoted by $Ind_{2k}$, is defined as follows. Fix a
Lagrangian subspace $V\subset \R^{2k}$. For each smooth path $\{ A_t
\}$, $0\leq t\leq 1$, in $Sp\, (2k)$ define $Ind_{2k}\, (\{ A_t \},
V)$ as
$$Ind_{2k}\, (\{ A_t \}, V) := RS (\{ A_t V\}, V).$$
The naturality of the $RS$ index implies that
$$RS (\{ B A_t B^{-1} (BV)\}, BV) =
RS (\{ B A_t V)\}, BV) = $$
$$= RS (\{ A_t V)\}, V)
\ {\rm for\ any}\ B\in Sp\, (2k)$$ and thus we get the following
naturality condition for $Ind_{2k}$:
\begin{equation}
\label{eqn-naturality-Ind} Ind_{2k}\, (\{BA_t B^{-1}\}, BV) =
Ind_{2k}\, (\{ A_t \}, V)\ {\rm for\ any}\ B\in Sp\, (2k).
\end{equation}

The second index, which we will call the {\it Conley-Zehnder index
of a matrix path} and which will be denoted by $CZ_{matr}$, is
defined as follows. For each $A\in Sp\, (2k)$ denote by $Gr\, A$ the
graph of $A$ which is a Lagrangian subspace of the symplectic vector
space $\R^{4k} = \R^{2k}\times \R^{2k}$ equipped with the symplectic
structure $\omega_{4k} = -\omega_{2k}\oplus \omega_{2k}$. Denote by
$\Delta$ the diagonal in $\R^{4k} = \R^{2k}\times \R^{2k}$ -- it is
a Lagrangian subspace with respect to $\omega_{4k}$. Now for any
smooth path $\{ A_t \}$, $0\leq t\leq 1$, in $Sp\, (2k)$ define
$CZ_{matr}$ as
$$CZ_{matr} (\{ A_t\}) := RS (\{ Gr\, A_t\}, \Delta).$$
Equivalently, one can define $CZ_{matr} (\{ A_t\})$ similarly to
the index $RS$ by looking at the intersections of $\{ A(t)\}$ with
the hypersurface $\Sigma\subset Sp\, (2k)$ formed by all the
symplectic $2k\times 2k$-matrices with eigenvalue $1$ and
translating the notions of a regular crossing and the
corresponding quadratic form to this setup.

Both indices $Ind_{2k}\, (\{ A_t \}, V)$ and $CZ_{matr} (\{ A_t\})$
depend only on the fixed endpoints homotopy class of the path $\{
A_t \}$ and are additive with respect to the concatenation of paths
in $Sp\, (2k)$. The relation between the two indices is as follows.
Denote by $I_{2k}$ the $2k\times 2k$ identity matrix. Given a smooth
path $\{ A_t \}$, $0\leq t\leq 1$, in $Sp\, (2k)$, set
$\widehat{A}_t := I_{2k} \oplus A_t  \in Sp\, (4k)$. Then
\begin{equation}
\label{eqn-CZ-Ind-4n} CZ_{matr} (\{ A_t\}) = Ind_{4k} (\{
\widehat{A}_t \}, \Delta).
\end{equation}

\bigskip
\begin{rem}
\label{rem-Ind-CZ-matr-bounded-change-under-perturbations} {\rm Note
that near each $W\in\Sigma_V$ there exists a local coordinate chart
(on ${\it Lagr}\, (k)$) in which $\Sigma_V$ can be defined by an
algebraic equation of degree bounded from above by a constant $C$
depending only on $k$ and $W$. Moreover, since for any two $V,V'\in
{\it Lagr}\, (k)$ there exists a diffeomorphism of ${\it Lagr}\,
(k)$ mapping $\Sigma_V$ into $\Sigma_{V'}$ we can assume that
$C=C(k)$ is independent of $W$ and depends only on $k$. Therefore
for any $V$, for any point $W\in\Sigma_V$ and for any sufficiently
small open neighborhood $U_W$ of $W$ in ${\it Lagr}\, (k)$ the
number of connected components of $U_W\setminus (U_W\cap \Sigma_V)$
is bounded by a constant depending only on $k$.

Using these observations and the fact that regular crossings are
isolated it is easy to show that there exists a constant $C(k)$,
depending only on $k$, such that for any Lagrangian subspace
$V\subset \R^{2k}$ and any path $\{ A_t\} \subset Sp\, (2k)$, $0\leq
t\leq 1$, there exists a $\delta>0$ such that for any smooth path
$\{ A'_t\}\subset Sp\, (2k)$, $0\leq t\leq 1$, which is
$\delta$-close to $\{ A_t\}$ in the $C^0$-metric, one has
$$|Ind_{2k} (\{ A_t\}, V) - Ind_{2k} (\{ A'_t\}, V| < C(k),$$
$$|CZ_{matr} (\{ A_t\}) - CZ_{matr} (\{ A'_t\}| < C(k).$$

}
\end{rem}

\medskip\noindent {\sc Leray theorem on the index $Ind_{2k}$:}
The following result follows from Theorem 5.1 in \cite{Rob-Sal}
which Robbin and Salamon credit to Leray \cite{Leray}, p.52. Denote
by $L$ the Lagrangian $(q_1,\ldots,q_k)$-coordinate plane in
$\R^{2k}$. Any symplectic matrix $S\in Sp\, (2k)$ can be decomposed
into $k\times k$ blocks as
$$S=\left(
\begin{array}{cc}
  E & F \\
  G & H \\
\end{array}
\right),$$ where the blocks satisfy, in particular, the condition
that
\begin{equation}
\label{eqn-ABT-BAT-0} EF^T - FE^T=0.
\end{equation}
If $SL\cap L = 0$ then the $k\times k$-matrix $F$ is invertible and
multiplying (\ref{eqn-ABT-BAT-0}) by $F^{-1}$ on the left and
$(F^T)^{-1} = (F^{-1})^T$ on the right, we get that $F^{-1}E - E^T
(F^{-1})^T = 0$. Therefore the matrix $Q_S := F^{-1} E$ is
symmetric.
\begin{thm}[\cite{Rob-Sal}, Theorem 5.1; \cite{Leray}, p.52]
\label{thm-leray} Assume $\{ A_t\}$, $\{ B_t\}$, $0\leq t\leq 1$,
are two smooth paths in $Sp\, (2k)$, such that $A_0=B_0=I_{2k}$ and
$A_1 L\cap L = 0$, $B_1 L\cap L = 0$, $A_1 B_1 L\cap L = 0$. Then
$$Ind_{2k} (\{ A_t B_t\}, L) = Ind_{2k} (\{ A_t\}, L) + Ind_{2k} (\{ B_t\}, L)
+\frac{1}{2} {\rm sign}\, (Q_{A_1} + Q_{B_1}),$$ where ${\rm sign}\,
(Q_{A_1} + Q_{B_1})$ is the signature of the quadratic form defined
by the symmetric $k\times k$-matrix $Q_{A_1} + Q_{B_1}$.
\end{thm}

\begin{cor}
\label{cor-Ind-qmm} Let $V$ be any Lagrangian subspace of $\R^{2k}$.
Then there exists a positive constant $C$, depending only on $k$,
such that for any smooth paths  $\{ X_t\}$, $\{ Y_t\}$, $0\leq t\leq
1$, in $Sp\, (2k)$, such that $X_0=Y_0=I_{2k}$ (there are no
assumptions on $X_1$, $Y_1$!),
$$| Ind_{2k} (\{ X_t Y_t\}, V) - Ind_{2k} (\{ X_t\}, V) -
Ind_{2k} (\{ Y_t\}, V)| < C.$$

\end{cor}

\begin{proof}

We will write $C_1, C_2,\ldots$ for (possibly different) positive
constants depending only on $k$.

Pick a map $\Psi\in Sp\, (2k)$ such that $\Psi V = L$. Denote $A_t =
\Psi X_t \Psi^{-1}$, $B_t = \Psi Y_t \Psi^{-1}$. Note that the paths
$\{ A_t\}$, $\{ B_t\}$ are based at the identity.

Using the naturality property (\ref{eqn-naturality-Ind}) of
$Ind_{2k}$ we get
$$| Ind_{2k} (\{ X_t Y_t  \}, V) - Ind_{2k} (\{ X_t\}, V) -
Ind_{2k} (\{ Y_t\}, V)| = $$
$$
= | Ind_{2k} (\{ \Psi X_t Y_t \Psi^{-1} \}, \Psi V) - Ind_{2k} (\{
\Psi X_t \Psi^{-1}\}, \Psi V) - $$
$$ - Ind_{2k} (\{ \Psi Y_t \Psi^{-1} \}, \Psi V)| = $$
$$ =
| Ind_{2k} (\{ (\Psi X_t \Psi^{-1}) (\Psi Y_t\Psi^{-1}) \}, L) -
Ind_{2k} (\{ \Psi X_t \Psi^{-1} \},  L) - $$
$$ - Ind_{2k} (\{ \Psi Y_t \Psi^{-1}\}, L)| = $$
$$ = | Ind_{2k} (\{ A_t B_t  \}, L) - Ind_{2k} (\{ A_t\}, L) -
Ind_{2k} (\{ B_t\}, L)|.$$ Thus
\begin{eqnarray}
\label{eqn-X-Y-A-B} | Ind_{2k} (\{ X_t Y_t  \}, V) - Ind_{2k} (\{
X_t\}, V) - Ind_{2k} (\{ Y_t\}, V)| = \cr \ \cr = | Ind_{2k} (\{ A_t
B_t  \}, L) - Ind_{2k} (\{ A_t\}, L) - Ind_{2k} (\{ B_t\}, L)|.
\end{eqnarray}
Further on,
Remark~\ref{rem-Ind-CZ-matr-bounded-change-under-perturbations}
implies that we can find sufficiently $C^0$-close identity-based
perturbations $\{ A'_t\}$, $\{ B'_t\}$ of $\{ A_t\}$, $\{ B_t\}$
such that
\begin{equation}
\label{eqn-3-end-point-conditions} A'_1 L \cap L = 0,\ B'_1 L\cap L
= 0,\ A'_1 B'_1 L \cap L = 0.
\end{equation}
and
\begin{eqnarray}
\label{eqn-A'-B'-A'B'} | Ind_{2k} (\{ A_t B_t  \}, L) - Ind_{2k} (\{
A_t\}, L) - Ind_{2k} (\{ B_t\}, L)| - \cr \ \cr - | Ind_{2k} (\{
A'_t  B'_t  \}, L) - Ind_{2k} (\{ A'_t\}, L) - Ind_{2k} (\{ B'_t\},
L)|
 < C_1,
\end{eqnarray}
for some $C_1$. On the other hand, since the three identity-based
paths $\{ A'_t\}$, $\{ B'_t\}$, $\{ A'_t B'_t\}$, satisfy the
conditions (\ref{eqn-3-end-point-conditions}), we can apply to them
Theorem~\ref{thm-leray}. Hence there exists $C_2$ such that
$$| Ind_{2k} (\{ A'_t B'_t  \}, L) - Ind_{2k} (\{ A'_t\}, L) -
Ind_{2k} (\{ B'_t\}, L)| < C_2.$$ Combining it with
(\ref{eqn-X-Y-A-B}) and (\ref{eqn-A'-B'-A'B'}) we get that there
exists $C_3$ such that
$$| Ind_{2k} (\{ X_t Y_t  \}, V) - Ind_{2k} (\{ X_t\}, V) -
Ind_{2k} (\{ Y_t\}, V)| < C_3,$$ which finishes the proof.
\end{proof}

\medskip \noindent {\sc Conley-Zehnder index as a quasi-morphism:}
Recall that $2n= {\rm dim}\, M$. Restricting $CZ_{matr}$ to the
identity-based paths in $Sp\, (2n)$ one gets a function on
$\widetilde{Sp\, (2n)}$ that will be still denoted by $CZ_{matr}$.

\begin{prop}[cf. \cite{DeGossons}]
\label{prop-CZ-qmm} The function $CZ_{matr}: \widetilde{Sp\,
(2n)}\to \R$ is a {\it quasi-morphism}. It means that there exists a
constant $C>0$ such that
$$ | CZ_{matr} (ab) - CZ_{matr} (a) - CZ_{matr} (b)|\leq C
\ \ \forall a,b\in \widetilde{Sp\, (2n)}. $$
\end{prop}

\begin{proof}
Represent $a$ and $b$ by identity-based paths $\{ A_t\}$, $\{
B_t\}$, $0\leq t\leq 1$, in $Sp\, (2n)$. Then use
(\ref{eqn-CZ-Ind-4n}) and apply Corollary~\ref{cor-Ind-qmm} for
$k=2n$, $V=\Delta$ to $\{ \widehat{A}_t\}$, $\{ \widehat{B}_t\}$ in
$Sp\, (4n)$.
\end{proof}

\medskip \noindent {\sc Maslov index of symplectic loops:}
The Conley-Zehnder index for identity-based loops in $Sp\, (2n)$ is
called the {\it Maslov index} of a loop. Its original definition,
going back to \cite{Arnold}, is the following: it is the
intersection number of an identity-based loop with the stratified
hypersurface $\Sigma$ whose principal stratum is equipped with a
certain co-orientation. Note that we do not divide the intersection
number by $2$ and thus in our case the Maslov index takes only even
values; for instance, the Maslov index of a counterclockwise
$2\pi$-twist of the standard symplectic $\R^2$ is $2$. We denote the
Maslov index of a loop $\{ B (t)\}$ by $\Maslov (\{ B (t)\})$.

\medskip\noindent {\sc Conley-Zehnder and Maslov indices of periodic
orbits:} The Con\-ley-Zehnder index for periodic orbits is defined
by means of the Conley-Zehnder index for matrix paths as follows.
Given $[\gamma, u]\in \tP_F$, build an identity-based path $\{
A(t)\}$ in $Sp\, (2n)$ as follows: take a symplectic
trivialization of the bundle $u^* (TM)$ over $\D^2$ and use the
trivialization to identify the linearized flow $d_{\gamma (0)}
f_t$, $0\leq t\leq 1$, along $\gamma$ with a symplectic matrix $\{
A (t)\}$. Then the Conley-Zehnder index $CZ_F ([\gamma, u])$ is
defined as
\begin{equation}
\label{eqn-def-CZ-in-Floer-theory} CZ_F ([\gamma, u]) := n -
CZ_{matr}\, (\{ A(t)\}).
\end{equation}
With such a normalization of $CZ_F$ for any sufficiently
$C^2{\hbox{\rm -small}}$ autonomous Morse Hamiltonian $F$, the
Conley-Zehnder index of an element of $\tP_F$, represented by a
pair $[x, u]$ consisting of a critical point $x$ of $F$ (viewed as
a constant path in $M$) and the trivial disk $u$, is equal to the
Morse index of $x$. Note that with such a normalization $CZ_F (Sy)
= CZ_F (y) + 2\int_S c_1(M)$ for every $y \in \tP_F$ and $S \in
H_2^S(M)$.

Similarly, if the time-1 flow generated by $F$ defines a loop in
$\Ham  (M)$ then to each $[\gamma, u]\in \tP_F$ one can associate
its Maslov index. Namely, trivialize the bundle $u^* (TM)$ over
$\D^2$ and identify the linearized flow $\{ d_x f_t \}$ along
$\gamma$ with an identity-based loop of symplectic $2n\times
2n$-matrices. Define the Maslov index $m_F ([\gamma, u])$ as the
Maslov index for the loop of symplectic matrices. Under the action
of $H_2^S (M)$ on $\tP_F$ the Maslov index changes as follows:
\[
m_F (S\cdot [\gamma, u]) = m_F ([\gamma, u]) - 2\int_S c_1(M), \ \ S
\in H_2^S (M).
\]

Let us make the following remark. Assume $\gamma\in \cP_F$ and
assume that a symplectic trivialization of the bundle $\gamma^*
(TM)$ over $\SP^1$ identifies $\{ d_{\gamma (0)} f_t\}$ with an
identity-based path $\{ A(t)\}$ of symplectic matrices. Assume
there is another symplectic trivialization of the same bundle,
coinciding with the first one at $\gamma (0)$, and denote by $\{
B(t)\}$ the identity-based loop of transition matrices from the
first symplectic trivialization to the second one. Use the second
trivialization to identify $\{ d_{\gamma (0)} f_t\}$ with an
identity-based path $\{ A^\prime (t)\}$. Then
\begin{equation}
\label{eqn-CZ-trivializations} CZ_{matr}\, (\{ A^\prime (t)\}) =
CZ_{matr}\, (\{ A (t)\}) + \Maslov (\{ B(t)\}),
\end{equation}
and if $\{ A(t)\}$ is a loop then so is $\{ A^\prime (t)\}$ and
\begin{equation}
\label{eqn-Maslov-trivializations} \Maslov (\{ A^\prime (t)\}) =
\Maslov (\{ A (t)\}) + \Maslov (\{ B(t)\}).
\end{equation}

\subsection{Spectral numbers} \label{sec-spectral-numbers}

Given the algebraic setup as above, the construction of the
Piu\-ni\-khin-Sa\-la\-mon-Schwarz (PSS) isomorphism \cite{PSS}
yields a $\Lambda$-linear isomorphism ({\it PSS-isomorphism})
$\phi_M: QH_* (M) \to HF_* (F,J)$ which preserves the grading and
which is actually a ring isomorphism (the pair-of-pants product
defines a ring structure on $HF_* (F,J)$).

Using the PSS-isomorphism one defines the {\it spectral numbers}
$c (a, F)$, where $0\neq a\in QH_* (M)$, in the usual way
\cite{Oh-spectral}. Namely, the action functional $\cA_F$ defines
a filtration on $C (F)$ which induces a filtration $HF^\alpha_*
(F,J)$, $\alpha\in\R$, on $HF_* (F,J)$, with $HF^\alpha_*
(F,J)\subset HF^\beta_* (F,J)$ as long as $\alpha <\beta$. Then
$$ c(a,F):= \inf\, \{ \alpha\, |\, \phi_M (a) \in HF^\alpha_* (F,J)\}.$$
Such spectral number is finite and well-defined (does not depend on
$J$). Here is a brief account of the relevant properties of spectral
numbers -- for details see \cite{Oh-spectral} (see also
\cite{Viterbo, Oh1, Schwarz, Oh2} for earlier versions of this
theory).

\begin{description}

\item[{\bf (Spectrality)}]\ $c (a, H)\in {\it {spec}}\,
(H)$, where {\it the spectrum ${\it {spec}}\, (H)$ of $H$} is
defined as the set of critical values of the action functional
$\cA_H$, i.e. ${\it {spec}}\, (H) := \cA_H (\tP_H) \subset \R$.

\item[{\bf (Quantum homology shift property)}]\ $c (\lambda a, H)
= c (a, H) + \nu (\lambda)$ for all $\lambda \in \Lambda$, where
$\nu$ is the valuation defined in Section~\ref{subsec-valuation}.

\item[{\bf (Hamiltonian shift property)}]\ $c (a, H+\lambda (t) )
= c (a, H) + \int_0^1 \lambda (t) \ dt$ for any Hamiltonian $H$
and function $\lambda: \SP^1 \to \R$.

\item[{\bf (Monotonicity)}]\ If $H_1\leq H_2$, then $c
(a, H_1)\leq c (a, H_2)$.

\item[{\bf (Lipschitz property)}]\ The map $H\mapsto c (a, H)$ is
Lipschitz on the space of (time-dependent) Hamiltonians $H:
M\times \SP^1\to\R$ with respect to the $C^0$-norm.

\item[{\bf (Symplectic invariance)}]\ $c (a,\phi^*H) = c (a,H)$
for every $\phi \in \Symp_0 (M)$, $H \in C^{\infty} (M)$; more
generally, $\Symp\, (M)$ acts on $H_* (M; \cF)$, and hence on
$QH_* (M)$, and  $c (a,\phi^*H) = c (\phi_* a,H)$ for any
$\phi\in\Symp\, (M)$.

\item[{\bf (Normalization)}]\ $c (a,0) = \nu (a)$ for every $a \in
QH_* (M)$.

\item[{\bf (Homotopy invariance)}]\ $c (a, H_1) = c (a, H_2)$
for any {\it normalized} $H_1, H_2$ generating the same
$\phi\in\tHam (M)$. Thus one can define $c (a,\phi)$ for any
$\phi\in\tHam (M)$ as $c (a,H)$ for any normalized $H$
generating $\phi$.

\item[{\bf (Triangle inequality)}]\
$c (a\ast b, \phi\psi)\leq c (a, \phi) + c (b, \psi)$.
\end{description}

\medskip
\noindent The commutative ring $QH_\bullet (M)$ admits a
$\cK$-bilinear and $\cK$-valued form $\Omega$ on $QH_\bullet (M)$
which associates to a pair of quantum homology classes $a, b\in
QH_\bullet (M)$ the coefficient (belonging to $\cK$) at the class
$[point] = [point]\cdot q^0$ of a point in their quantum product
$a\ast b \in QH_\bullet (M)$ ({\it the Frobenius structure}). Let
$\tau: \cK \to \cF$ be the map sending each series
$\sum_{\theta\in\Gamma} z_\theta s^\theta$, $z_\theta\in \cF$, to
its free term $z_0$. Define a non-degenerate $\cF$-valued
$\cF$-linear pairing on $QH_\bullet (M)$ by
\begin{equation}
\label{def-Pi-pairing} \Pi (a,b) := \tau\Omega(a,b) =
\tau\Omega(a*b,[M])\;.
\end{equation}
Note that $\Pi$ is symmetric and
\begin{equation}
\label{eqn-Pi-Frobenius} \Pi (a*b, c) = \Pi (a, b*c) \ \forall
a,b,c\in QH_\bullet (M).
\end{equation}
 With this notion at hand, we can present another
important property of spectral numbers:

\begin{description}
\item[{\bf (Poincar\'e duality)}]\ $c(b,\phi) = -\inf_{a \in
\Upsilon(b)} c(a,\phi^{-1})$ for all $b \in QH_\bullet (M)
\setminus\{0\}$ and $\phi$. Here $\Upsilon(b)$ denotes the set of
all $a \in QH_\bullet (M)$ with $\Pi(a,b) \neq 0$.
\end{description}

\medskip
\noindent The Poincar\'e duality can be extracted from \cite{PSS}
(cf. \cite{EP-qmm}) -- for a proof see \cite{Ostr-qmm}.

\medskip \noindent
The next property is an immediate consequence of the definitions
(see \cite{EP-qmm} for a discussion in the monotone case):

 \begin{description}
\item[{\bf (Characteristic exponent property)}]\ Given $0\neq
\lambda \in \cF$, $a,b \in QH_* (M)$, $a,b,a+b\neq 0$, and a
(time-dependent) Hamiltonian $H$, one has \break $c(\lambda \cdot
a, H) = c(a,H)$ and $c(a+b,H) \leq \max(c(a,H),c(b,H)) $.
\end{description}

\subsection{Partial symplectic
quasi-states}\label{sec-qst-qmm-non-monotone}

 Given a non-zero idempotent
$a\in QH_{2n} (M)$ and a time-independent Hamiltonian $H: M\to
\R$, define
\begin{equation}
\label{eqn-def-zeta} \zeta (a,H): = \lim_{l \to +\infty}\; \frac{c
(a,lH)}{l}\;.
\end{equation}

\medskip
\noindent When $a$ is fixed, we shall often abbreviate $\zeta(H)$
instead of $\zeta(a,H)$. The limit in the formula
(\ref{eqn-def-zeta}) always exists and thus the functional $\zeta:
C^\infty (M)\to \R$ is well-defined. The functional $\zeta$ on
$C^\infty (M)$ is Lipschitz with respect to the $C^0$-norm $\|H\|
= \max_M |H|$ and therefore extends to a functional $\zeta: C
(M)\to \R$, where $C(M)$ is the space of all continuous functions
on $M$. These facts were proved in \cite{EP-qst} in the case
$a=[M]$ but the proofs actually go through for any non-zero
idempotent $a\in QH_{2n} (M)$.

 Here we will list the properties of $\zeta$ for such an $M$.
Again, these properties were proved in \cite{EP-qst} in the case
$a=[M]$ but the proof goes through for any non-zero idempotent $a\in
QH_{2n} (M)$. The additivity with respect to constants property was
not explicitly listed in \cite{EP-qst} but follows immediately from
the definition of $\zeta$ and the Hamiltonian shift property of
spectral numbers. The triangle inequality follows readily from the
definition of $\zeta$ and from the triangle inequality for the
spectral numbers.

\begin{thm}
\label{thm-partial-qstate-partial-qmm-basic}\ The functional $\zeta:
C (M)\to \R$ satisfies the following properties:

\medskip
\noindent \underline{\it Semi-homogeneity:} $\zeta (\alpha F) =
\alpha \zeta (F)$ for any $F$ and any $\alpha \in \R_{\geq 0}$.

\medskip
\noindent \underline{\it Triangle inequality:} If $F_1, F_2\in
C^\infty (M)$, $\{ F_1, F_2\} =0$ then $\zeta (F_1+F_2)\leq \zeta
(F_1) + \zeta (F_2)$.

\medskip
\noindent \underline{\it Partial additivity and vanishing:} If $F_1,
F_2\in C^\infty (M)$, $\{ F_1, F_2\} =0$ and the support of $F_2$ is
displaceable, then $\zeta (F_1 + F_2) = \zeta (F_1)$; in particular,
if the support of $F\in C(M)$ is displaceable, $\zeta (F) = 0$.

\medskip
\noindent \underline{\it Additivity with respect to constants and
normalization:} $\zeta (F +\alpha) = \zeta (F) + \alpha$ for any $F$
and any $\alpha\in\R$. In particular, $\zeta (1) = 1$.

\medskip
\noindent \underline{\it Monotonicity:} $\zeta (F) \leq \zeta (G)$
for $F \leq G$.

%\medskip
%\noindent\underline{\it Vanishing:} $\zeta (F)= 0$, provided $\supp\, F$ is
%displaceable;

\medskip
\noindent \underline{\it Symplectic invariance:} $\zeta (F) =
\zeta (F\circ f)$ for every symplectic diffeomorphism $f \in
\Symp_0\, (M)$.

\medskip
\noindent \underline{\it Characteristic exponent property:}
$\zeta(a_1+a_2,F) \leq \max(\zeta(a_1,F), \zeta(a_2,F))$ for each
pair of non-zero idempotents $a_1,a_2$ with $a_1*a_2=0$,
$a_1+a_2\neq 0$ (in this case $a_1 + a_2$ is also a non-zero
idempotent), and for all $ F \in C(M)\;.$
\end{thm}

We will call the functional $\zeta: C(M)\to \R$ satisfying all the
properties listed in
Theorem~\ref{thm-partial-qstate-partial-qmm-basic} {\it a partial
symplectic quasi-state}.

\section{Basic properties of (super)heavy
sets}\label{sec-basic-proofs}

In this section we prove parts (i) and (iii) of
Theorem~\ref{thm-intro-heavy}, as well as
Theorem~\ref{thm-diff-idempotents}. We shall use that a partial
symplectic quasi-state $\zeta$ extends by continuity in the
uniform norm to a monotone functional on the space of {\bf
continuous} functions $C(M)$, see
Section~\ref{sec-qst-qmm-non-monotone} above. In particular, one
can use continuous functions instead of the smooth ones in the
definition of (super)heaviness in formulae \eqref{eq-heavy-0} and
\eqref{eq-superheavy-0}.

Assume a partial quasi-state $\zeta$ defined by a non-zero
idempotent is fixed and we consider heaviness and superheaviness
with respect to $\zeta$. We start with the following elementary

\begin{prop}\label{prop-def-heavy}
A closed subset $X \subset M$ is heavy if and only if for every $H
\in C^{\infty}(M)$ with $H|_X=0$, $H \leq 0$ one has $\zeta(H)=0$. A
closed subset $X \subset M$ is superheavy if and only if for every
$H \in C^{\infty}(M)$ with $H|_X=0$, $H \geq 0$ one has
$\zeta(H)=0$.
\end{prop}

\begin{proof}
The  ``only if" parts follow readily from the monotonicity
property of $\zeta$. Let us prove the  ``if" part in the  ``heavy
case" -- the  ``superheavy" case is similar.  Take a function $H$
on $M$ and put
$$F = \min (H - \inf_X H, 0)\;.$$
Note that $F|_X = 0$ and $F \leq 0$. Thus $\zeta(F) =0$ by the
assumption of the proposition. Thus
$$0 = \zeta(F) \leq \zeta(H - \inf_X H)=\zeta(H) - \inf_X H\;,$$
which yields heaviness of $X$.
\end{proof}

\medskip
\noindent The following proposition proves part (i) of
Theorem~\ref{thm-intro-heavy}.

\begin{prop}\label{prop-supheavy-yields-heavy}
Every superheavy set is heavy.
\end{prop}

\begin{proof} Let $X \subset M$ be a superheavy subset.
Assume that $H|_X=0$, $H \leq 0$. By the triangle inequality for
$\zeta$ we have $\zeta(H)+\zeta(-H) \geq 0$. Note that $-H|_X =0$,
$-H \geq 0$. Superheaviness yields $\zeta(-H)=0$, so $\zeta(H)
\geq 0$. But by monotonicity $\zeta(H) \leq 0$. Thus $\zeta(H)=0$
and the claim follows from
Proposition~\ref{prop-def-heavy}.\end{proof}

\medskip
\noindent Superheavy sets have the following user-friendly
property.

\begin{prop}\label{prop-superheavy-use} Let $X \subset M$
be a superheavy set. Then for every $\alpha \in \R$ and $H \in
C^{\infty}(M)$ with $H|_X \equiv \alpha$ one has
$\zeta(H)=\alpha$.
\end{prop}

\begin{proof} Since $\zeta(H+\alpha) = \zeta(H)+\alpha$ it suffices to
prove the proposition for $\alpha = 0$.  Take any function $H$
with $H|_X=0$. Since $X$ is superheavy and, by
Proposition~\ref{prop-supheavy-yields-heavy}, also heavy, we have
$$0=\zeta(-|H|)\leq \zeta(H)\leq \zeta(|H|)=0\;,$$
which yields $\zeta(H)=0$.
\end{proof}

\medskip
\noindent As an immediate consequence we get part (iii) of
Theorem~\ref{thm-intro-heavy}.

\begin{prop}\label{prop-intersec}
Every superheavy set intersects with every heavy set.
\end{prop}

\begin{proof}
Let $X$ be a superheavy set and $Y$ be a heavy set. Assume on the
contrary that $X \cap Y = \emptyset$. Take a function $H \leq 0$
with $H|_Y \equiv 0$ and $H|_X \equiv -1$. Then $\zeta(H)=-1$ by
Proposition~\ref{prop-superheavy-use}. On the other hand,
$\zeta(H)=0$ since $Y$ is heavy, and we get a contradiction.
\end{proof}

\medskip
\noindent Note that two heavy sets do not necessarily intersect
each other: a meridian of $\T^2$ is heavy (see
Corollary~\ref{cor-meridian-heavy} below), while two meridians can
be disjoint.

\medskip \noindent
{\bf Proof of Theorem~\ref{thm-diff-idempotents} (i) and (ii):}
The triangle inequality yields
$$c(a,H)=c(a*[M],0+H)\leq c(a,0) + c([M],H)= \nu (a) + c([M],H).$$
Passing to the partial quasi-states $\zeta (a,H)$ and $\zeta
([M],H)$ we get:
$$\zeta (a,H)= \lim_{k\to +\infty} c(a,kH)/k\leq
$$
$$
\leq \lim_{k\to +\infty} (\nu (a) + c([M],kH))/k =
 \lim_{k\to +\infty} c([M],kH)/k =
\zeta ([M],H).$$ The result now follows from the definition of
heavy and superheavy sets (see Definition~\ref{def-heavy}).\Qed

\medskip \noindent
{\bf Proof of Theorem~\ref{thm-diff-idempotents} (iii):} By the
characteristic exponent property of spectral invariants,
\begin{equation}
\label{eq-quant-semis} \zeta(a,F) \leq \max_{i=1,\ldots,l}
\zeta(e_i,F) \;\; \forall F \in C^{\infty}(M)\;.
\end{equation}

Choose a sequence of functions $G_j \in C^{\infty}(M)$, $j \to
+\infty$, with the following properties: $G_k \leq G_j$ for $k >
j$, $G_j=0$ on $X$, $G_j \leq 0$ and for every function $F\leq 0 $
which vanishes {\it on an open neighborhood} of $X$ there exists
$j$ so that $G_j \leq F$ (existence of such a sequence can be
checked easily). In view of inequality \eqref{eq-quant-semis}, we
have that for every $j$ there exists $i$ so that $\zeta(a,G_j)
\leq \zeta(e_i, G_j)$. Passing, if necessary, to a subsequence
$G_{j_k}, j_k \to +\infty$, we can assume without loss of
generality that $i$ is {\it the same} for all $j$. In view of
heaviness of $X$ with respect to $a$, we have that
$\zeta(a,G_j)=0$. Therefore $\zeta(e_i,G_j) \geq 0$.

Choose any function $F \leq 0$ on $M$ which vanishes {\it on an
open neighborhood} of $X$. Then there exists $j$ large enough so
that $F \geq G_j$. By monotonicity combined with the previous
estimate we have
$$0 \geq \zeta(e_i,F) \geq \zeta(e_i,G_j) \geq 0\;,$$
which yields $\zeta(e_i,F)=0$.

Now let $F$ be any continuous function on $M$ that vanishes {\it
on} $X$. Take a sequence of continuous functions $F_j$, converging
to $F$ in the $C^0$-norm, so that each $F_j$ vanishes on an open
neighborhood of $X$. Then $\zeta (e_i, F_j) = \lim_{j\to +\infty}
\zeta (e_i, F_j) = 0$, because $\zeta (e_i, \cdot)$ is Lipschitz
with respect to the $C^0$-norm. The heaviness of $X$ with respect
to $e_i$ now follows from Proposition~\ref{prop-def-heavy}. This
finishes the proof of the theorem. \Qed

\section{Products  of (super)heavy sets }\label{sec-proofs-prod}

In this section we prove Theorem~\ref{thm-products} on products of
(super)heavy subsets.

\subsection{Product formula for spectral invariants} The proof
of Theorem~\ref{thm-products} is based on the following general
result.

\begin{thm}\label{thm-spectral-sum}
For every pair of time-dependent Hamiltonians $G_1,G_2$ on $M_1$
and $M_2$, and all non-zero $a_1 \in QH_{i_1} (M_1)$, $a_2 \in
QH_{i_2} (M_2)$ we have
$$c(a_1 \otimes a_2, G_1(z_1,t)+G_2(z_2,t)) = c(a_1,G_1)+c(a_1,G_2)\;.$$
Here $G_1(z_1,t)+G_2(z_2,t)$ is a time-dependent Hamiltonian on
$M_1\times M_2$.
\end{thm}

\medskip
\noindent Let us deduce Theorem~\ref{thm-products} from
Theorem~\ref{thm-spectral-sum}.

\medskip
\noindent {\bf Proof of Theorem~\ref{thm-products}:} We show that
the product of superheavy sets is superheavy (the proof for heavy
sets goes without any changes). We denote by $\zeta_1,\zeta_2$ and
$\zeta$ the partial quasi-states on $M_1,M_2$ and $M:=M_1\times
M_2$ associated to the idempotents $a_1,a_2$ and $a_1 \otimes a_2$
respectively. Let $X_i \subset M_i$, $i=1,2$, be a superheavy set.
By Proposition~\ref{prop-def-heavy} it suffices to show that if a
non-negative function $G \in C^{\infty}(M)$ vanishes on some
neighborhood, say $U$, of $X:=X_1 \times X_2$ then $\zeta(G)=0$.
(Since $\zeta$ is Lipschitz with respect to the $C^0$-norm this
would imply that $\zeta(G)=0$ for any non-negative $G\in C(M)$
that vanishes {\it on} $X$). Put $K:=\max_M G$. Choose
neighborhoods $U_i$ of $X_i$ so that $U_1 \times U_2 \subset U$.
Choose non-negative functions $G_i$ on $M_i$ which vanish on $X_i$
and such that $G_i(z)
> K$ for all $z \in M_i \setminus U_i$. Observe that $G \leq G_1
+G_2$. But, in view of Theorem~\ref{thm-spectral-sum} and
superheaviness of $X_i$, we have
$$\zeta(G_1 +G_2) = \zeta_1(G_1)+\zeta_2(G_2) = 0\;.$$
By monotonicity $$0 \leq \zeta(G) \leq \zeta(G_1 +G_2)=0\;,$$ and
thus $\zeta(G)=0$.\qed

\medskip
\noindent It remains to prove Theorem~\ref{thm-spectral-sum}. Note
that the left-hand side of the equality stated in the theorem does
not exceed the right-hand side: this is an immediate consequence
of the triangle inequality for spectral invariants. However, we
were unable to use this observation for proving the theorem. Our
approach is based on a rather lengthy algebraic analysis which
enables us to calculate separately the left and the right-hand
sides  ``on the chain level". A simple inspection of the results
of this calculation yields the desired equality.

\subsection{Decorated $\Z_2$-graded complexes}

 A {\it $\Z_2$-complex} is a $\Z_2$-graded
finite-dimensional vector space $V$ over a field $\cK$ equipped
with a $\cK$-linear differential $\partial: V\to V$ satisfying
$\partial^2=0$ and shifting the grading. {\it A decorated complex}
over $\cK=\cK_\Gamma$ includes the following data:
\begin{itemize}
\item a countable subgroup $\Gamma \subset \R$; \item a
$\Z_2$-graded complex $(V, d)$ over $\cK_{\Gamma}$; \item a
preferred basis $x_1,\ldots,x_n$ of $V$; \item a function
$F:\{x_1,\ldots,x_n\} \to \R$ (called {\it the filter}) which
extends to $V$ by
 $$F(\sum \lambda_j x_j)= \max \{\nu(\lambda_j) +F(x_j)
\;\Big{|}\; \lambda_j \neq 0\},$$ and satisfies $F(dv) < F(v)$ for
all $v \in V \setminus \{0\}$. The convention is that
$F(0)=-\infty$. Here $\nu$ is the valuation defined in Section
\ref{subsec-valuation} above.
\end{itemize}

\medskip
\noindent We shall use the notation $${\bf V}:=
(V,\{x_i\}_{i=1,\ldots,n},F,d,\Gamma)$$ for a decorated complex.

\medskip
\noindent The {\it $\widehat{\otimes}_{\cK}$-tensor product} ${\bf
V}= {\bf V_1}\widehat{\otimes}_{\cK} {\bf V_2}$ of decorated
complexes
$${\bf V_i}= (V_i, \{x^{(i)}_j\}_{j=1,\ldots,n_i},F_i,d_i,\Gamma_i)\;,\;i=1,2$$
is defined as follows. Consider the  space $V = V_1
\widehat{\otimes}_{\cK} V_2$ (see formula \eqref{eq-tensor-hat}
above) with the natural $\Z_2$-grading. Define the differential
$d$ on $V$ by $$d(x\otimes y) = d_1x \otimes y + (-1)^{\deg x}
x\otimes d_2y\;.$$ The preferred basis in $V$ is given by
$\{x_{pq}:= x^{(1)}_p \otimes x^{(2)}_q\}$ and the filter $F$ is
defined by
$$F(x_{pq})=F_1(x^{(1)}_p)+F_2(x^{(2)}_q).$$ Finally, we put
${\bf V}:= (V,\{x_{pq}\},F,d,\Gamma_1+\Gamma_2)\;.$

\medskip
\noindent The ($\Z_2$-graded) homology of decorated complexes are
denoted by $H_* ({\bf V})$ -- they are $\cK$-vector spaces. By the
K\"{u}nneth formula, $H({\bf V_1} \widehat{\otimes}_\cK  {\bf
V_2})= H({\bf V_1}) \widehat{\otimes}_\cK H({\bf V_2})$.

\medskip
\noindent Next we define {\it spectral invariants} associated to a
decorated complex ${\bf V}:= (V,\{x_{pq}\},F,d)\;.$ Namely, for $a
\in H ({\bf V})$ put
$$c(a):=\inf\{F(v)\;|\; a=[v], v \in \text{Ker}\, d\}\;.$$
We shall see below that $c(a) > -\infty$ for each $a \neq 0$.

\medskip
\noindent The purpose of this algebraic digression is to state the
following result:

\medskip
\noindent\begin{thm}\label{thm-algebra} For any two decorated
complexes ${\bf V_1},{\bf V_2}$
$$c(a_1 \otimes a_2) = c(a_1)+c(a_2) \;\; \forall a_1 \in H({\bf V_1}),a_2 \in H({\bf V_2})\;$$
\end{thm}

\subsection{Reduced Floer and Quantum homology}

The $2$-periodicity of the Floer complex and Floer homology
defined by the multiplication by $q$ (see Proposition
\ref{prop-basic-alg-properties-of-Floer-complex-Floer-homology}
above) allows to encode their algebraic structure in a decorated
$\Z_2$-complex. Consider a regular pair $(G,J)$ consisting of a
Hamiltonian function and a compatible almost-complex structure on
$M$ (both, in general, are time-dependent). Let $(C_*(G),d_{G,J})$
be the corresponding Floer complex. Let us associate to it a
$\Z_2$-complex: a $\Z_2$-graded vector space $V_{G}$ over
$\cK_{\Gamma}$, defined as
$$V_{G}:= C_0(G)\oplus C_1 (G),$$ with the obvious $\Z_2$-grading,
and a differential $\partial_{G,J}: V_G\to V_G$, defined as the
direct sum of $d_{G,J}: C_1 (G)\to C_0 (G)$ and $qd_{G,J}: C_0
(G)\to C_1 (G)$. One readily checks that this is indeed a
$\Z_2$-complex because $d_{G,J}: C(G)\to C(G)$ is
$\Lambda_{\Gamma}$-linear. We will call $(V_G,
\partial_{G,J})$ the {\it $\Z_2$-complex associated to $(G,J)$}.

%\begin{prop}
%\label{prop-V-F}
Note that the cycles and the boundaries of $(V_G,
\partial_G)$ having $\Z_2$-degree $i \in \{0,1\}$ in $V_G$ coincide,
respectively, with the cycles and the boundaries having
$\Z$-degree $i$ of $(C(G),d_{G,J})$. Therefore the Floer homology
$HF_i (G,J)$ is isomorphic, as a vector space over $\cK_{\Gamma}$,
to the $i$-th degree component of the homology of the complex
$(V_G, \partial_{G,J})$.
%\end{prop}

The $\Z_2$-complex $(V_G,
\partial_{G,J})$ carries a structure of the decorated complex ${\bf V}_{G,J}$ as
follows. Let $\gamma_i(t), i = 1,\ldots,m$, be the collection of
all contractible $1$-periodic orbits of the Hamiltonian flow
generated by $G$. Choose disc $u_i$ in $M$ spanning $\gamma_i$.
For each $i$ there exists unique integer, say $r_i$, so that the
Conley-Zehnder index of the element $x_i:= q^{r_i}\cdot [\gamma_i,
u_i]$ lies in the set $\{0,1\}$. Clearly, the collection $\{x_i\}$
forms a basis of $V_G$ over $\cK_{\Gamma}$. We shall consider it
as a preferred basis. Note that the preferred basis is unique up
to multiplication of $x_i$'s by elements of the form
$s^{\alpha_i}, \alpha_i \in \Gamma$. Finally, the action
functional associated to $G$ defines a filtration on $V_G$.

The homology of $(V_G, \partial_{G,J})$ can be canonically
identified via the PSS-isomorphism with the object which we call
{\it reduced} quantum homology: $$QH_{red}(M):= QH_0(M)\oplus
QH_1(M)\;.$$ We call this isomorphism {\it the reduced}
PSS-isomorphism and denote it by $\psi_{G,J}$.

Note that we have a natural projection $p: QH_*(M) \to
QH_{red}(M)$ which sends any degree homogeneous element $a$ to
$aq^r$ with $\text{deg}\;a + 2r \in \{0,1\}$. With this notation,
the usual Floer-homological spectral invariant $c(a,G)$ coincides
with the spectral invariant $c(p(a))$ of the decorated complex
${\bf V}_{G,J}$.

\subsection{Proof of Theorem~\ref{thm-spectral-sum}}
 By the Lipschitz property of spectral numbers it
is enough to consider the case when $G_1$ and $G_2$ belong to
regular pairs $(G_i, J_i)$, $i=1,2$. Set $$G (z_1,z_2,t) := G_1
(z_1,t) + G (z_2,t)$$ and $J := J_1\times J_2$. Then $(G,J)$ is
also a regular pair. Put $\Gamma_i = \Gamma(M_i,\omega_i)$. It is
straightforward to see that the decorated complex ${\bf V}_{G,J}$
is the $\widehat{\otimes}_{\cK}$-tensor product of the decorated
complexes ${\bf V}_{G_i,J_i}$ for $i=1,2$.

Put $(M,\omega) = (M_1 \times M_2, \omega_1 \oplus \omega_2)$. An
obvious modification of  the K\"unneth formula for quantum
homology (see e.g. \cite[Exercise 11.1.15]{MS2} for the statement
in the monotone case) yields a natural monomorphism
$$\imath: QH_{i_1}(M_1,\omega_1)  \widehat{\otimes}_{\cK}
QH_{i_2}(M_1,\omega_1)\to QH_{i_1+i_2}(M,\omega)\;.$$ Since in our
setting quantum homologies are $2$-periodic, the collection of
these isomorphisms for all pairs $(i_1,i_2)$ from the set
$\{0,1\}$ induces an isomorphism
$$j: QH_{red}(M_1)\widehat{\otimes}_\cK QH_{red}(M_2)\to
QH_{red}(M)\;.$$ It has the following properties: First, given two
elements $a_1 \in QH_{i_1}(M_1,\omega_1)$ and $a_2 \in
QH_{i_2}(M_2,\omega_2)$ we have that
$$p(a_1) \otimes p(a_2) = p(a_1 \otimes a_2)\;.$$
Second, the following diagram commutes:
\[
\xymatrix{ H(V_{G_1},\partial_{G_1,J_1}) \widehat{\otimes}_\cK
H(V_{G_2},\partial_{G_2,J_2}) \ar[r]^-{k} \ar[d]^{\psi_{G_1,J_1}
\otimes \psi_{G_2,J_2}} & H(V_{G},\partial_{G,J})
\ar[d]^{\psi_{G,J}}
 \\
QH_{red}(M_1)\widehat{\otimes}_\cK QH_{red}(M_2) \ar[r]^-{j} &
QH_{red}(M) }
\]
Here $k$ is the isomorphism coming from the K\"unneth formula for
$\Z_2$-com\-ple\-xes, and $\psi_{G_i,J_i},\psi_{G,J}$ stand for
the reduced PSS-isomorphisms. It follows that the definition of $c
(a_i, G_i)$, $c (a_1\otimes a_2, G)$ matches the definition  of $c
(p(a_i))$ and $c (p(a_1)\otimes p(a_2))$. By
Theorem~\ref{thm-algebra} we get that
$$c(a_1 \otimes a_2, G) = c (p(a_1)\otimes p(a_2))= c
(p(a_1))+c(p(a_2))= c(a_1,G_1) + c(a_2,G_2)\;.$$ This proves
Theorem~\ref{thm-spectral-sum} modulo Theorem~\ref{thm-algebra}.
\qed

\subsection{Proof of algebraic
Theorem~\ref{thm-algebra}
} A decorated complex is called {\it
generic} if $F(x_i)-F(x_j) \notin \Gamma$ for all $i \neq j$
(recall that under our assumptions $\Gamma$, the group of periods
of the symplectic form $\omega$ over $\pi_2(M)$, is a countable
subgroup of $\R$). We start from some auxiliary facts from linear
algebra. Let ${\bf V}:= (V,\{x_i\}_{i=1,\ldots,n},F,d,\Gamma)$ be
a generic decorated complex. We recall once again that for brevity
we write $\cK$ instead of $\cK_\Gamma$ wherever it is clear what
$\Gamma$ is taken.

\medskip
\noindent  An element $x \in V$ is called {\it normalized} if
$$x = x_p + \sum_{i \neq p}\lambda_ix_i\;, \lambda_i \in \cK,\; F(x_p) > \max_{i\neq p} F(\lambda_ix_i)\;.$$
We shall use the notation $x = x_p +o(x_p)$. In generic complexes,
every element $x\neq 0$ can be uniquely written as $x = \lambda
(x_p +o(x_p))$ for some $p=1,\ldots,n$ and $\lambda \in \cK$. A
system of vectors $e_1,\ldots,e_m$ in $V$ is called {\it normal}
if every $e_i$ has the form $e_i = x_{j_i} + o(x_{j_i})$ for $j_i
\in \{1,\ldots,n\}$ and the numbers $j_i$ are pair-wise distinct.

\medskip
\noindent
\begin{lemma}\label{lem-alg-1} Let $e_1,\ldots,e_m$ be a normal system. Then
$$F(\sum_{i=1}^n \lambda_i e_i)= \max_i F(\lambda_i e_i)\;.$$
\end{lemma}

\begin{proof} We prove the result using induction in $m$. For $m=1$ the statement is
obvious. Let's check the induction step $m-1 \to m$. Observe that
it suffices to check that
\begin{equation}\label{eq-alg-1}
F(e_1 + \sum_{i=2}^n \lambda_i e_i )\geq F(e_1)\;.
\end{equation}
Then obviously
$$F(\sum_{i=1}^n \lambda_i e_i)\geq \max_i F(\lambda_i e_i)\;,$$
while the reversed inequality is an immediate consequence of the
definitions.

By the induction step,
$$F(\sum_{i=2}^n \lambda_i e_i)= \max_{i=2,\ldots,n} F(\lambda_i e_i)\;.$$
In view of the genericity, the maximum at the right hand side can
be uniquely written as $F(\lambda_{i_0} x_{i_0})$. Without loss of
generality we shall assume that $e_i = x_{i} + o(x_{i})$ and $i_0
= 2$.

Put
$$v = \sum_{i \geq 2} \lambda_2^{-1}\lambda_i e_i = x_2 + o(x_2)\;.$$
Write
$$e_1 = x_1 + \alpha x_2 + X,\; v = x_2 + \beta x_1 + Y,\;\;$$
where $\alpha,\beta \in \cK$ and $X,Y \in \text{Span}_\cK
(x_3,\ldots,x_n)$. Note that $F(x_1) > F(\alpha x_2)$, $F(x_2) >
F(\beta x_1)$, which yields
\begin{equation}\label{eq-alg-2}
\nu(\alpha) < F(x_1)-F(x_2) < -\nu(\beta)=\nu(\beta^{-1})\;.
\end{equation}
In particular, $\nu(\alpha) < \nu(\beta^{-1})$. Note that
$$e_1+\lambda_2v = (1+\lambda_2\beta)x_1+ (\alpha+\lambda_2)x_2 +Z,\; Z \in \text{Span}_\cK (x_3,\ldots,x_n)\;.$$
Thus
$$F(e_1+\lambda_2v) \geq \max(\nu(1+\lambda_2\beta)+F(x_1), \nu(\alpha+\lambda_2)+F(x_2))\;.$$
If $\nu(1+\lambda_2\beta)\geq 0$ we have $F(e_1+\lambda_2v)\geq
F(x_1)=F(e_1)$ and inequality \eqref{eq-alg-1} follows. Assume
that $\nu(1+\lambda_2\beta)<0 = \nu(1)$. Then
$\nu(\lambda_2\beta)=0 = \nu(\lambda_2)+\nu(\beta)$, and hence
$\nu(\lambda_2)=\nu(\beta^{-1}) \neq \nu(\alpha)$. Thus
$$\nu(\alpha + \lambda_2) \geq \nu(\lambda_2) =-\nu(\beta)\;.$$
Combining this inequality with \eqref{eq-alg-2} we get that
$$F(e_1+\lambda_2v) \geq \nu(\alpha+\lambda_2) + F(x_1) + (F(x_2)-F(x_1)) $$
$$\geq F(x_1) + (\nu(\alpha+\lambda_2) + \nu(\beta)) \geq F(x_1)=F(e_1)\;.$$
This completes the proof of inequality \eqref{eq-alg-1}, and hence
of the lemma.
\end{proof}

\medskip
\noindent It readily follows from the lemma that every normal
system is linearly independent.

\begin{lemma} \label{lem-alg-3}
Every subspace $L \subset V$ has a normal basis.
\end{lemma}

\begin{proof} We use induction over $m = \dim_\cK L$.
The case $m=1$ is obvious, so let us handle the induction step
$m-1 \to m$. It suffices to show the following: Let
$e_1,\ldots,e_{m-1}$ be a normal basis in a subspace $L'$, and let
$v \notin L'$ be any vector. Put $L = \text{Span}_\cK (L' \cup
\{v\})$. Then there exists $e_m \in L$ so that $e_1,\ldots,e_m$ is
a normal basis. Indeed, assume without loss of generality that for
all $i =1,\ldots,m-1$ one has $e_i = x_i + o(x_i)$. Put $W =
\text{Span}_\cK (x_m,\ldots,x_n)$. We claim that $L' \cap W =
\{0\}$. Indeed, otherwise
$$\lambda_1 e_1 +\ldots +\lambda_{m-1} e_{m-1} = \lambda_m x_m +\ldots+\lambda_n x_n$$
where the linear combinations in the right and the left-hand sides
are non-trivial. Apply $F$ to both sides of this equality. By
Lemma~\ref{lem-alg-1}
$$F(\lambda_1 e_1 +\ldots +\lambda_{m-1} e_{m-1})= F(x_p)\; \mod \; \Gamma ,\;\; \text{where}\;\; 1\leq p \leq m-1\;,$$
while
$$F(\lambda_m x_m +\ldots+\lambda_n x_n) = F(x_q) \;\mod\; \Gamma ,\;\; \text{where}\;\; q \geq m\;.$$
This contradicts the genericity of our decorated complex, and the
claim follows. Since $\dim L' + \dim W = \dim V$, we have that $V
= L' \oplus W$. Decompose $v$ as $u+w$ with $u \in L',w \in W$,
and note that $w \in L$. Note that $e_1,\ldots,e_{m-1},w$ are
linearly independent. Furthermore, $w = \lambda(x_p + o(x_p))$ for
some $p \geq m$. Put $e_m = \lambda^{-1}w$. The vectors
$e_1,\ldots,e_m$ form a normal basis in $L$.
\end{proof}

\medskip
\noindent The same proof shows that if $L_1 \subset L_2$ are
subspaces of $V$, every normal basis in $L_1$ extends to a normal
basis in $L_2$.

\medskip
\noindent Now we turn to the analysis of the differential $d$.
Choose a normal basis $g_1,\ldots,g_q$ in $\text{Im}\, d$, and
extend it to a normal basis $g_1,\ldots,g_q,h_1,\ldots,h_p$ in
$\text{Ker}\, d$. Note that each of these $p+q$ vectors has the
form $x_j + o(x_j)$ with distinct $j$. Let us assume without loss
of generality that the remaining $n-p-q$ elements of the preferred
basis in $V$ are $x_1,\ldots,x_q$, and
$$g_i = x_{i+q} + o(x_{i+q}), h_j = x_{j+2q} + o(x_{j+2q})\;.$$
Here we use that, by the dimension theorem, $n=p+2q$. Note that
$$x_1,\ldots,x_q,g_1,\ldots,g_q,h_1,\ldots,h_p$$
is a normal system, and hence a basis in $V$. We call such a basis
a {\it spectral basis} of the decorated complex ${\bf V}$.

Note that $[h_1],\ldots,[h_p]$ is a basis in the homology $H({\bf
V})$. Consider any homology class $a = \sum \lambda_i [h_i]$.
Every element $v \in V$ with $a = [v]$ can be written as $v = \sum
\lambda_i h_i + \sum \alpha_j g_j$. Thus, by
Lemma~\ref{lem-alg-1}, $F(v) \geq \max_i F(\lambda_i h_i)$ and
hence
\begin{equation}\label{eq-alg-10}
c(a) = \max_i F(\lambda_i h_i)\;.
\end{equation}
This proves in particular that the spectral invariants are {\it
finite} provided $a \neq 0$.

\medskip
\noindent For finite sets $A = \{v_1,\ldots,v_s\}$ and $B =
\{w_1,\ldots,w_s\}$ we write $A \otimes B$ for the finite set
$\{v_i \otimes w_j\}$.

\medskip
\noindent Assume now that ${\bf V_1},{\bf V_2} $ are generic
decorated complexes. We say that they are {\it in  general position}
if their tensor product $ {\bf V} = {\bf V_1}
\widehat{\otimes}_{\cK} {\bf V_2}$ is generic. Let
$$B_i = \{x^{(i)}_1,\ldots,x^{(i)}_{q_i},g^{(i)}_1,\ldots,g^{(i)}_{q_i},h^{(i)}_1,\ldots,h^{(i)}_{p_i}\},\; i=1,2$$
be a spectral basis in ${\bf V_i}$. Obviously, $B_1 \otimes B_2$
is a normal basis in $V_1 \widehat{\otimes}_{\cK} V_2$. We shall
denote by $d_1,d_2,d$ the differentials and by $F_1,F_2,F$ the
filters in ${\bf V_1},{\bf V_2}$ and ${\bf V}$ respectively. Put
$G_i =\{g^{(i)}_1,\ldots,g^{(i)}_{q_i}\}$, $H_i =
\{h^{(i)}_1,\ldots,h^{(i)}_{p_i} \}$ and $K = G_1 \otimes B_2 \cup
B_1 \otimes G_2$. Observe that
$$\text{Im}\, d \subset W:=\text{Span} (K)\;.$$ Take any two classes
$$a_i = \sum \lambda^{(i)}_j [h^{(i)}_j] \in H({\bf V}_i)\;, i=1,2.$$
Suppose that $a_1 \otimes a_2= [v]$. Then $v$ is of the form
$$v = \sum_{m,l}\lambda^{(1)}_m \lambda^{(2)}_lh^{(1)}_m\otimes h^{(2)}_l +w$$
where $w$ must lie in $W$. Observe that $(H_1 \otimes H_2) \cap K
=\emptyset$.  By Lemma~\ref{lem-alg-1},
$$F(v) \geq \max_{m,l} F(\lambda^{(1)}_m \lambda^{(2)}_lh^{(1)}_m\otimes h^{(2)}_l)\;,$$
and hence
$$c(a_1 \otimes a_2) = \max_{m,l} F(\lambda^{(1)}_m \lambda^{(2)}_lh^{(1)}_m\otimes h^{(2)}_l)$$
$$= \max_{m,l} \; F_1(\lambda^{(1)}_mh^{(1)}_m) + F_2(\lambda^{(2)}_lh^{(2)}_l)$$
$$= \max_m F_1(\lambda^{(1)}_mh^{(1)}_m) + \max_l F_2(\lambda^{(2)}_lh^{(2)}_l) = c(a_1)+c(a_2)\;.$$
In the last equality we used \eqref{eq-alg-10}. This completes the
proof of Theorem~\ref{thm-algebra} for decorated complexes in
general position.

It remains to remove the general position assumption. This will be
done with the help of the following lemma. We shall work with a
family of decorated complexes $${\bf V}:=
(V,\{x_i\}_{i=1,\ldots,n},F,d,\Gamma)$$ which have exactly the
same data (preferred basis, grading, differential and $\Gamma$)
with the exception of the filter $F$ which will be allowed to vary
in the class of filters. The corresponding spectral invariants
will be denoted by $c(a,F)$.

\vfill\eject

\begin{lemma}\label{lem-alg-perturb}$\;$
\begin{itemize}
\item[{(i)}] If filters $F,F'$ satisfy $F(x_i) \leq F'(x_i)$ for
all $i=1,\ldots,n$, then $c(a,F) \leq c(a,F')$ for all non-zero
classes $a \in H({\bf V})$. \item[{(ii)}] If $F$ is a filter and
$\theta \in \R$, then $F+\theta$ is again a filter and
$c(a,F+\theta)=c(a,F)+\theta$ for all non-zero classes $a \in
H({\bf V})$.
\end{itemize}
\end{lemma}

\medskip
\noindent The proof is obvious and we omit it. It follows that for
any two filters $F,F'$
$$|c(a,F)-c(a,F')| \leq ||F-F'||_{C^0}\;\;\forall a \in
H({\bf V})\setminus \{0\}\;.$$

\medskip
\noindent Assume now that ${\bf V_1},{\bf V_2} $ are decorated
complexes. Denote by $F_1,F_2$ their filters. Fix $\epsilon > 0$.
By a small perturbation of the filters we get new filters, $F'_1$
and $F'_2$, on our complexes so that the complexes become generic
and in general position, and furthermore
$$||F_1-F'_1||_{C^0} \leq \epsilon\;, ||F_2-F'_2||_{C^0} \leq
\epsilon\;.$$ Given homology classes $a_i \in H({\bf V}_i)$ we
have
$$|c(a_1,F_1) + c(a_2,F_2) -c(a_1 \otimes a_2,F_1+F_2)| \leq$$ $$
  |c(a_1,F'_1) + c(a_2,F'_2) -c(a_1 \otimes
  a_2,F'_1+F'_2)|+4\epsilon = 4\epsilon\;.$$
  Here we used that Theorem~\ref{thm-algebra} is already proved
  for generic complexes in general position. Since $\epsilon >0$ is
  arbitrary, we get that
$$c(a_1,F_1) + c(a_2,F_2) -c(a_1 \otimes a_2,F_1+F_2)=0\;,$$
which completes the proof of Theorem~\ref{thm-algebra} in full
generality. \qed

\section{Stable non-displaceability of heavy
sets}\label{sec-stab-nondispl-heavy-sets}

In this section we prove part (ii) of Theorem~\ref{thm-intro-heavy}.

\begin{prop}\label{prop-heavy-stab-disp}
Every heavy subset is stably non-displaceable.
\end{prop}

\medskip
\noindent For the proof we shall need the following auxiliary
statement. Given $R > 0$, consider the torus $\T^2_R$ obtained as
the quotient of the cylinder $T^* \SP^1 = \R (r) \times \SP^1\,
(\theta\ {\rm mod}\ 1)$ by the shift $(r,\theta)\mapsto
(r+R,\theta)$. For $\alpha >0$ define the function $F_{\alpha}
(r,\theta):= \alpha f (r)$ on $\T^2_R$, where $f (r)$ is any
$R$-periodic function having only two non-degenerate critical
points on $[0,R]$: a maximum point at $r=0$ with $f(0)=1$, and a
minimum point at $r=R/2$, $f (R/2) =: - \beta < 0$. We denote by
$[T]$ the fundamental class of $\T^2_R$. We work with the
symplectic form $dr \wedge d\theta$ on $\T^2_R$.

\begin{lemma}\label{lem-torus} $c([T],F_{\alpha})=\alpha$.

\end{lemma}

\begin{proof} Note that the contractible closed orbits of period $1$ of the
Hamiltonian flow generated by $F_{\alpha}$ are fixed points forming
circles $S_+ = \{ r = 0\}$ and $S_- = \{ r = R/2 \}$. The actions of
the fixed points on $S_{\pm}$ equal respectively to $\alpha$ and $-
\alpha \beta$, and thus the spectral invariants of $F_{\alpha}$ lie
in the set $\{\alpha,-\alpha\beta\}$. Recall from \cite{Schwarz}
that $c([T],F_{\alpha}) > c([\text{point}],F_{\alpha})$. Thus
$c([T],F_{\alpha})=\alpha$.
\end{proof}

\medskip
\noindent
\begin{lemma}\label{lem-max}
Let $H \in C^{\infty}(M)$ so that $H^{-1}(\max H)$ is
displaceable. Then $\zeta(H) < \max H$.
\end{lemma}
\begin{proof}
Choose $\epsilon > 0$  so that the set $$H^{-1}((\max H -\epsilon,
\max H])$$ is displaceable. Choose a real-valued cut-off function
$\rho: \R \to [0,1]$ which equals $1$ near $\max H$ and which is
supported in $(\max H -\epsilon, \max H+\epsilon)$. Thus $\rho
(H)$ is supported in $H^{-1}((\max H -\epsilon; \max H])$ and
$\zeta (\rho (H)) = 0$. Since $H$ and $\rho (H)$ Poisson-commute,
the vanishing and the monotonicity axioms yield
$$\zeta(H) =  \zeta (\rho(H)) +
\zeta (H -\rho(H)) \leq \max (H - \rho(H)) < \max
H\;.$$\end{proof}

\medskip
\noindent {\bf Proof of Proposition~\ref{prop-heavy-stab-disp}:} It
suffices to show that for every $R > 0$ the set
$$Y:=X \times \{r=0\} \subset M':= M \times \T^2_R$$
is non-displaceable. Assume on the contrary that $Y$ is
displaceable. Choose a function $H$ on $M$ with $H \leq 0$,
$H^{-1}(0)= X$. Put
$$H'=H+ F_1 = H + f(r): M' \to \R.$$
Assume that the partial quasi-state $\zeta$ on $M$ is associated
to some non-zero idempotent $a \in QH_* (M)$ by means of
(\ref{eqn-zeta-part-qstate}). Denote by $\zeta'$ the quasi-state
on $M'$ associated to $a \otimes T$. Note that
$$Y=(H')^{-1}(\max H')\;,\;\;\; \text{where}\;\;\;\max H' =1\;,$$
while Theorem~\ref{thm-spectral-sum} and Lemma~\ref{lem-torus}
imply that
$$\zeta'(H') = \zeta(H)+1\;.$$
By Lemma~\ref{lem-max} $\zeta'(H') < 1$ and so $\zeta(H) <0$. In
view of Proposition~\ref{prop-def-heavy}, we get a contradiction
with the heaviness of $X$. \qed

\medskip
\noindent Lemma~\ref{lem-torus} also yields a simple proof of the
following result which also follows from
Corollary~\ref{thm-Albers-elem-implies-heavy}:

\begin{cor}
\label{cor-meridian-heavy} Any meridian of $\T^2$ is heavy (with
respect to the fundamental class $[T]$).

\end{cor}

\begin{proof}
In the notation as above identify $\T^2$ with $\T^2_1$ for $R=1$.
Since any two meridians of $\T^2$ can be mapped into each other by
a symplectic isotopy and since such an isotopy preserves
heaviness, it suffices to prove that the meridian $S:=S_+ = \{ r =
0\}$ (see the proof of Lemma~\ref{lem-torus}) is heavy.

Let $H: \T^2\to \R$ be a Hamiltonian and let us show that $\zeta
(H) \geq \inf_S H$, where $\zeta$ is defined using $[T]$. Shifting
$H$, if necessary, by a constant, we may assume without loss of
generality that $\inf_S H =1$. Pick $f = f(r): \T^2\to \R$ as in
the definition of $F_\alpha$ so that $F_1 = f\leq H$ on $\T^2$
(note that $f$ equals $1$ on $S$). Then Lemma~\ref{lem-torus}
yields
$$
\zeta (H) \geq \zeta (F_1) = 1 = \inf_S H.
$$
\end{proof}

\section{Analyzing stable stems}\label{sec-anal-stable-stems}

\medskip
\noindent {\bf Proof of Theorem~\ref{thm-stem}:} Assume that
${\mathbb A}$ is a Poisson-commutative subspace of $C^{\infty}(M)$,
$\Phi: M \to {\mathbb A}^*$ its moment map with the image $\Delta$,
and let $X=\Phi^{-1}(p)$ be a stable stem of ${\mathbb A}$.

Take any function $H\in C^{\infty}({\mathbb A}^*)$ with $H \geq 0$
and $H(p)=0$. We claim that $\zeta(\Phi^*H)=0$. By an arbitrarily
small $C^0$-perturbation of $H$ we can assume that $H=0$ in a
small neighborhood, say $U$, of $p$. Choose an open covering
$U_0,U_1,\ldots,U_N$ of $\Delta$ so that $U_0=U$, and all
$\Phi^{-1}(U_i)$ are stably displaceable for $i \geq 1$ (it exists
by the definition of a stem). Let $\rho_i: \Delta\to\R$,
$i=0,\ldots,N$, be a partition of unity subordinated to the
covering $\{ U_i\}$.

Take the two-torus $\T^2_R$ as in
Section~\ref{sec-stab-nondispl-heavy-sets}.
 Choose $R>0$ large
enough so that $\Phi^{-1} (U_i) \times \{r=\text{const}\}$ is
displaceable in $M \times \T^2_R$ for all $i \geq 1$. Choose now a
sufficiently fine covering $V_j, j=1,\ldots,K$, of the torus
$\T^2_R$ by sufficiently thin annuli $\{ |r -r_j| < \delta\}$ so
that the sets $\Phi^{-1} (U_i) \times V_j $ are displaceable in $M
\times \T^2_R$ for all $i \geq 1$ and all $j$. Let $\varrho_j =
\varrho_j(r)$, $j=1,\ldots,K$, be a partition of unity
subordinated to the covering $\{V_j\}$.

Denote by $\zeta'$ the partial quasi-state corresponding to $a
\otimes T$. Put $F(r,\theta)= \cos ({2\pi r}/{R})$. Write
$$\Phi^*H + F= \sum_{i=0}^N\sum_{j=1}^K (\Phi^*H + F)\cdot
\Phi^*\rho_i\cdot \varrho_j = $$ $$\Phi^*(H\rho_0) +
F\cdot\Phi^*\rho_0 + \sum_{i=1}^N\sum_{j=1}^K (\Phi^*H + F)\cdot
\Phi^*\rho_i\cdot \varrho_j\;.$$

Note that $H\rho_0 =0$ and $F\cdot\Phi^*\rho_0 \leq 1$. Applying
partial quasi-additivity and monotonicity we get that
$$\zeta'(\Phi^*H + F) = \zeta'(F\cdot \Phi^*\rho_0) \leq 1.$$
By Lemma~\ref{lem-torus} and the product formula
(Theorem~\ref{thm-spectral-sum} above) we have
$$\zeta'(\Phi^*H + F)=\zeta(\Phi^*H) +1 \leq 1$$
and hence $\zeta(\Phi^*H) \leq 0$. On the other hand,
$\zeta(\Phi^*H) \geq 0$ since $H \geq 0$. Thus $\zeta(\Phi^*H) =0$
and the claim follows.

Further, given any function $G$ on $M$ with $G \geq 0$ and $G|_X
=0$, one can find a function $H$ on ${\mathbb A}^*$ with $H(p)=0 $ so that
$G \leq \Phi^*H$. By monotonicity and the claim above
$$0\leq \zeta(G) \leq \zeta(\Phi^*H) =0\;,$$
and hence $\zeta(G)=0$. Thus $X$ is superheavy. \qed

\section{Monotone Lagrangian submanifolds }\label{sec-Lagr-proofs}

The main tool of proving (super)heaviness of monotone Lagrangian
submanifolds satisfying the Albers condition is the spectral
estimate in Proposition~\ref{pro-BC}(iii) below, which originated
in the work by Albers \cite{Albers}. Later on Biran and Cornea
pointed out a mistake in \cite{Albers}, and suggested a correction
together with a far reaching generalization in \cite{BC}. Let us
mention that the original Albers estimate was used in the first
version of the present paper. We thank Biran and Cornea for
informing us about the mistake, explaining to us their approach
and helping us to correct a number of our results affected by this
mistake.

The main ingredient of Biran-Cornea techniques which is needed for
our purposes is the following result. Let $(M,\omega)$ be a closed
monotone symplectic manifolds with $[\omega] = \kappa \cdot
c_1(M)$, $\kappa >0$. Write $N$ for the minimal Chern number of
$(M,\omega)$. Let $L^n \subset M^{2n}$ be a closed monotone
Lagrangian submanifold with the minimal Maslov number $N_L \geq
2$.

We shall treat slightly differently the cases when $N_L$ is even
and odd. Let us mention that for orientable $L$, $N_L$ is
automatically even. Thus, due to our convention, when $N_L$ is odd
we work with the basic field $\cF = \Z_2$. Let $\Gamma = \kappa N
\cdot \Z$ be the group of periods of $M$. Recall that the quantum
ring has the form $QH_*(M) = H_*(M;\cF) \otimes_{\cF} \Lambda$,
where the Novikov ring $\Lambda$ is defined as $\Lambda =
\cK_{\Gamma}[q,q^{-1}]\;.$ Put $\Gamma' = (\kappa N/2) \cdot \Z$.
Consider an extended Novikov ring $\Lambda':=
\cK_{\Gamma'}[q^{\frac12},q^{-\frac12}]$. Define now $QH'_*(M)$ as
$QH_*(M)$ if $N_L$ is even, and as $H_*(M,\Z_2) \otimes_{\Z_2}
\Lambda'$ if $N_L$ is odd. In the latter case $QH'_*(M)$ is an
extension of $QH_*(M)$, and we shall consider without special
mentioning $QH_*(M)$, $\Lambda$, $\cK_\Gamma$ as subrings of
$QH'_* (M)$, $\Lambda'$, $\cK_{\Gamma'}$. The grading of
$QH_*'(M)$ is determined by the condition $\text{deg}\;
q^{\frac12}=1$. As before, we shall use notation
$QH'_{\bullet}(M)$, where $\bullet = \text{``even"}$ when $\cF =
\C$ and $\bullet = *$ when $\cF=\Z_2$.

Note that the spectral invariants (and hence partial symplectic
quasi-states) are well-defined over the extended ring, and
furthermore, their values and properties, by tautological reasons,
do not alter under such an extension (cf. a discussion in
\cite{BC}, Section 5.4). Put $w:= s^{\kappa N_L/2}q^{N_L/2}$.
Recall that $j$ stands for the natural morphism $H_\bullet
(L;\cF)\to H_\bullet (M;\cF)$.

\medskip
\noindent
\begin{prop}\label{pro-BC} Assume that $k > n+1-N_L$. If $\cF=\C$ assume in addition that $k$ is even.
Then there exists a canonical homomorphism $j^q: H_k (L;\cF) \to
QH'_k (M)$  with the following properties\footnote{The letter
``$q$" in $j^q$ stands for {\it quantum}.}:
\begin{itemize}
\item[{(i)}] $j^q(x) = j(x)+w^{-1}y$, where $y$ is a polynomial in
$w^{-1}$ with coefficients in $H_\bullet (M;\cF)$; \item[{(ii)}]
$j^q([L]) = j([L])$; \item[{(iii)}] If $j^q (x)\neq 0$ then $c(j^q
(x),H) \leq \sup_L H$ for every $H \in C^{\infty}( M)$.
\end{itemize}
\end{prop}

\medskip
\noindent In particular, if $S$ is an Albers element of $L$, we
have $j^q(S)= j(S)+ O(w^{-1}) \neq 0$.

This proposition was proved by Biran and Cornea in \cite{BC} in
the case $\cF=\Z_2$: The map $j^q$ is essentially the map $i_L$
appearing in Theorem A(iii) in \cite{BC}.
Proposition~\ref{pro-BC}(i) above is a combination of Theorem
A(iii) and Proposition 4.5.1(i) in \cite{BC}. Our variable $w$
corresponds to the variable $t^{-1}$ in \cite{BC}, while our
$s^{N\kappa}q^N$ corresponds to the variable $s^{-1}$ in  Section
2.1.2 of \cite{BC}. After such an adjustment of the notation, the
formula $w:= s^{\kappa N_L/2}q^{N_L/2}$ above can be extracted
from Section 2.1.2 of \cite{BC}. For Proposition~\ref{pro-BC}(ii)
above see Remark 5.3.2.a in \cite{BC}.
Proposition~\ref{pro-BC}(iii) above follows from Lemma 5.3.1(ii)
in \cite{BC}. Finally, let us repeat the disclaimer made in
Section~\ref{subsec-mon-lagr}: we take for granted that
Proposition~\ref{pro-BC} remains valid for the case $\cF=\C$.

Let us pass to the proofs of our results on (super)-heaviness of
monotone Lagrangian submanifolds. We start with the following
remark. Let $S$ be an Albers element of $L$. The Poincar\'e
duality  property of spectral invariants (see
Section~\ref{sec-spectral-numbers} above) extends verbatim to the
case of the ring $QH'(M)$ with the following modification: When
$N_L$ is odd, the pairing $\Pi$ introduced in
Section~\ref{sec-spectral-numbers} extends in the obvious way to a
non-degenerate $\cF$-valued pairing on $QH'_\bullet (M)$ which we
still denote by $\Pi$. Applying Poincar\'e duality and
substituting $H:=-F$ into Proposition~\ref{pro-BC} (iii) above we
get that for every function $ F \in C^{\infty}(M)$
$$c(T,F) \geq \inf_L F\; \;\; \forall T \in QH'_\bullet (M)\;\;
\text{with}\;\; \Pi (T,j^q(S))\neq 0.$$ In particular, given a
non-zero idempotent $a\in QH'_\bullet (M)$ and a class $b\in
QH'_\bullet (M)$, so that $\Pi (a*b,j^q(S))\neq 0$, we get, using
the normalization property of spectral invariants, that
\begin{equation}\label{eq-3-1}
c(a,F) +\nu (b)  \geq  c(a*b, F) \geq \inf_L F\; \;\; \forall F
\in C^{\infty}(M)\;.
\end{equation}
Therefore, applying \eqref{eq-3-1} to $kF$ for $k\in\N$, dividing
by $k$ and passing to the limit as $k\to +\infty$, we get that for
the partial quasi-state $\zeta$, defined by $a$,
\[
\zeta(F)\geq \inf_L F\; \;\; \forall F \in C^{\infty}(M),
\]
meaning that $L$ is heavy with respect to $a$.

\medskip
\noindent {\bf Proof of
Theorem~\ref{thm-Albers-elem-implies-heavy}:} Let $S$ be an Albers
element of $L$. Let $T \in H_\bullet (M;\cF)$ be any singular
homology class such that $T \circ j(S) \neq 0$. Thus, applying
Proposition~\ref{pro-BC} (i) we see that $\Pi([M]*T,j^q(S)) =
\Pi(T,j^q(S)) \neq 0$, and hence inequality \eqref{eq-3-1},
applied to $a=[M], b=T$, yields that $L$ is heavy with respect to
$[M]$. \qed

\medskip
Let us pass to the proof of Theorem~\ref{thm-lagr-semismp-a} on the
effect of semi-simplicity of the quantum homology. It readily
follows from the next more general statement. Let $L_1,\ldots,L_m$
be Lagrangian submanifolds satisfying the Albers condition. Let
$S_i$ be any Albers element of $L_i$. Denote by $Z_i = j^q(S_i) \in
QH'_\bullet (M)$ its image under the inclusion morphism from
Proposition~\ref{pro-BC} above.

\medskip
\noindent
\begin{thm}\label{thm-semis-gen} Given such $L_1,\ldots,L_m$ and $Z_1,\ldots,Z_m$,
assume, in addition, that $QH_{2n}(M)$ is semi-simple and the
Lagrangian submanifolds $L_1,\ldots,L_m$ are pair-wise disjoint.
Then the classes $Z_1,\ldots,Z_m$ are linearly independent over
$\cK_{\Gamma'}$.
\end{thm}
%XXXXXXXXXXXXXXXXXXXXXXXXXXXXXXXXXXXXXXXXXXXXXXXXXXXXXXXXXXXXXXXXXXXXXXXXX
\begin{proof} Arguing by contradiction, assume that
\begin{equation}
\label{eqn-j-L-1-lin-combin} Z_1 = \alpha_2 Z_2 + \ldots +
\alpha_m Z_m \end{equation} for some
$\alpha_2,\ldots,\alpha_m\in\cK_{\Gamma'}$. Since $QH_{2n} (M)$ is
semi-simple, it decomposes into a direct sum of fields with
unities $e_1,\ldots,e_d$. Since the pairing $\Pi$ (on $QH'_\bullet
(M; \cF)$) is non-degenerate, there exists $T\in QH'_\bullet (M;
\cF)$ such that
\begin{equation}
\label{eqn-exists-T} \Pi (T, Z_1)\neq 0.
\end{equation}
Let us write $T$ as
\begin{equation}
\label{eqn-T-as-a-lin-comb} T = [M] * T = \sum_{i=1}^d e_i*T.
\end{equation}
Equations \eqref{eqn-exists-T}, \eqref{eqn-T-as-a-lin-comb} imply
that there exists $l \in [1,d]$ such that
\begin{equation}
\label{eqn-exists-l}\Pi (e_l*T, Z_1)\neq 0\;.
\end{equation}
Then \eqref{eqn-j-L-1-lin-combin} implies that there exists $r \in
[2,m]$ such that
$$\Pi (e_l * T,\alpha_r Z_r)\neq 0.$$
Using \eqref{eqn-Pi-Frobenius} (for $\Pi$ on $QH'_\bullet (M;
\cF)$) we can rewrite the last equation as
\begin{equation}
\label{eqn-exists-r} \Pi (e_l* \alpha_r T, Z_r)\neq 0.
\end{equation}
Applying now formula \eqref{eq-3-1} for $S= Z_1\in H_\bullet
(L_1;\cF)$, $a=e_l$, $b =T$, and also for $S= Z_r\in H_\bullet
(L_r;\cF)$, $a=e_l$, $b =\alpha_r T$, we conclude that both $L_1$
and $L_r$ are heavy with respect to $e_l$. Thus they are
superheavy with respect to $e_l$, because $e_l$ is the unity in a
field factor of $QH_{2n} (M)$ (see
Section~\ref{subsec-effect-semisimp}). Hence they must intersect
-- in contradiction to the assumption of the theorem. This
finishes the proof of the first part of the theorem.
\end{proof}

\medskip
\noindent {\bf Proof of Theorem~\ref{thm-lagr-semismp-a}(a):}
Assume that $L_1,\ldots,L_m$ are pair-wise disjoint Lagrangian
submanifolds satisfying the condition (a) from the formulation of
the theorem. Denote by $N_i$ the minimal Maslov number of $L_i$.
Since $N_i > n+1$, the class of a point from $H_0(L_i;\cF)$ is an
Albers element for $L_i$. Let $Z_i \in QH'_0(M)$ be its image
under the Biran-Cornea inclusion morphism associated to $L_i$.
Note that by Proposition~\ref{pro-BC}(i) $Z_i = p + a_i w_i^{-1}$,
where $w_i= s^{\kappa N_i/2}q^{N_i/2}$, $a_i \in H_{N_i}(M;\cF)$
and $p\in H_0 (M; \cF)$ is the homology class of a point. Observe
that $\deg w_i = N_i
> n+1$, and hence the expression for $Z_i$ cannot contain terms in
$w_i^{-1}$ of order two and higher, since $H_{kN_i}(M;\cF)=0$ for
$k \geq 2$.

Recall now that all $N_i$'s lie in some set $Y$ of positive
integers. Let $W \subset QH'_0(M)$ be the span over
$\cK_{\Gamma'}$ of $$H_0(M;\cF) \oplus \bigoplus_{E \in Y}
s^{-\kappa E/2}q^{-E/2} \cdot H_E(M;\cF) \;.$$ Note that
$$\dim_{\cK_{\Gamma'}} W = \beta_Y(M)+1 < m\;.$$
Thus the elements $Z_i$, $i=1,\ldots,m$, are linearly dependent
over $\cK_{\Gamma'}$. By Theorem~\ref{thm-semis-gen}, $QH_{2n}(M)$
is not semi-simple. \qed

\medskip
\noindent {\bf Proof of Theorem~\ref{thm-lagr-semismp-a}(b):} Assume
that $L_1,\ldots,L_m$ are pair-wise disjoint homologically
non-trivial Lagrangian submanifolds. By Proposition \ref{pro-BC}(ii)
$j^q([L_i]) = j([L_i])$ for every $i=1,\ldots,m$. Since the classes
$j([L_i])$ are linearly dependent, Theorem~\ref{thm-semis-gen}
implies that $QH_{2n}(M)$ is not semi-simple. \qed

%(Formally speaking, the morphisms $j^q$ and $j$ depend on the
%Lagrangian submanifold $L_i$; for the sake of brevity, we do not
%indicate this dependence in the notation).

\medskip
\noindent {\bf Proof of Theorem~\ref{thm-monot-superh}:} Combining
Proposition~\ref{pro-BC} (ii) and (iii) we get that for any $H \in
C^{\infty}(M)$
$$c(j([L]),H) \leq \sup_L H \;\;\forall H \in C^{\infty}(M)\;.$$
By the hypothesis of the theorem, we can write $j([L])*b =a$ for
some $b$. Then
$$c(a,H) = c(j([L])*b,H) \leq c(j([L]),H) + c(b,0)\;.$$
Thus
$$c(a,H) \leq \sup_L H + c(b,0)\;.$$
Applying this inequality to $E\cdot H$ with $E >0$, dividing by
$E$ and passing to the limit as $E \to +\infty$ we get that
$\zeta(H) \leq \sup_L H$ for all $H$. Thus $L$ is superheavy. \qed

\medskip
\noindent
\begin{rem}\label{rem-BC-gen} {\rm The results above admit the following
generalizations in the framework of the Biran-Cornea theory. The
main object of this theory is the quantum homology ring $QH_* (L)$
of a monotone Lagrangian submanifold, which is isomorphic to the
Lagrangian Floer homology $HF_* (L,L)$ of $L$ up to a shift of the
grading.
\begin{itemize}

\item[{(i)}] If $QH_* (L)$ does not vanish then $L$ is heavy (see
Remark 1.2.9b in \cite{BC}). In fact, it follows from \cite{BC}
that if $L$ satisfies the Albers condition, $QH_* (L)$ does not
vanish.

\item[{(ii)}] The map $j^q$ of the Proposition~\ref{pro-BC} above
is a footprint of the quantum inclusion map $i_L: QH_* (L) \to QH'_*
(M)$ constructed in \cite{BC}. The analogue of the action estimate
in item (iii) of the proposition is obtained in \cite{BC} for the
classes $i_L(x)$ for elements $x\in QH_* (L)$ of a certain special
form, yielding the following generalization of
Theorem~\ref{thm-monot-superh}: for these special classes $x \in
QH_* (L)$ the condition that the class $i_L(x)$ does not vanish and
divides a non-trivial idempotent $a$ implies that $L$ is superheavy
with respect to $a$. This enables, for instance, to generalize
Example~\ref{exam-quadric} on Lagrangian spheres in quadrics above
to the case when $\dim L$ is odd. \item[{(iii)}] In \cite{BC} one
can find another action estimate which comes from the $QH_*
(M)$-module structure on $QH_* (L)$, which yields more results on
(super)heaviness of monotone Lagrangian submanifolds.
\end{itemize}}
\end{rem}

\medskip \noindent
{\bf Proof of Proposition~\ref{prop-T2n-blown-up-at-one-point}:}
The quantum homology $QH_{2n} (M)$ splits as an algebra over $\cK$
into a direct sum of two algebras one of which is a field. This
was proved by McDuff  for the field $\cF = \C$ (see
\cite{McDuff-uniruled} and \cite[Section 7]{EP-toric-proc}), but
the proof goes through for the case $\cF=\Z_2$ as well. Denote the
unity of the field by $a$. It is a non-zero idempotent in $QH_{2n}
(M)$. As we already pointed out in
Remark~\ref{rem-semisimplicity-genuine-qstate}, such an idempotent
$a$ defines a genuine symplectic quasi-state and therefore the
classes of heavy and superheavy sets with respect to $a$ coincide.

By Theorem~\ref{thm-intro-heavy}, the Lagrangian torus $L\subset
M$ cannot be superheavy with respect to $a$, since it can be
displaced from itself by a symplectic (non-Hamiltonian) isotopy.
Indeed, take an obvious symplectic isotopy $\phi_t$ of $\T^{2n}$
that displaces $L$ (a parallel shift) and compose it with a
Hamiltonian isotopy $\psi_t$ so that for every $t$ we have that
$\psi_t$ is constant on $\phi_t (L)$ and $\psi_t \phi_t$ is
identity on the ball where the blow up of $\T^{2n}$ was performed.
Clearly, the resulting symplectic isotopy $\psi_t \phi_t$ extends
to a symplectic isotopy of $M$ that displaces $L$.

On the other hand, $N_L\geq 2$ because in this case $N_L = 2N$,
where $N\geq 1$ is the minimal Chern number of $M$. Finally, note
that $L$ represents a non-trivial homology class in $H_{n} (M;
\Z_2)$. Therefore we can apply
Theorem~\ref{thm-Albers-elem-implies-heavy} and get that $L$ is
heavy with respect to $[M]$.\Qed

\section{Rigidity of special fibers of Hamiltonian actions}
\label{sect-pf-thm-part-qstate}

 In this section we prove
Theorem~\ref{thm-main-nondispl-toric}. Denote the special fiber of
$\Phi$ by $L:=\Phi^{-1} ({p}_{spec})$.

\medskip
\noindent {\sc Reduction to the case of $\T^1$-actions:} First, we
claim that it is enough to prove the theorem for Hamiltonian
$\T^1$-actions and the general case will follow from it. Indeed,
assume this is proved. The superheaviness of the special fiber
immediately yields that for any function $\bar{H}: \R \to \R$
\begin{equation}\label{eq-spec-fiber-T1-action}
\zeta(\Phi^* \bar{H}) = \bar{H} ({p}_{spec}),
\end{equation}
where $\Phi: M\to\R$ is the moment map of the $\T^1$-action.

Let us turn to the multi-dimensional situation and let $\Phi: M
\to \R^k$ be the normalized moment map of a Hamiltonian $\T^k$-action on $M$.
For a ${\bf v} \in \R^k$ denote by $\Phi_{\bf v} ({\bf x}) =
\langle {\bf v},\Phi({\bf x})\rangle$, where $\langle\cdot,\cdot
\rangle$ is the standard Euclidean inner product on $\R^k$. Note
that if ${\bf v} \in \Z^k$ the function $\Phi_{\bf v}$ is the normalized moment map of a
Hamiltonian circle action and its special value
is $\langle {\bf v},
{p}_{spec}\rangle$. Thus by \eqref{eq-spec-fiber-T1-action}
\begin{equation}\label{eq-spec-fiber-T1-action-2}
\zeta(\Phi_{\bf v}^* {K}) = {K} (\langle {\bf v},
{p}_{spec}\rangle) \;\forall {K} \in C^{\infty}(\R)\;.
\end{equation}
By homogeneity of $\zeta$, equality \eqref{eq-spec-fiber-T1-action-2}
holds for all ${\bf v}
\in \Q^k$, and by continuity for all ${\bf v} \in \R^k$.

Observe that for each pair of smooth functions ${P}, {Q}\in
C^\infty (\R)$ and for each pair of vectors ${\bf v},{\bf w} \in
\R^k$ the functions $\Phi_{\bf v}^*P$ and $\Phi_{\bf w}^*Q$
Poisson-commute on $M$. Thus the triangle inequality for the
spectral numbers (see Section~\ref{sec-spectral-numbers}) yields
\begin{equation}\label{eq-spec-fiber-T1-action-3}
\zeta(\Phi_{\bf v}^*{P} +\Phi_{\bf w}^* {Q})\leq \zeta(\Phi_{\bf
v} ^* {P}) +\zeta(\Phi_{\bf w}^* {Q})\;.
\end{equation}
Since $M$ is compact, it suffices to assume that the function
$\bar{H} \in C^{\infty}(\R^k)$ on $\R^k$ is compactly supported.
By the inverse Fourier transform we can write
$$\bar{H} ({p})
= \int_{\R^k} \big\{ \sin \langle {\bf v}, {p} \rangle \cdot
{F}({\bf v}) + \cos \langle {\bf v}, {p} \rangle \cdot {G} ( {\bf
v} ) \big\} d{\bf v}$$ for some rapidly (say, faster than
$(|p|+1)^{-N}$ for any $N \in \N$) decaying  functions ${F}$ and
${G}$ on $\R^k$. For every ${\bf v} \in \R^k$ define a function
$K_{\bf v}\in C^{\infty}(\R)$ by
$$K_{\bf v}(s):= \sin s \cdot F({\bf v}) + \cos s \cdot G({\bf v})\;.$$
Observe that
$$\Phi^* \bar{H} = \int_{\R^k} \Phi_{\bf v}^*K_{\bf v} \; d{\bf v}\;.$$
Denote by $B(R)$ the Euclidean ball of radius $R$ in $\R^k$ with
the center at the origin. Put
$$\bar{H}_{R}(p) = \int_{B(R)} K_{\bf v}(\langle {\bf v},p\rangle) \; d{\bf v},\; p \in \R^k\;.$$
Since the functions $F$ and $G$ are rapidly decaying, we get that
\begin{equation}\label{eq-Forier-vsp2}
||\bar{H}_R - \bar{H}||_{C^0(\R^k)} \to 0 \;\;\text{as}\;\; R \to
\infty\;.
\end{equation}
We claim that for every $R$
\begin{equation}\label{eq-Fourier-1}
\zeta(\Phi^*\bar{H}_R) \leq \bar{H}_R (p_{spec})\;.
\end{equation}
Indeed, for $\epsilon > 0$ introduce the integral sum
$$\bar{H}_{R,\varepsilon} (p)  = \sum_{{\bf v} \in \; \varepsilon \cdot
\Z^k\cap B(R)} \varepsilon^k \cdot K_{\bf v}(\langle {\bf
v},p\rangle    )\;.$$ Then
$$\Phi^*\bar{H}_{R,\varepsilon} = \sum_{{\bf v} \in \; \varepsilon \cdot
\Z^k\cap B(R)} \varepsilon^k \cdot \Phi_{\bf v}^*K_{\bf v}\;.$$
Applying
 repeatedly \eqref{eq-spec-fiber-T1-action-3} and
\eqref{eq-spec-fiber-T1-action-2} we get that
$$\zeta(\Phi^*\bar{H}_{R,\varepsilon}) \leq \bar{H}_{R,\varepsilon}(p_{spec})\;.$$
Note now that for fixed $R$ the family $\bar{H}_{R,\epsilon}$
converges to $\bar{H}_R$ as $\varepsilon \to 0$  in the uniform norm
on $C^0(\R^k)$. Using that $\zeta$ is Lipschitz with respect to the
uniform norm on $C^0(M)$ we readily get the inequality
\eqref{eq-Fourier-1}.

Combining the fact that $\zeta$ is Lipschitz with
\eqref{eq-Forier-vsp2} and \eqref{eq-Fourier-1} we get that
$$\zeta(\Phi^*\bar{H}) = \lim_{R \to \infty} \zeta(\Phi^*\bar{H}_R) \leq \lim_{R \to
\infty} \bar{H}_R (p_{spec})= \bar{H} (p_{spec})\;. $$ Now, assume
that $\bar{H} \geq 0$ and $\bar{H} ({p}_{spec})=0$. We just have
proved that $\zeta(\Phi^* \bar{H}) \leq 0$, and hence $\zeta(H)=0$,
which immediately yields the desired superheaviness of the special
fiber. This completes the reduction of the general case to the
1-dimensional case.

{\bf From now on we will consider only the case of an effective
Hamiltonian $\T^1$-action on $M$ with a moment map $\Phi: M\to
\R$. Its moment polytope $\Delta$ is a closed interval in $\R$ and
${p}_{spec} = - I (\Phi) \in \R$.}

\medskip
\noindent {\sc Reduction to the case of a strictly convex
function:} We claim that it is enough to show the following
proposition:

\begin{prop}
\label{prop-toric-main} Assume $\bar{H}: \R \to \R$ is a strictly
convex smooth function reaching its minimum at ${p}_{spec}$. Set
$H:= \Phi^* \bar{H}$. Then $\zeta (H) = \bar{H} ({p}_{spec})$.

\end{prop}

\medskip
\noindent Postponing the proof of the proposition for a moment let
us show that it implies the theorem. Indeed, let $F: M\to \R$ be a
Hamiltonian on $M$. In order to show the superheaviness of
$L=\Phi^{-1} ({p}_{spec})$ we need to show that $\zeta (F)\leq
\sup_L F$. Pick a very steep strictly convex function $\bar{H}:
\R\to\R$ with the minimum value $\sup_L F$ reached at ${p}_{spec}$
and such that $\Phi^* \bar{H} =: H\geq F$ everywhere on $M$. Then
using Proposition~\ref{prop-toric-main} and the monotonicity of
$\zeta$ we get
$$\zeta (F)\leq \zeta (H) = \bar{H} ({p}_{spec}) = \sup_L F,$$
yielding the claim.

\medskip
\noindent {\sc Preparations for the proof of
Proposition~\ref{prop-toric-main}:} Given a (time-dependent, not
necessarily regular) Hamiltonian $G$, we associate to every pair
$[\gamma, u]\in \tP_G$ a number
\[
D_G ([\gamma, u]) := \cA_G ([\gamma, u]) - \frac{\kappa}{2}\cdot
CZ_G ([\gamma, u]).
\]
(Recall that we defined the Conley-Zehnder index for {\it all}
Hamiltonians and not only the regular ones -- see
Section~\ref{subsubsec-CZ-Maslov}). The number $D_G ([\gamma, u])$ is
invariant under a change of the spanning disc $u$ -- an addition
of a sphere $jS\in H_2^S (M)$ to the disc $u$ changes both $\cA_G
([\gamma, u])$ and $\kappa/2\cdot CZ_G ([\gamma, u])$ by the same
number. Thus we can write $D_G ([\gamma, u]) = D_G (\gamma)$.

Given $[\gamma, u]\in \tP_G$ and $l\in \N$ define $\gamma^{(l)}$
and $u^{(l)}$ as the compositions of $\gamma$ and $u$ with the map
$z\to z^l$ on the unit disc $\D^2\subset \C$ (here $z$ is a
complex coordinate on $\C$). Denote by $t\mapsto g_t$ the time-$t$
flow of $G$ and by $ G^{(l)}: M\times \R\to \R$ the Hamiltonian
whose time-$t$ flow is $t\mapsto (g_t )^l$ and which is defined by
$$ G^{(l)} := G \sharp\ldots \sharp G\ \ (l\ {\rm times}),$$
where $G\sharp K (x, t) := G (x, t) + K (g_t^{-1} x, t)$ for any
$K: M\times \R\to \R$.

\begin{prop}
\label{prop-multiple-orbit}

There exists a constant $C>0$, depending only on
$n$, with the following property. Given a 1-periodic orbit
$\gamma\in \cP_G$ of the flow $t\mapsto g_t$ generated by $G$,
assume that $\gamma^{(l)}$ is a 1-periodic orbit of the flow
$t\mapsto g_t^l$ generated by $G^{(l)}$, and therefore for any $u$
such that $[\gamma, u]\in \tP_G$ we have $[\gamma^{(l)},
u^{(l)}]\in \tP_{G^{(l)}}$.
Then
\[
| D_{G^{(l)}} ([\gamma^{(l)}, u^{(l)}]) - l D_{G} ([\gamma,
u])|\leq l\cdot C.
\]

\end{prop}

\bigskip
\noindent \begin{proof} The action term in $D_G$ gets multiplied
by $l$ as we pass from $G$ to $G^{(l)}$. As for the Conley-Zehnder
term, the quasi-morphism property of the Conley-Zehnder index (see
Proposition~\ref{prop-CZ-qmm}) implies that there exists a
constant $C
>0$ (depending only on $n$) such that
$$| l CZ_G [\gamma, u] - CZ_{G^{(l)}}
([\gamma^{(l)}, u^{(l)}])|\leq C.$$
This immediately proves the proposition.\end{proof}

\begin{prop}
\label{prop-main-1}

Let $G: M\times [0,1]\to\R$ be a Hamiltonian as above. Then one
can choose $\epsilon >0$, depending on $G$, and a constant $C_n >
0$, depending only on $n= {\rm dim}\, M/2$, so that any function
$F: M\times [0,1]\to \R$ which is $\epsilon$-close to $G$ in a
$C^\infty$-metric on $ C^\infty(M\times [0,1])$ satisfies the
following condition: for every $\gamma_0\in \cP_F$ there exists
$\gamma\in \cP_{G}$ such that the difference between $D_F
(\gamma_0)$ and $D_{G} (\gamma)$ is bounded by $C_n$.

\end{prop}

\begin{proof}

Denote the flow of $G$ by $g_t$ (as before) and the flow of $F$ by
$f_t$. We will view time-1 periodic trajectories of these flows
both as maps of $[0,1]$ to $M$ having the same value at $0$ and
$1$ and as maps from $\SP^1$ to $M$.

First, consider the fibration $\D^2 \times M \to M$ and, slightly
abusing notation, denote the natural pullback of $\omega$ again by
$\omega$. Second, look at the fibration $pr: \D^2\times M\to
\D^2$. Denote by $Vert$ the vertical bundle over $\D^2\times M$
formed by the tangent spaces to the fibers of $pr$. For each loop
$\sigma: \SP^1\to M$ define by $\widehat{\sigma}: \SP^1\to
\D^2\times M$ the map $\widehat{\sigma} (t) := (t, \gamma (t))$.
The bundles $\sigma^* TM$ and $\widehat{\sigma}^* Vert$ over
$\SP^1$ coincide. Similarly for each $w: \D^2\to M$ denote by
$\widehat{w}: \D^2\to \D^2\times M$ the map $\widehat{w} (z):= (z,
w (z))$.

There exists $\delta >0$, depending on $G$, such that for each
$\gamma\in \cP_{G}$ a tubular $\delta$-neighborhood of the image
of $\widehat{\gamma}$ in $\SP^1\times M\subset \D^2\times M$,
denoted by $U_{\widehat{\gamma}}$, has the following properties:

\begin{itemize}

\item{}
there exists a 1-form $\lambda$ on $U_{\widehat{\gamma}}$ satisfying
$d\lambda = \omega$;

\item{}
$Vert$ admits a trivialization over $U_{\widehat{\gamma}}$.

\end{itemize}

Given an $\epsilon > 0$, we can choose $F$ sufficiently
$C^\infty$-close to $G$ so that the paths $t\mapsto f_t$ and
$t\mapsto g_t$ in $\Ham  (M)$ are arbitrarily $C^\infty$-close and
therefore

\begin{itemize}

\item{} for every $x\in \Fix (F)$ there exists $y\in \Fix (G)$
which is $\epsilon$-close to $x$ (think of the fixed points as
points of intersection of the graph of a diffeomorphism with the
diagonal);

\item{}
the $C^\infty$-distance between the maps $\gamma_0: t\mapsto f_t
(x)$ and $\gamma: t\mapsto g_t (y)$ from $[0,1]$ to $M$ is bounded
by $\epsilon$ and the image of $\widehat{\gamma}_0$ lies in
$U_{\widehat{\gamma}}$.

\end{itemize}

Pick a map $u_0: \D^2\to M$, $\left. u\right|_{\partial \D^2} =
\gamma_0$. Since $\gamma_0$ and $\gamma$ are $C^\infty$-close one
can enlarge $\D^2$ to a bigger disc $\D^2_1\supset \D^2$ and find
a smooth map $u: \D^2_1\to M$ so that
\begin{itemize}
\item{} $\left. u\right|_{\partial \D^2_1} = \gamma$; \item{}
$\left. u\right|_{\D^2} = u_0$; \item{} $u (\D^2_1\setminus \D^2)
\subset U_{\widehat{\gamma}}$.

\end{itemize}
Rescaling $\D^2_1$ we may assume without loss of generality that
$[\gamma, u]\in \cP_{G}$.

Trivialize the vector bundles $\gamma_0^* TM$ and ${\gamma}^* TM$
so that the trivializations extend to a trivialization of $u^* TM$
over $\D_1^2$ (and hence of $u_0^* TM$ over $\D^2$). Using the
trivializations we can identify the paths $t\mapsto d_{\gamma_0
(0)} f_t$ and $t\mapsto d_{\gamma (0)} g_t$ with some
identity-based paths of symplectic matrices $A(t)$, $B(t)$. Fixing
a small $\epsilon$ as above, we can also assume that $F$ is chosen
so $C^\infty$-close to $G$ that, in addition to all of the above,
the $C^\infty$-distance between the paths $t\mapsto A(t)$ and
$t\mapsto B(t)$ in $Sp\, (2n)$ is bounded by $\epsilon$ (for
instance, make sure first that the matrix paths obtained by
writing the paths $t\mapsto d_{\gamma_0 (0)} f_t$ and $t\mapsto
d_{\gamma (0)} g_t$ using some trivialization of $Vert$ over
$U_{\widehat{\gamma}}$ are close enough -- then the matrix paths
$t\mapsto A(t)$ and $t\mapsto B(t)$ will also be close enough).

We claim that by choosing $\epsilon$ sufficiently small in the
construction above we can bound the difference between $D_F
([\gamma_0, u_0])$ and $D_{G} ([\gamma, u])$ by a quantity
depending only on ${\rm dim}\, M$.

Indeed, the difference $|\int_0^1 F (\gamma_0 (t), t) dt -
\int_0^1 G (\gamma (t)) dt |$ is bounded by a quantity depending
only on some universal constants and $\epsilon$, because
$\gamma_0$ is $\epsilon$-close to $\gamma$ and $F$ is
$\epsilon$-close to $G$ with respect to the $C^\infty$-metrics. It
can be made arbitrarily small by choosing a sufficiently small
$\epsilon$. The difference
$$|\int_{\D^2} u_0^* \omega - \int_{\D^2} {u}^* \omega| =
|\int_{\D^2} \widehat{u}_0^* \omega - \int_{\D^2} \widehat{u}^*
\omega|$$ is bounded by the difference $| \int_0^1
\widehat{\gamma}_0^* \lambda - \int_0^1 \widehat{\gamma}^* \lambda
|$. Since, $\gamma_0$ and $\gamma$ are $\epsilon$-close in the
$C^\infty$-metric the later difference can be made less than $1$
if we choose a sufficiently small $\epsilon$. Thus we have shown
that by choosing a sufficiently small $\epsilon$ we can bound
$|\cA_F ([\gamma_0, u_0]) - \cA_{G} ([\gamma, u])|$ by $1$.

Now, as far as the Conley-Zehnder indices are concerned, our
choice of the trivializations means that the difference between
$CZ_F ([\gamma_0, u_0])$ and $CZ_{G} ([\gamma, u])$ is just the
difference between the Conley-Zehnder indices for the matrix paths
$t\mapsto A(t)$ and $t\mapsto B(t)$. But the latter paths in $Sp\,
(2n)$ are $\epsilon$-close in the $C^\infty$-sense, hence
represent close elements of $\widetilde{Sp\, (2n)}$ and if
$\epsilon$ was chosen sufficiently small, then, as we mentioned in
Section~\ref{subsubsec-CZ-Maslov}, their Conley-Zehnder indices
differ at most by a constant depending only on $n$.

This finishes the proof of the claim and the proposition.
\end{proof}

\medskip
\noindent{\sc Plan of the proof of
Proposition~\ref{prop-toric-main}:} We assume now that $\bar{H}$
is a fixed strictly convex function on $\R$. Our calculations will
feature $E$ as a large parameter. For quantities $\alpha, \beta$
depending on $E$ we will write $\alpha\preceq \beta$ if
$\alpha\leq \beta +{\textit const}$ holds for large enough $E$,
where ${\textit const}$ depends only on $(M,\omega)$, $\Phi$ and
$\bar{H}$, and in particular  does not depend on $E$. We will
write $\alpha\approx \beta$ if $\alpha\preceq \beta$ and
$\beta\preceq \alpha$. Using this language the proposition can be
restated as
\begin{equation}
\label{eqn-spectral-number-growth} c(a,EH) \approx E \bar{H}
({p}_{spec}).
\end{equation}

In general, 1-periodic orbits of the flow of $EH$ are not isolated
and therefore the Hamiltonian is not regular. Let $F$ be a regular
(time-periodic) perturbation of $EH$.

By the spectrality axiom,  the spectral number $c(a, F)$ for $a\in
QH_{2n} (M)$ equals $\cA_F ([\gamma_0, u_0])$ for some pair
$[\gamma_0, u_0]\in \tP_F$ with $CZ_F ([\gamma_0, u_0]) = 2n$.
Thus $c(a, F) \approx D_F (\gamma_0)$. Combining this with
Proposition~\ref{prop-main-1} we get that for {\it some}
$\gamma\in\cP_{EH}$
\begin{equation}
\label{eqn-spectral-number-EH-approx-D-EH-gamma} E \bar{H}
({p}_{spec}) \preceq c (a, EH) \approx c(a,F) \approx    D_F
(\gamma_0) \approx D_{EH} (\gamma)\;.
\end{equation}
Thus it would be enough to show
that
\begin{equation}
\label{eqn-EH-preceq-D-EH} D_{EH} (\gamma) \preceq E \bar{H}
({p}_{spec}) \ \ {\rm for}\ {\it all}\ \gamma\in\cP_{EH}\;,
\end{equation}
which together with (\ref{eqn-spectral-number-EH-approx-D-EH-gamma}) would imply (\ref{eqn-spectral-number-growth}).

Inequality (\ref{eqn-EH-preceq-D-EH}) will be proved in the
following way.  Note that each $\gamma\in\cP_{EH}$ lies in
$\Phi^{-1} ({p})$ for some ${p}\in\Delta$. We will show that
\begin{equation}
\label{eqn-concav} D_{EH} (\gamma) \approx E\bar{H} ({p}) +
E\bar{H}' (p) ({p}_{spec} - {p}).
\end{equation}

Note that (\ref{eqn-concav}) implies (\ref{eqn-EH-preceq-D-EH}).
Indeed, since $\bar{H}$ is strictly convex and reaches its
minimum at ${p}_{spec}$, it follows from (\ref{eqn-concav})
that
\[
D_{EH} (\gamma) \approx E\bar{H} ({p}) + E
\bar{H}' (p) ({p}_{spec} - {p}) \leq
E\bar{H} ({p}_{spec}),
\]
which is true for any $\gamma\in \cP_{EH}$ thus yielding
(\ref{eqn-EH-preceq-D-EH}).

\medskip
\noindent {\sc Proof of (\ref{eqn-concav}):} Let the $\T^1$-action
on $M$ be given by a loop of symplectomorphisms $\{ \phi_t \}$,
$t\in \R$, $\phi_t = \phi_{t+1}$. The flow of $EH$ has the form
$h_tx = \phi_{E\bar{H}' (\Phi(x))t}x$.

We view $\gamma$ as a map $\gamma : [0,1]\to M$ satisfying $\gamma
(0) = \gamma (1)$. Denote $x := \gamma (0)$. The curve $\gamma$
lies in $\Phi^{-1} ({p})$.

Denote $N := \gamma ([0,1])$. This is the $\T^1$-orbit of $x$ and it
is either a point or a circle.

In the first case $\gamma$ is a constant trajectory concentrated
at a fixed point $N\in M$ of the action.
Using this constant curve $\gamma$ together with the
constant disc $u$ spanning for the definitions of $I(\Phi)$ and $D_{EH} (\gamma)$ one gets
$${p}_{spec} - p = m_\Phi (\gamma, u) \cdot \kappa/2,$$
and
$$D_{EH} (\gamma) = E\bar{H} ({p}) - \kappa/2 \cdot CZ_{EH} ([\gamma, u]).$$
Thus proving (\ref{eqn-concav}) reduces in this case to proving
$$- CZ_{EH} ([\gamma, u])\approx E\bar{H}' (p)\cdot m_\Phi (\gamma, u).$$
Let us fix a symplectic basis of $T_N M$ and view each differential $d_N \phi_t$
as a symplectic matrix $A(t)$, so that $\{ A(t)\}$ is an identity-based loop in $Sp\, (2n)$. Then
$$- CZ_{EH} ([\gamma, u])\approx CZ_{matr} (\{ A (E\bar{H}'(p)t) \}),$$
while
$$E\bar{H}' (p)\cdot m_\Phi (\gamma, u) \approx E\bar{H}' (p) Maslov (\{ A(t)\}).$$
Thus we need to prove
$$CZ_{matr} (\{ A (E\bar{H}'(p)t) \}) \approx E\bar{H}' (p) Maslov (\{ A(t)\}),$$
which follows easily from the definitions of the Conley-Zehnder index and the Maslov class.

Thus from now on we will assume that $N$ is a circle. Take any
point $x \in N$. The stabilizer of $x$ under the $\T^1$-action is
a finite cyclic group of order $k \in \N$. Thus the orbit of the
$\T^1$-action turns $k$ times along $N$. Since $\gamma$ is a
non-constant closed orbit of the Hamiltonian flow generated by
$E\Phi^*\bar{H}$, it turns $r$ times along $N$ with $ r \in \Z
\setminus \{0\}$. This implies that $E\bar{H}' (p)= r/k$.  We
claim that without loss of generality we may assume that $l:= r/k
$ is an integer.

Indeed, we can always pass to $\gamma^{(k)}\in \cP_{kEH}$, so that
$(kE\bar{H})' (p) \in \Z$, and if we can prove the proposition for
$\gamma^{(k)}$, then
\[
D_{k EH} (\gamma^{(k)}) \approx k E\bar{H} ({p}) + k
E\bar{H}' (p) (p_{spec} - {p}).
\]
Applying Proposition~\ref{prop-multiple-orbit} we get
\[
k D_{EH} (\gamma) \approx k E\bar{H} ({p}) + k E
\bar{H}' (p) (p_{spec} -p) + k\cdot {\it const},
\]
and hence
\[
D_{EH} (\gamma) \approx E\bar{H} ({p}) + E
\bar{H}' (p) (p_{spec} - p),
\]
proving the claim for the original $\gamma$.

From now on we assume that  $l:=E\bar{H}' (p) \in\Z \setminus \{0\}$
and that $[\gamma, u]\in \tP_{l\Phi}$. Consider the Hamiltonian
vector field  $X:=\text{sgrad}\, \Phi$ at a point $x \in N$. Since
$N$ is a non-constant orbit we get $X\neq 0$. Then $V = T_x
(\Phi^{-1} (p))$ is the skew-orthogonal complement to $X$. Choose a
$\T^1$-invariant $\omega$-compatible almost complex structure $J$ in
a neighborhood of $N$. Together $\omega$ and $J$ define a
$\T^1$-invariant Riemannian metric $g$. Decompose the tangent bundle
$TM$ along $N$ as follows. Put $Z = \text{Span}(JX,X)$ and set $W$
to be the $g$-orthogonal complement to $X$ in $V$. Thus we have a
$\T^1$-invariant decomposition
\begin{equation}\label{eq-6}
T_xM = W\oplus Z\;, x \in N\;.
\end{equation}
Furthermore, $W$ and $Z$ carry canonical symplectic forms. Thus
$W$ and $Z$ define symplectic (and hence trivial)  subbundles of
$TM$ over $N$. They induce trivial subbundles of the bundle
$\gamma^* TM$ over $\SP^1$.

We calculate
\begin{equation}\label{eq-4}
dh_t (x) \xi = d\phi_{EH'(\Phi(x))t} (x) \xi + EH''(\Phi(x))\cdot
d\Phi(\xi)\cdot X\;.
\end{equation}
We consider two trivializations of the bundle $\gamma^* TM$ over
$\SP^1$. The first trivialization is defined by means of sections
invariant under the $\T^1$-action. The second one is chosen in
such a way that it extends to a trivialization of $u^* TM$ over
$\D^2$. Using these trivializations we can identify $dh_t (x)$,
respectively, with two identity-based paths $\{ C_t\}$, $\{
C^\prime_t\}$ of symplectic matrices. The decomposition
\eqref{eq-6} induces a split
$$ C_t = \id \oplus B_t\;.$$
We claim that $|CZ_{matr} (\{ B_t\})|$ is bounded by a constant
independent of $E$. Indeed, observe that in the basis $(X,JX)$ of
$Z$
\[ B_t= \left( \begin{array}{cc}
1 & b_{12}(t)  \\
0 & 1  \end{array} \right)\;.\] Denote by $L$ the line spanned by
$X=(1,0)$. Perturb $\{B_t\}$ to a path $\{B'_t = R_{\delta
t}B_t\}$, where $R_t$ is the rotation by angle $t$, and $\delta>0$
is small enough. Observe that $B'(t)L \cap L =\{0\}$ for $t>0$. It
follows readily from the definitions that $|CZ_{matr}(B'_t)|$ and
$|CZ_{matr}(R_{\delta t})|$ do not exceed $2$. Thus by the
quasi-morphism property of the Conley-Zehnder index (see
Proposition~\ref{prop-CZ-qmm}) we have that $|CZ_{matr} (\{
B_t\})|$ is bounded by a constant independent of $E$, which yields
the claim. Therefore
$$ CZ_{matr}\, (\{ C_t\})\approx 0\;.$$
On the other hand, by formula \eqref{eqn-CZ-trivializations}
$$
CZ_{matr}\, (\{ C'_t\})=CZ_{matr}\, (\{ C_t\})+ m_{l\Phi}
([\gamma, u])\;.$$ Thus
\begin{equation}
\label{eqn-CZ-final} CZ_{EH} ([\gamma, u]):= n - CZ_{matr}\, (\{
C'_t\}) \approx - m_{l\Phi} ([\gamma, u]).
\end{equation}
Since the periodic trajectory $\gamma$
lies inside $\Phi^{-1} ({p})$, we get
\begin{equation}
\label{eqn-action-final} \cA_{EH} ([\gamma, u]) = \int_0^1 EH
(\gamma (t)) dt - \int_{\D^2} u^\ast \omega = E\bar{H} ({p}) -
\int_{\D^2} u^\ast \omega.
\end{equation}
Using (\ref{eqn-action-final}) and (\ref{eqn-CZ-final}) the
precise equality
\[
D_{EH} ([\gamma, u]) = \cA_{EH} ([\gamma, u]) -
\frac{\kappa}{2}\cdot CZ_{EH} ([\gamma, u])
\]
can be turned into an asymptotic inequality
\begin{equation}
\label{eqn-D-asympt1} D_{EH} ([\gamma, u]) \approx E\bar{H} ({p})
- \int_{\D^2} u^\ast \omega + \frac{\kappa}{2} m_{l\Phi} ([\gamma,
u]).
\end{equation}
Since the periodic trajectory $\gamma$ lies inside $\Phi^{-1}
({p})$, we have
\begin{equation}
\label{eqn-Theta-action} \cA_{l\Phi} ([\gamma, u]) = \int_0^1
l\Phi (\gamma (t)) dt - \int_{\D^2} u^* \omega = l p  -
\int_{\D^2} u^* \omega.
\end{equation}
Adding and subtracting $lp$ from the
right-hand side of (\ref{eqn-D-asympt1}) and using
(\ref{eqn-Theta-action}) we get
\[
D_{EH} (\gamma) = D_{EH} ([\gamma, u]) \approx \bigg( E\bar{H} (p)
- lp)  \bigg) + \bigg( lp  - \int_{\D^2} u^\ast \omega +
\frac{\kappa}{2} m_{l\Phi} ([\gamma, u]) \bigg) =
\]
\[
= \bigg( E\bar{H} ({p}) - lp  \bigg) +
\bigg( \cA_{l\Phi} ([\gamma, u]) + \frac{\kappa}{2} m_{l\Phi}
([\gamma, u]) \bigg) =
\bigg( E\bar{H} ({p}) - lp  \bigg) - I (l\Phi)
=
\]
\[
= E\bar{H} ({p}) + l (- I(\Phi) - {p}) = E\bar{H} ({p}) + l
(p_{spec} - p)\;.\] Recalling that $l = EH' (p)$, we finally
obtain that
\[D_{EH} (\gamma)
= E\bar{H} ({p}) + EH' (p) (p_{spec} - p),
\]
 which is precisely the equation
(\ref{eqn-concav}) that we wanted to get. This finishes the proof
of Proposition~\ref{prop-toric-main} and
Theorem~\ref{thm-main-nondispl-toric}. \Qed

\subsection{Calabi and mixed action-Maslov}
\label{sec-pf-thm-Calabi-qmm-Ham-loops}

\noindent {\bf  Proof of Theorem~\ref{thm-Calabi-qmm-Ham-loops}.}

Assume $H: M\times [0,1]\to \R$ is a normalized Hamiltonian which
generates a loop in $\Ham  (M)$ representing a class
$\alpha\in\pi_1 (\Ham  (M))\subset \tHam (M)$. Then $H^{(l)}$ is
also normalized and generates a loop representing $\alpha^l$. Let
us compute $\mu (\alpha) = - {\hbox{\rm vol}\, (M)}\cdot
\lim_{l\to +\infty} c (a, H^{(l)} )/l$.

Arguing as in the proof of
(\ref{eqn-spectral-number-EH-approx-D-EH-gamma}) we get that there
exists a constant $C>0$ such that for each $l\in\N$ there exists
$\gamma\in \cP_{H^{(l)}}$ for which $| c (a, H^{(l)}) -
D_{H^{(l)}} (\gamma) |\leq C$. But, as it follows from the
definitions and from the fact that $I$ is a homomorphism,
$D_{H^{(l)}} (\gamma)$ does not depend on $\gamma$ and equals $-I
(\alpha^l) = -l I(\alpha)$. This immediately implies that $\mu
(\alpha) = {\hbox{\rm vol}\, (M)}\cdot I (\alpha) $. \Qed

\bigskip
\noindent {\bf Acknowledgements.} The origins of this paper lie in
our joint work with P.Biran on the paper \cite{BEP} -- we thank
him for fruitful collaboration at an early stage of this project,
as well as for his crucial help with Example~\ref{exam-Fermat} on
Lagrangian spheres in projective hypersurfaces. We also thank him
and O.Cornea for pointing out to us a mistake in the original
version of this paper and helping us with the correction (see
Section~\ref{sec-Lagr-proofs}).  We thank F.~Zapolsky for his help
with the  ``exotic" monotone Lagrangian torus in $\SP^2 \times
\SP^2$ discussed in Example~\ref{exam-exotic-torus}. We thank
C.~Woodward for pointing out to us the link between the special
point in the moment polytope of a symplectic toric manifold and
the Futaki invariant, and E.~Shelukhin for useful discussions on
this issue. We are also grateful to V.L.~Ginzburg, Y.~Karshon,
Y.~Long, D.~McDuff, M.~Pinsonnault, D.~Salamon and M.~Sodin for
useful discussions and communications. We thank K.~Fukaya, H.~Ohta
and K.~Ono, the organizers of the Conference on Symplectic
Topology in Kyoto (February 2006), M.~Harada, Y.~Karshon,
M.~Masuda and T.~Panov, the organizers of the Conference on Toric
Topology in Osaka (May 2006), O.~Cornea, V.L.~Ginzburg, E.~Kerman
and F.~Lalonde, the organizers of the Workshop on Floer theory
(Banff, 2007), and A.~Fathi, Y.-G.~Oh and C.~Viterbo, the
organizers of the AMS Summer Conference on Symplectic Topology and
Measure-Preserving Dynamical Systems (Snowbird, July 2007), for
giving us an opportunity to present a preliminary version of this
work and for the superb job they did in organizing these
conferences. Finally, we thank an anonymous referee for helpful
comments and corrections.

\bigskip

\bibliographystyle{alpha}

\bigskip

\noindent

\begin{tabular}{@{} l @{\ \ \ \ \ \ \ \ \ \ \,} l }
Michael Entov & Leonid Polterovich \\
Department of Mathematics & School of Mathematical Sciences \\
Technion & Tel Aviv University \\
Haifa 32000, Israel & Tel Aviv 69978, Israel \\
entov@math.technion.ac.il &
polterov@post.tau.ac.il\\
\end{tabular}

\end{document}